\pdfoutput=1 
\documentclass{amsart}

\usepackage[margin=1in]{geometry}
\usepackage{amsmath}
\usepackage{amssymb}
\usepackage{amsthm}
\usepackage{mathtools}
\usepackage{float}
\usepackage{hyperref}
\usepackage{longtable}
\usepackage{enumitem}
\usepackage{xcolor,colortbl}
\usepackage{tikz}
\usetikzlibrary{arrows.meta}
\tikzstyle{flag}=[]
\tikzstyle{vertex}=[circle, draw]
\tikzstyle{v}=[rectangle, draw, fill, minimum size=0.1cm, inner sep = 0cm]
\tikzstyle{b}=[circle, draw, minimum size=0.2cm, inner sep = 0cm]
\tikzstyle{c}=[circle, draw, minimum size=0.2cm, inner sep = 0cm, fill = gray]
\usepackage{tikz-cd}
\usepackage{forest}
\forestset{
fairly nice empty nodes/.style={
delay={where content={}
{shape=coordinate, for siblings={anchor=north}}{}},
for tree={s sep=4mm}
}
}
\usepackage[colorinlistoftodos]{todonotes}
\usepackage[square,numbers]{natbib}
\usepackage{hyperref}
\usepackage[capitalise]{cleveref}
\usepackage{array} 
\newcolumntype{C}{>{$}l<{$}} 
\usepackage{caption}

\newtheorem{samecounter}{samecounter}
\numberwithin{samecounter}{subsection}

\newtheorem{corollary}[samecounter]{Corollary}
\newtheorem{proposition}[samecounter]{Proposition}
\newtheorem{theorem}[samecounter]{Theorem}

\newtheorem{lemma}[samecounter]{Lemma} 
\newtheorem{construction}[samecounter]{Construction}

\newtheorem*{corollary*}{Corollary}
\newtheorem*{proposition*}{Proposition}
\newtheorem*{theorem*}{Theorem}
\newtheorem*{lemma*}{Lemma}

\theoremstyle{definition}
\newtheorem{example}[samecounter]{Example}
\newtheorem{definition}[samecounter]{Definition}
\newtheorem{remark}[samecounter]{Remark} 

\newcommand{\wheel}{\circlearrowright}

\newcommand{\fC}{\mathfrak{C}}
\newcommand{\uc}{\underline{c}}
\newcommand{\ud}{\underline{d}}
\newcommand{\dc}{\binom{\ud}{\uc}}
\newcommand{\cbprime}{\binom{\uc'}{\ub'}}

\newcommand{\dch}{(\uc;\ud)}
\newcommand{\ua}{\underline{a}}
\newcommand{\ub}{\underline{b}}
\newcommand{\ba}{\binom{\ub}{\ua}}

\newcommand{\ue}{\underline{e}}
\newcommand{\uf}{\underline{f}}
\newcommand{\fe}{\binom{\uf}{\ue}}

\newcommand{\andspace}{\quad\text{and}\quad}
\newcommand{\pofc}{\mathcal{P}(\fC)}
\newcommand{\pofcop}{\pofc^{op}}
\newcommand{\ptwoc}{\pofcop\times\pofc}
\newcommand{\SC}{\mathcal{S}(\fC)}
\newcommand{\SV}{\mathcal{S}(\mathbb{V})}

\newcommand{\fF}{\mathfrak{F}}

\newcommand{\vertex}{Vt}
\newcommand{\Flag}{Flag}

\newcommand{\inp}{in}
\newcommand{\out}{out}
\newcommand{\contra}[2]{\varepsilon_{#1}^{#2}}
\newcommand{\hcomp}{\otimes_h}

\newcommand{\vcomp}{\circ_V}
\newcommand{\properadic}[1]{\boxtimes_{#1}}

\newcommand{\eproperadic}[2]{\circ_{#1,#2}}
\newcommand{\eproperadicop}[2]{\hat{\circ}_{#1,#2}}
\newcommand{\eproperadicunspecified}{\circ_{v}}
\newcommand{\eproperadicopunspecified}{\hat{\circ}_{v}}
\newcommand{\symmetricTreesClasses}{\Sigma T}
\newcommand{\regularTreesClasses}{RT}
\newcommand{\shuffleTreesClasses}{ST}
\newcommand{\symmetricTreesClassesWithEmpty}{\overline{{\Sigma}T}}
\newcommand{\strict}{\overline{RT}}
\newcommand{\shuffleTreesClassesWithEmpty}{\overline{ST}}
\newcommand{\antishriek}{\text{\raisebox{\depth}{\textexclamdown}}}

\usepackage{tkz-euclide}
\usepackage{xstring}

\def\splicelist#1{
\StrCount{#1}{,}[\numofelem]
\ifnum\numofelem>0\relax
     \StrBehind[\numofelem]{#1}{,}[\mylast]%
\else
    \let\mylast#1%
\fi
}

\newcommand{\myroundpoly}[3][thin,color=black]{
\splicelist{#2}
\foreach \vertex [remember=\vertex as \succvertex
    (initially \mylast)] in {#2}{
    \coordinate (\succvertex-next) at ($(\succvertex)!#3!90:(\vertex)$);
    \coordinate (\vertex-previous) at ($(\vertex)!#3!-90:(\succvertex)$);
    \draw[#1] (\succvertex-next) --  (\vertex-previous);
}
\foreach \vertex in {#2}{
    \tkzDrawArc[#1](\vertex,\vertex-next)(\vertex-previous)
}
}

\begin{document}

\title{Koszul Operads Governing Props and Wheeled Props}

\author[Kurt Stoeckl]{Kurt Stoeckl}
\address[Kurt Stoeckl]{School of Mathematics and Statistics, The University of Melbourne, Victoria, Australia}
\email[Kurt Stoeckl]{kstoeckl@student.unimelb.edu.au}
\urladdr{\url{https://kstoeckl.github.io/}}

\maketitle

\begin{abstract}
In this paper, we construct groupoid coloured operads governing props and wheeled props, and show they are Koszul. This is accomplished by new biased definitions for (wheeled) props, and an extension of the theory of Groebner bases for operads to apply to groupoid coloured operads. Using the Koszul machine, we define homotopy (wheeled) props, and show they are not formed by polytope based models. Finally, using homotopy transfer theory, we construct Massey products for (wheeled) props, show these products characterise the formality of these structures, and re-obtain a theorem of Mac Lane on the existence of higher homotopies of (co)commutative Hopf algebras.
\end{abstract}

\DeclareRobustCommand{\SkipTocEntry}[5]{}
\setcounter{tocdepth}{2}
\def\l@section{\@tocline{1}{4pt}{4pt}{}{}}
\let\oldtocsection=\tocsection
\renewcommand{\tocsection}[2]{\bf\oldtocsection{#1}{#2}}
\let\oldtocsubsection=\tocsubsection
\renewcommand{\tocsubsection}[2]{\quad\oldtocsubsection{#1}{#2}}
\tableofcontents

\pagebreak

\section{Introduction} \label{introdution}

Props, introduced by Adams and Mac Lane (\cite{maclane1965categorical}), are a type of symmetric monoidal category. 
Their representations model various (mixed) algebraic and coalgebraic structures including associative, commutative, Lie and Hopf algebras. They arise naturally in the study of deformation theory \cite{andersson2022deformation}, \cite{vallette2009deformationI}, differential geometry \cite{markl2009wheeled}, knot theory \cite{dancso2021circuit}, \cite{dancso2023topological}, and topology \cite{boardman2006homotopy}, \cite{wahl2016hochschild}. Homotopy props arise when one wishes to study variants of the algebraic structures above, which are only weakly associative. 
For instance, Mac Lane introduced PACTs, a version of homotopy associative props, and used them to identify Massey products of (co)commutative Hopf algebras (\cref{Mac Lane theorem}).
Another example are the cobordism categories of Segal whose morphism spaces are moduli spaces of Riemann surfaces. In these cobordism categories, categorical composition is associative on the associated chain complex, but is only homotopy associative at the space level (e.g. Remark 3.31 of \cite{santander2015comparing}).
\\\\
Wheeled props are props with an additional non-degenerate bilinear form or ``trace'' operation. They were introduced by Markl, Merkulov and Shadrin \cite{markl2009wheeled}, where they used (homotopy) wheeled props to study Batalin–Vilkovisky quantisation formalism in theoretical physics. Wheeled props are also known to arise in invariant theory \cite{derksen2023invariant}, and they play a key role in the study of universal finite type invariants of virtual and welded tangles \cite{dancso2021circuit}. This is far from a spanning list, and we direct the interested reader to the more complete literature surveys of \cite{hackney2015infinity}, \cite{merkulov2010wheeled} and \cite{yau2015foundation}.  
\\\\
These examples of homotopy (wheeled) props motivate a more systematic study of homotopy associative, or $\infty$-props. While there are many known models for $\infty$-operads and $\infty$-properads in the literature (e.g. 
\cite{barkan2022equifibered}, 
\cite{cisinski2013dendroidal},
\cite{granaaker2007strong},
\cite{hackney2015infinity},
\cite{lurie2016higher}, 
\cite{MW_Kan},
\cite{vallette2009deformationI} etc.) there are currently only a few suggested models for $\infty$-props (including \cite{caviglia2015dwyer}, \cite{hackney2017sprop} and \cite{haugseng2024operads}), and $\infty$-wheeled props (including \cite{raynor2021brauer}).
This paper uses the machinery of Koszul duality to provide a algebraic model for homotopy, or $\infty$, (wheeled) props.  
\\\\
Over the last 20 years, the Koszul machine has been used to construct homotopy associative versions of many operadic structures. In brief, an operad $P$ is said to be Koszul if, and only if, it admits a minimal model $P_\infty$ with a quadratic differential (\cref{def:koszul}).
The minimal model $P_\infty$ is characterised by the property that algebras over $P_\infty$ are homotopy (associative) $P$-algebras. 
A coloured Koszul operad governing non-symmetric operads was first constructed in \cite{van2003coloured}.
Later in \cite{ward2022massey}, Ward constructed a groupoid coloured operad governing modular operads, and showed it was Koszul.
Most recently it has been shown by two different groups of authors (\cite{batanin2021koszul}, \cite{kaufmann2021koszul}) that the operads governing operadic families living on connected graphs (including operads, dioperads, wheeled properads, e.t.c) are all Koszul. 
A common thread unifying the last three papers, is that they all use the theory of convex polytopes, explicitly or implicitly, to interpret the minimal models of the operads.
\\\\
Our present work, builds on this story with three main results. Firstly, in Sections \ref{The Operad Governing Wheeled Props} and \ref{The Operad Governing Props}, we construct groupoid coloured operads $\mathbb{W}$ and $\mathbb{P}$, governing wheeled props and props respectively.
We then prove, in Sections \ref{proof The Operad Governing Wheeled Props is Koszul} and \ref{proof The Operad Governing Props is Koszul}, that 
\begin{theorem}\label{The operads governing (wheeled) props are Koszul} 
The groupoid coloured operads $\mathbb{W}$ and $\mathbb{P}$ are Koszul.
\end{theorem}
By virtue of the Koszul machine, the associated minimal models $\mathbb{P}_\infty$ (and $\mathbb{W}_\infty$) govern homotopy associative dg-(wheeled) props.
Then in \cref{homotopy (wheeled) props}, we provide simple counter examples showing the following.
\begin{theorem*}[\ref{nesting complexes of props not polytopes}]
There exist subcomplexes of $\mathbb{W}_\infty$ and $\mathbb{P}_\infty$ which are \textbf{not} isomorphic, as lattices, to the face poset of convex polytopes.
\end{theorem*}
Consequently, there are no polytope based minimal models for the operads $\mathbb{W}$ and $\mathbb{P}$.
This shows that the previous techniques for proving the operads governing operadic structure are Koszul do not readily extend to (wheeled) props (\cref{Nesting complex not convex obstructions}).
Indeed, in order to prove that $\mathbb{P}$ and $\mathbb{W}$ are Koszul we develop a more generally applicable tool. 
In \cref{Rewriting Systems and Groebner Bases},
we transfer the techniques of Groebner bases for operads (\cite{dotsenko2010grobner}, \cite{kharitonov2021grobner}) to groupoid coloured operads with the following result.

\begin{theorem*}[\ref{Koszul if corresponding shuffle operad has QGB}] 
Let $P$ be a groupoid coloured operad such that the associated coloured shuffle operad $(P^f)_*$ admits a quadratic Groebner basis, then $P$ is Koszul.
\end{theorem*}

The proof of this theorem, in \cref{The Groupoid Coloured Extension}, relies on the studying the extent that the forgetful functor $-^{f_{\mathbb{V}}}$, from groupoid coloured modules to (discrete-groupoid) coloured modules, fails to be monoidal with respect to the operadic monoidal products on both categories.
In particular, given a groupoid coloured shuffle operad $P$ we construct an epimorphism, of groupoid coloured shuffle cooperads, between the bar complexes $B(P^{f_{\mathbb{V}}}) \twoheadrightarrow B(P)$ (\cref{epimorphism of bar constructions}).
This lets us infer the following.
\begin{proposition*}[\ref{groupoid coloured shuffle operad is Koszul if its forgetful image is Koszul}]
{\small
A groupoid coloured shuffle operad $P$ is Koszul if the coloured shuffle operad $P^{f_{\mathbb{V}}}$ is Koszul.
}
\end{proposition*}

Thus by obtaining an explicit simple presentation for $P^{f_{\mathbb{V}}}$ which we denote $P_*$ (\cref{An alternate presentation after forgetting the groupoid}), we may apply known rewriting techniques for coloured operads (\cite{loday2012algebraic}, \cite{malbos2023shuffle}) to prove $P$ is Koszul (\cref{Rewriting Systems and Groebner Bases}). 
We note that \cref{Koszul if corresponding shuffle operad has QGB} can also be applied to prove that the operads governing connected operadic structures are Koszul (\cref{Koszul operads for other operadic}).
Two less dramatic results that may also be of independent interest, are new biased definitions of (wheeled) props (\ref{Alternate biased wheeled prop}, \ref{Alternate biased prop}).
Critically, for defining $\mathbb{W}$ and $\mathbb{P}$, the equivariance axioms of these structures enable a simple canonical form for every composite of operations.
Thus, when we translate these equivariance axioms into actions of the groupoid (in Sections \ref{biased presentation of the operad governing wheeled props} and \ref{biased presentation of the operad governing props}), the operads inherit these simple canonical forms.
\\\\
The general theory and constructions of this paper can be used for many purposes.
For instance, as we now know that $\mathbb{W}$ and $\mathbb{P}$ are Koszul, we can apply the technique of homotopy transfer theory (HTT) to these structures (\cref{homotopy transfer theory}).
As homology is an example of a homotopy retract (over a field of characteristic $0$), we are able to construct generalised Massey products.
This immediately provides an alternate characterisation of the formality of a wheeled prop (\cref{Massey products characterise formality}), and a means of reinterpreting a theorem of Mac Lane as a particular instance of HTT (\cref{Mac Lane theorem}).
In addition, HTT provides an alternate pathway to studying deformations of (wheeled) props (\cref{deformation theory link}).
\\\\
Currently, the author is working with Philip Hackney and Marcy Robertson to develop graphical and Segal models for infinity props.
It is expected that the nerve functor of \cite{le2017homotopy}, from strictly unital algebraic homotopy operads to dendroidal sets, will extend to homotopy props.
That is to say, a unital version of an algebraic homotopy prop (i.e. an algebra over $\mathbb{P}_\infty$), will be an example of graphical $\infty$-prop (i.e. a graphical set satisfying an inner horn condition).
Additionally, it is expected that graphical $\infty$-props will be cofibrant objects in the model category of simplicially enriched props \cite{caviglia2015dwyer}, \cite{hackney2017sprop}.
\\\\
Another area for future work is pursuing combinatorially simpler models for homotopy (wheeled) prop(erad)s. Although homotopy props do not admit a nesting model governed by polytopes, properads do.
In particular, the appropriate minimal model for the operad governing properads is given by the poset associahedra of \cite{galashin2021poset} (\cref{acyclic tubings provide minimal model of operad governing properads}).
This poset has been recently realised as a convex polytope in \cite{chiara2023cyclonestohedra} and \cite{sack2023realization}.
As such, one could define tensor products of homotopy properads by extending the program of \cite{laplante2022diagonal}.
Finally, it is plausible that 'simpler' combinatorial models for homotopy (wheeled) props exist.
Given the inclusion of the operad $Com$ into $\mathbb{W}$ and $\mathbb{P}$ (\cref{inclusion of com}), an extension of existing graphical models for the coLie cooperad (\cite{sinha2005pairing}, \cite{tourtchine2007other}) might bear fruit.

\subsection{Acknowledgements}

I would like to thank my supervisor, Marcy Robertson, for countless discussions and her careful review of this paper.
I'm very grateful for an early conversation with Ben Ward, in which he suggested that the techniques thus far developed for wheeled props, might also be applicable to props.
I would like to thank Guillaume Laplante-Anfossi for many useful conversations, and comments on this paper. 
I'm thankful to Andrew Sack for pointing out that a subcomplex of the minimal model for the operad governing props violated the diamond condition.
I would like to thank both Dominik Trnka and Bruno Vallette for catching errors in an earlier version.
Additionally, I would like to thank my referees for their detailed comments, and in particular, for spotting an error in my initial proof of \cref{Koszul if corresponding shuffle operad has QGB}.
Finally, I acknowledge the support of, an Australian Government Research Training Program (RTP) Scholarship, and the Australian Research Council Future Fellowship FT210100256.

\section{Preliminaries} \label{Preliminaries}

Throughout this paper, let $\mathcal{E}$ be a co-complete, closed symmetric monoidal category and let $\mathbb{V}$ be a small groupoid. 
Informally, we say that a coloured wheeled prop is a $\mathfrak{C}$-coloured bimodule in $\mathcal{E}$, together with composition operations that are described by wheeled graphs (11.33 of \cite{yau2015foundation} and \cref{Alternate biased wheeled prop}). As such, in order to define (wheeled) props, we first need to introduce some preliminaries on bimodules and wheeled graphs. Furthermore, as we will construct a groupoid coloured operad governing wheeled props, we will need to define groupoid coloured bimodules, and their morphisms.

\subsection{Bimodules}\label{Bimodules}

We first recall the definition of $\mathfrak{C}$-coloured profiles and introduce notation, which we will use throughout this paper. 

\begin{definition}[Section 1.1 of \cite{yau2015foundation}]
\label{def:colours}
Let $\fC$ be a non-empty set (of colours) and let $\Sigma_n$ denote the symmetric group on $n$ letters.
\begin{enumerate}
\item
An element in $\fC$ will be called a \textbf{colour}.
\item
A \textbf{$\fC$-profile of length $n$} is a finite sequence
\[
\uc = (c_1, \ldots , c_n) \label{note:profile}
\]
of colours.  We write $|\uc| = n$ for the length. The \textbf{empty profile}, with $n=0$, is denoted by $\varnothing$.
\item
Given two $\fC$-profiles $\uc, \ud$ and $i \in \{1,...,|\uc|\}$ we define the following $\fC$-profiles
\begin{align*}
    \uc \circ_i \ud &= (c_1,...,c_{i-1},d_1,...,d_{|\ud|},c_{i+1},...,c_{|\uc|})\\
    (\uc, \ud) &= (c_1,...,c_{|\uc|},d_1,...,d_{|\ud|})\\
    \uc \setminus c_i &= (c_1,...,c_{i-1},c_{i+1},...,c_{|\uc|})
\end{align*}
\item
For a $\fC$-profile $\uc$ of length $n$ and $\sigma \in \Sigma_n$, define the left and right actions
\[
\sigma\uc = \left(c_{\sigma(1)}, \ldots , c_{\sigma(n)}\right) \andspace 
\uc\sigma = \left(c_{\sigma^{-1}(1)}, \ldots , c_{\sigma^{-1}(n)}\right).
\]
\item
The groupoid of all $\fC$-profiles with left (resp., right) symmetric group actions as morphisms is denoted by $\pofc$\label{note:pc} (resp., $\pofcop$).
\item
Define the product category $\SC := \ptwoc$. Its elements are pairs of $\fC$-profiles and are written either horizontally as $\dch$ or vertically as $\dc$.\label{note:dc}
\end{enumerate}
\end{definition}

A \textbf{bimodule} is a functor $P \in \mathcal{E}^{\SC}$ (Definition 10.28 of \cite{yau2015foundation}). We may generalise this definition to a groupoid coloured bimodule as follows.

\begin{definition}\label{Groupoid Coloured Bimodule}
Let $\mathcal{W}_k(\mathbb{V}):= \mathbb{V}^k \rtimes \Sigma_k$, where we suppose that $\Sigma_k$ acts on $\mathbb{V}^k$ from the left, and let $\mathcal{W}(\mathbb{V}) := \coprod_{k\geq 0} \mathcal{W}_k(\mathbb{V})$.

\begin{itemize}
    \item A $\mathbb{V}_{\Sigma}$-\textbf{bimodule} is a functor $P \in \mathcal{E}^{\SV}$, where $\mathcal{S}(\mathbb{V}):= {\mathcal{W}(\mathbb{V})^{op}\times \mathcal{W}(\mathbb{V})}$.  
    \item A $\mathbb{V}_{\Sigma}$-\textbf{module} is a functor $P \in \mathcal{E}^{\mathcal{W}(\mathbb{V})^{op} \times \mathbb{V}}$.
    \item A \textbf{non-symmetric} groupoid coloured bimodule, or a $\mathbb{V}$-bimodule is a $\mathbb{V}_{\Sigma}$-bimodule in which $\Sigma_k$ acts trivially. Similarly, one can define a non-symmetric $\mathbb{V}$-module.
\end{itemize}

\end{definition}
In the case that $\mathbb{V}$ is the discrete category of colours, then we recover the definition of a bimodule. In the more general case, we can unpack the definition of a groupoid coloured bimodule to see that it consists of the following data.
\begin{itemize}
    \item For any pair of profiles, it has an object $P\dc \in \mathcal{E}$.
    \item A (iso)morphism $\binom{\delta}{\phi}:P\dc\to P\binom{\ud'}{\uc'}$ in $\mathcal{W}(\mathbb{V})$, is equivalently a pair of compatible permutations and isomorphisms $\binom{\delta}{\phi}=(\binom{\sigma}{\tau},\binom{g}{f})$, where $(\tau;\sigma) \in \Sigma_{|\uc|}\times \Sigma_{|\ud|}$ and $g:\sigma\ud \to \ud' \in \mathbb{V}^{|\ud|} , f:\uc'\to \uc\tau \in \mathbb{V}^{|\uc|}$.
    Thus, this morphism of $\mathcal{W}(\mathbb{V})$ induces an action of the groupoid coloured module 
    \begin{align}\label{Groupoid Coloured Bimodule Action}
        P \binom{\ud}{\uc} \xrightarrow{\binom{\delta}{\phi}=(\binom{\sigma}{\tau},\binom{g}{f})} P \binom{\ud'}{\uc'}.
    \end{align}
    We note in particular that, the identity of $\mathcal{W}(\mathbb{V})$ is  $(\binom{id}{id},\binom{id}{id})$ and composition satisfies
    \begin{align} \label{Composition in Groupoid Coloured Bimodule}
        (\binom{\sigma'}{\tau'},\binom{g'}{f'}) \circ (\binom{\sigma}{\tau},\binom{g}{f}) = (\binom{\sigma'\sigma}{\tau\tau'},\binom{g''}{f''}).
    \end{align}
    Here if $g':\sigma'\ud'\to \ud''$ and $f':\uc''\to \uc' \tau'$ then $g''$ and $f''$ are given by the clear composites
    \begin{align*}
        g''&:\sigma'\sigma \ud \to \sigma' \ud' \to \ud''\\
        f''&:\uc'' \to \uc' \tau' \to \uc \tau \tau'.
    \end{align*}
\end{itemize}

\begin{remark} \label{remark:groupoid coloured bimodules}
It is straightforward to confirm that our definition of a groupoid coloured module is dual (and isomorphic) to Definition $3.3$ of \cite{petersen2013operad} (when restricted to groupoids). We also note that our definition is equivalent to the dual of the definition of a $\mathbb{V}$-coloured sequence presented in \cite{ward2022massey} if a very minor error is corrected. In Definition 2.2.1, the (dual) definition of a $\mathbb{V}$-coloured module is explicitly unpacked (here it is called a $\mathbb{V}$-coloured sequence, and it is also restricted to the case in which $k>0$, see Remark 2.3.1 of \cite{ward2022massey}). However, in this unpacking of the definition, some necessary morphisms are omitted. 
There are more morphisms acting on $\mathbb{V}$-coloured modules than those given via the automorphisms of $\mathbb{V}$ (unless $\mathbb{V}$ is skeletal), in particular there are morphisms
\begin{align*}
    \binom{v_0}{v_1,...,v_r} \xrightarrow{(\sigma,f_1,...,f_r;f_0^{op})} \binom{v_0'}{v'_1,...,v'_r}
\end{align*}
where $\sigma \in \mathbb{S}_r$, $f_0^{op}:v_0\to v'_0$ and $f_i:v_i \to v'_{\sigma^{-1}(i)}$. This minor misidentification changes none of the results of the paper, this is because the groupoid still acts on a tree $T$ via the automorphism group of the internal edges of $T$ (using \cref{Composition in Groupoid Coloured Bimodule} of \cref{Groupoid Coloured Bimodule} to define \cref{Decorated Symmetric Tree} of \cref{Trees decorated with groupoid coloured modules}, instead of (2.1) to define (2.4) as in \cite{ward2022massey}).
\end{remark}

\begin{definition}\label{def:morphism of groupoid coloured bimodules}
Given two small groupoids $\mathbb{U}, \mathbb{V}$, and two symmetric monoidal categories $\mathcal{D}, \mathcal{E}$, let $P \in \mathcal{D}^{\mathcal{S}(\mathbb{U})}$ and $Q \in {\mathcal{E}}^{\mathcal{S}(\mathbb{V})}$ be two distinct bimodules.
A \textbf{morphism of bimodules} $h:P\to Q$ consists of a morphism of groupoids $h_0:\mathbb{U}\to \mathbb{V}$, and a sequence of set maps
\begin{center}
$\{P\dc \xrightarrow{{h_1^{\dc}}} Q\binom{h_0(\ud)}{(\uc)h_0}\}_{\dc \in ob(\SC)},\quad$ so that all squares $\quad$
\begin{tikzcd}
P\dc \arrow[d,"{(\binom{\sigma}{\tau},\binom{g}{f})}"'] \arrow[r, "{h_1^{\dc}}"] & Q\binom{h_0(\ud)}{(\uc)h_0} \arrow[d,"{(\binom{\sigma}{\tau},\binom{h_0(g)}{(f)h_0})}"]\\
P\binom{\ud'}{\uc'} \arrow[r, "{h_1^{\binom{\ud'}{\uc'}}}"] & Q\binom{h_0(\ud')}{(\uc')h_0} 
\end{tikzcd}
$\quad$ commute.
\end{center}
Above, $(\binom{\sigma}{\tau},\binom{g}{f})$ is any action of the $\mathbb{U}_\Sigma$-bimodule $P$ given by \cref{Groupoid Coloured Bimodule Action}, and in an abuse of notation we use $h_0(-), (-)h_0$ to denote the morphisms of groupoids $\mathbb{U}^{|\ud|}\to \mathbb{V}^{|\ud|}$, and respectively ${\mathbb{U}^{|\ud|}}^{op}\to {\mathbb{V}^{|\ud|}}^{op}$, induced by $h_0:\mathbb{U} \to \mathbb{V}$.
To be explicit $h_0(\ud):= (h_0(d_1),...,h_0(d_k))$, and if $g=(g_1,...,g_{\ud}) \in Hom_{\mathbb{V}^{|\ud|}}(\sigma \ud, \ud')$ then $h_0(g)= (h_0(g_1),...,h_0(g_{\ud}))\in  Hom_{\mathbb{U}^{|\ud|}}(h_0(\sigma \ud), h_0(\ud'))$.
\\\\
A morphism $h$ of groupoid coloured operads is said to be 
\begin{itemize}
    \item \textbf{full}, if the sequence of set maps are injective,
    \item \textbf{faithful}, if the sequence of set maps are surjective,
    \item an \textbf{epimorphism}, if it is faithful, and the underlying morphism of groupoids is essentially surjective,
    \item an \textbf{isomorphism}, if it is fully faithful and $h_0$ is an isomorphism of groupoids.
\end{itemize}
We similarly define and characterise the morphisms of (non-symmetric) groupoid coloured modules.
Other familiar types of morphisms are possible, but will not be needed in this text.
\end{definition}

We note this definition of a morphism of groupoid coloured bimodules allows for both differing groupoids and differing ambient symmetric monoidal categories.
This extension to allow differing groupoids will be essential in the study of particular morphisms of groupoid coloured operads (see \cref{The Groupoid Coloured Extension} and \cref{inclusion of com}).
If we specialise \cref{def:morphism of groupoid coloured bimodules} to a fixed groupoid $\mathbb{V}$ and fixed symmetric monoidal category $\mathcal{E}$, then a morphism of groupoid coloured bimodules is a natural transformation.

\subsection{Fundamental Graphical Definitions}

Operations in (wheeled) props are parameterised by wheeled graphs. 
In this section, we provide a formal definition of a wheeled graph (\cref{Wheeled Graph}) and describe when two such graphs are isomorphic (\cref{isomorphic graphs}). The definitions of this subsection all stem verbatim from \cite{hackney2015infinity} and \cite{yau2015foundation}, further related definitions and examples may be found there.

\subsubsection{Generalised Graphs}

Fix an infinite set $\fF$ once and for all.

\begin{definition}
\label{def:graph}
A \textbf{generalised graph} $G$ is a finite set $Flag(G) \subset \fF$ with
\begin{itemize}
\item
a partition $Flag(G)= \coprod_{\alpha \in A} F_\alpha$ with $A$ finite,
\item
a distinguished partition subset $F_\epsilon$ called the \textbf{exceptional cell},
\item
an involution $\iota$ satisfying $\iota F_\epsilon \subseteq F_\epsilon$, and
\item
a free involution $\pi$ on the set of $\iota$-fixed points in $F_\epsilon$.
\end{itemize}
\end{definition}
For a given generalised graph $G$, we unpack this definition and introduce some notation. The elements of $Flag(G)$ are called \textbf{flags} (or half edges). Every non-exceptional partition subset of $Flag(G)$ say $F_{\alpha}\neq F_{\epsilon}$ is a \textbf{vertex}, and if $f \in F_{\alpha}$ we say the flag $f$ is \textbf{adjacent} (or attached) to the vertex. We note it is possible for a vertex to have no adjacent flags, in which case $F_{\alpha} = \{\emptyset\}$. The exceptional cell $F_{\epsilon}$ then corresponds to flags which are 'free floating' in the graph, and are not associated to any vertices. We shall denote the vertices of a graph as $Vt(G)$.
\\\\
The non-trivial orbits of the involutions are called \textbf{edges} of the graph, and the fixed points of the involutions are called \textbf{legs} of the graph. Orbits of $\iota$ are called \textbf{internal edges}, whereas orbits of $\pi$ are called \textbf{exceptional edges}. Note that because the involution $\pi$ is free on the fixed points of $\iota$, every fixed point of $\iota$ is part of an exceptional edge. Non-trivial orbits of flags in $F_{\epsilon}$ are called \textbf{exceptional loops}. Finally, the fixed points of $\iota$ (there are no fixed points of $\pi$) are called \textbf{legs} of the graph, these can also be considered exceptional (or ordinary) if they are (not) in $F_{\epsilon}$.
\begin{example}[1.12 of \cite{yau2015foundation}]
Consider the following diagrammatic representation of a generalised graph $\gamma$
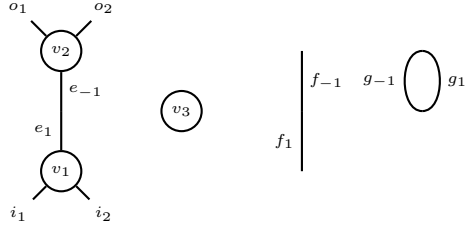
\begin{center}
\tiny{
\begin{tikzpicture}[node distance={8mm},thick] 
\node[vertex] (v2) {$v_2$};
\node[flag] [above left of=v2] (o1) {$o_1$};
\node[flag] [above right of=v2] (o2) {$o_2$};
\coordinate[below of=v2] (e);
\node[vertex] (v1) [below of=e]{$v_1$};
\node[flag] [below left of=v1] (i1) {$i_1$};
\node[flag] [below right of=v1] (i2) {$i_2$};
\draw [black] (v2) to (o1);
\draw [black] (v2) to (o2);
\draw (v2) edge["${e_{-1}}$"] (e);
\draw (v1) edge["${e_{1}}$"] (e);
\draw [black] (v1) to (i1);
\draw [black] (v1) to (i2);
\coordinate[right of=e] (gap1);
\node[vertex] (v3) [right of=gap1]{$v_3$};
\coordinate[right of=v3] (gap2);
\coordinate[right of=gap2] (f);
\coordinate[above of=f] (f-);
\coordinate[below of=f] (f+);
\draw (f-) edge["${f_{-1}}$"] (f);
\draw (f+) edge["${f_{1}}$"] (f);
\coordinate[right of=f] (gap3);
\coordinate[right of=gap3] (gbot);
\coordinate[above of=gbot] (gtop);
\draw (gbot) edge["${g_{-1}}$", bend left=90] (gtop);
\draw (gtop) edge["${g_{1}}$", bend left=90] (gbot);
\end{tikzpicture}
}
\end{center}
The underlying partition set of $\gamma$ is $\{\{i_1,i_2,e_1\}, \{o_1,o_2,e_{-1}\},\emptyset, \{f_{\mp1},g_{\mp1}\}\}$ with the last element being the exceptional cell. The $2$ orbits of the involution $\iota$ are $e_{\mp1}$ and $g_{\mp1}$, all other flags are fixed points. The involution $\pi$ has one $2$ orbit corresponding to $f_{\mp1}$ all other flags are fixed. The graph $\gamma$ has three vertices, a single exception edge with flags $f_{\mp1}$ and a single exceptional loop with flags $g_{\mp1}$.

\end{example}

\subsubsection{Structures on Generalised Graphs}

We now introduce further structure on generalised graphs in order to define wheeled props. The extra data we introduce is: a colour for each edge; an orientation for each edge; and a labelling of the incoming/outgoing flags of each vertex, as well as labelling of the generalised graph. 

\begin{definition}
Suppose $G$ is a generalised graph.
\begin{enumerate}
    \item A \textbf{colouring} for $G$ is a $\Flag(G) \xrightarrow{\kappa} \fC$
    that is constant on orbits of both involutions $\iota$ and $\pi$.
    \item A \textbf{direction} for $G$ is a function $\Flag(G) \xrightarrow{\delta} \{-1,1\}$
    such that
    \begin{itemize}
        \item if $\iota x \not= x$, then $\delta(\iota x) = -\delta(x)$, and
        \item if $x \in F_\epsilon$, then $\delta(\pi x) = -\delta(x)$.
    \end{itemize}
    \item For $G$ with direction, an \textbf{input} (resp. \textbf{output}) of a vertex $v$ is a flag $x \in v$ such that $\delta(x) = 1$ (resp. $\delta(x) = -1$).  An \textbf{input} (resp., \textbf{output}) of the graph $G$ is a leg $x$ such that $\delta(x) = 1$ (resp. $\delta(x) = -1$).  For $u \in \{G\} \cup \vertex(G)$, the set of inputs (resp. outputs) of $u$ is written as $\inp(u)$ (resp. $\out(u)$). \label{note:in}
    \item A \textbf{listing} for $G$ with direction is a choice for each $u \in \{G\} \cup \vertex(G)$ of a bijection of pairs of sets,
    $$
    \left(\inp(u), \out(u)\right)
    \xrightarrow{\ell_u} 
    \left(\{1,\ldots,|\inp(u)|\}, \{1, \ldots , |\out(u)|\}\right),
    $$
    where for a finite set $T$ the symbol $|T|$ denotes its cardinality.
\end{enumerate}
\end{definition}

\begin{definition}\label{Wheeled Graph}
A ($\fC$-coloured) \textbf{wheeled graph}, is a generalised graph together with a choice of a colouring, a direction, and a listing. For such a graph $G$, we may use the colouring, listing, and direction to speak of the profile of the graph and any of its vertices. If $u \in \{G\} \cup \vertex(G)$, with $in(u) = \uc$ and $out(u)=\ud$ (treating these as sets ordered by the listing) then we say that $u$ has profile $\dc$.
\\\\
A \textbf{wheel} of a (directed) wheeled graph is a directed cycle. A \textbf{wheel-free} graph is a wheeled graph without any wheels.
\end{definition}

\textbf{A convention:} All graphs in this paper will be drawn with their inputs on the bottom and outputs on the top, e.g. the two graphs in \cref{Isomorphic Graphs} each have one vertex with two inputs and no outputs.

\subsubsection{Isomorphic Graphs}

So having defined what a graph is, all that remains is the question of when should two graphs be considered the same.

\begin{definition}[4.1 and 4.15 of \cite{yau2015foundation}]\label{isomorphic graphs}
Let $G$ and $G'$ be wheeled graphs. 
\begin{itemize}
    \item A \textbf{weak isomorphism} $\varphi:G\to G'$ is a bijection of partitioned sets that commutes with both involutions and leaves invariant both the colouring and the directions.
    \item A \textbf{strict isomorphism} $\varphi: G\to G'$ is a weak isomorphism that also leaves invariant the listing.
\end{itemize}

\end{definition}

\begin{example}\label{Isomorphic Graphs}
The following two graphs whose listings are given by the planar embedding are weakly isomorphic, but not strongly isomorphic.
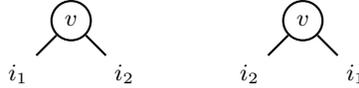
\begin{center}
\small{
\begin{tikzpicture}[node distance={10mm},thick] 
\node[vertex] (v1) {$v$};
\node[flag] [below left of=v1] (i1) {$i_1$};
\node[flag] [below right of=v1] (i2) {$i_2$};
\draw [black] (v1) to (i1);
\draw [black] (v1) to (i2);
\end{tikzpicture}
$\quad \quad \quad$
\begin{tikzpicture}[node distance={10mm},thick] 
\node[vertex] (v1) {$v$};
\node[flag] [below left of=v1] (i1) {$i_2$};
\node[flag] [below right of=v1] (i2) {$i_1$};
\draw [black] (v1) to (i1);
\draw [black] (v1) to (i2);
\end{tikzpicture}
}
\end{center}
\end{example}

\subsubsection{Trees}

Operations in operads are parameterised by trees. So, in this section, we specialise the prior definitions of wheeled graphs and their isomorphisms to describe this important type of graph.

\begin{definition} \label{tree definitions}
A \textbf{tree} is a simply connected graph where every vertex has a single output. We shall refer to the sole output of a tree as the root. We say a tree $T$ with profile $(\uc;d)$ is:
\begin{itemize}
    \item \textbf{reduced} if every vertex has at least one input.
    \item a \textbf{$\fC$-tree} if it is a reduced tree with a colouring $\fC$ (a key colouring being $ob(\mathbb{V})$).
    \item \textbf{level} if the height of every input leg (leaf) is the same. The height of a leg or vertex of a tree is defined to be the number of vertices on the directed path from it to the root. We say a tree with $0$ levels is \textbf{empty}.
    \item a \textbf{shuffle} tree if for every vertex $v$ of the tree say with profile $(v_1,...,v_n;v_0)$, that for $ 1\leq i<j \leq n$ we have that $min(\uc_i) < min(\uc_j)$ where $\uc_k = \{ c\in \uc:$ there exists a directed path from $c$ to $v_k$\}. 
    \item a \textbf{strongly regular} tree if the listing of $T$ is given by performing a depth first search on the tree using the listing of the vertices to determine the order of the exploration of the children. In other words, the inputs of the tree are ordered left to right through embedding the tree into the plane. 
\end{itemize}
Examples of shuffle trees and strongly regular trees are given in \cref{Reducing to Equivalent Definitions for Coloured Operads}.

\end{definition}
We note that this definition of a $ob(\mathbb{V})$-coloured tree is the same as Definition 2.3 of \cite{ward2022massey}, with the additional caveat that every internal vertex of the tree has an explicit listing as well. We explicitly refer to these trees as being \textbf{object-coloured} to stress that the morphisms of the groupoid are not (yet) involved. 
\\\\
Using the prior definitions of isomorphic graphs, we define the following classes of trees.
\begin{definition}\label{def:isomorphism classes of trees}
Let,
\begin{itemize}
    \item $\symmetricTreesClassesWithEmpty:=$ weak isomorphism classes of, possibly empty, coloured trees. 
    \item $\strict:=$ strict isomorphism classes of, possibly empty, coloured strongly regular trees.
    \item $\shuffleTreesClassesWithEmpty:=$ strict isomorphism classes of, possibly empty, coloured shuffle trees.
\end{itemize}
If we wish to exclude the empty trees, we shall drop the over-line.
The subscripts $-_{n,\vec{v}}$ refer to the restrictions of these classes to those trees with $n$ levels and profile $\vec{v}$.
\end{definition}
We note that the notion of an isomorphism of $ob(\mathbb{V})$ coloured trees in \cite{ward2022massey} (Definition 2.9) is equivalent to that of a weak isomorphism.

\section{Groupoid Coloured Operads} \label{groupoid coloured operads}

In this section, we present alternate definitions (monoidal and monadic) of variants (symmetric, non-symmetric and shuffle) of groupoid coloured operads, and define what it means for these operads to be Koszul.
The necessary theory is a natural amalgamation of shuffle operads (\cite{dotsenko2010grobner}, \cite{kharitonov2021grobner}) and groupoid coloured operads (\cite{petersen2013operad}, \cite{ward2022massey}).
We note that groupoid coloured operads are a specific instance of category coloured operads \cite{petersen2013operad}, and the latter is explicitly related to internal operads \cite{batanin2018regular}, and substitudes \cite{day2003abstract}, in \cite{trnka2023category}.
Furthermore, \cite{batanin2018regular} establishes the connection between substitudes \cite{day2003abstract}, regular patterns \cite{getzler2009operads}, and Feynman categories \cite{kaufmann2017feynman}.

\subsection{Monoidal Groupoid Coloured Operads}

We start by presenting monoidal definitions of the operads and providing fundamental examples. Our definition of a symmetric groupoid coloured operad is precisely that of Ward's (accounting for \cref{remark:groupoid coloured bimodules}). To define our monoids, we will decorate trees with groupoid coloured modules, but because our trees are reduced, we will need the following assumption.
\begin{definition} \label{reduced}
A groupoid coloured module $A:\mathcal{W}(\mathbb{V})^{op}\times \mathbb{V}\to \mathcal{E}$ is said to be \textbf{reduced} if, for all $v\in \mathbb{V}$, $A(\emptyset,v) = \emptyset$, the terminal object in $\mathcal{E}$. 
\end{definition}
For the remainder of this paper, we assume all groupoid coloured modules (symmetric or non-symmetric) are reduced. 
That is to say, all of our 'operations' have at least one input.
We make this assumption to align with the theory of Ward (see Definition 2.35 of \cite{ward2022massey}).
It is also required to establish a monoidal functor between symmetric and shuffle groupoid coloured operads (\cref{Forgetful functor is monoidal}).
\\\\
Recall from Definitions \ref{Wheeled Graph} and \ref{tree definitions} that every vertex $w$ of a tree $T$ has a listing/profile.

\begin{definition} [The objects $A_\Sigma(T)$ and $A(T)$] \label{Trees decorated with groupoid coloured modules}
Let $A_\Sigma$ be a $\mathbb{V}_{\Sigma}$-module, and let $A$ be a $\mathbb{V}$-module.
Let $w$ be a vertex of a $ob(\mathbb{V})$-tree with profile $(v_1,...,v_r;v_0)$ then set
\begin{align*}
    A_\Sigma(w) &:= [ \coprod_{\sigma \in S_r} A_\Sigma(v_{\sigma(1)},...,v_{\sigma(r)};v_0)]_{S_r}\\
    A(w) &:= A(v_1,...,v_r;v_0)
\end{align*}
Here $A_\Sigma(w)$ inherits a natural action of $(\prod_i^U Hom_{\mathbb{V}}(-,v_i))\times Hom_{\mathbb{V}}(v_0,-)$ where $\times_i^U$ is the unordered product, and $A(w)$ inherits a natural action of $(\prod_i Hom_{\mathbb{V}}(-,v_i))\times Hom_{\mathbb{V}}(v_0,-)$. We also explicitly define our objects for empty trees,
\begin{align*}
    A_{\Sigma}(|_v) = A(|_v):= \coprod_{Aut(v)} 1_{\mathcal{E}}
\end{align*}
with their induced action by $Hom(-,v)\times Hom(v,-)$. Let $T$ be a $ob(\mathbb{V})$-tree and form the unordered tensor products $\bigotimes_{w \in Vt(T)} A_\Sigma(w)$ and $\bigotimes_{w \in Vt(T)} A(w)$. Then any edge $E$ of $T$ with colour $v_E$, input flag $x$ and output flag $y$ provides the group $Aut(v_E)$ with an action on this tensor product by acting simultaneously on the vertex connected to $x$ on the left and the vertex connected to $y$ on the right. This action commutes across edges by \cref{Composition in Groupoid Coloured Bimodule}, and hence provides an action of the group $\prod^U_{E \in Edges(T)} Aut(v_E)$ on these tensor products. We let $[-]_{Edges(T)}$ denote the coinvariants of this action and define,
\begin{align}
    A_\Sigma(T) &:= [ \bigotimes_{w \in Vt(T)} A_\Sigma(w)]_{Edges(T)} \label{Decorated Symmetric Tree}\\
    A(T) &:= [ \bigotimes_{w \in Vt(T)} A(w)]_{Edges(T)} \label{Decorated Non-Symmetric Tree}
\end{align}
\end{definition}

\begin{definition}
For a tree $T$ with $n$ levels, $\mathbb{V}_\Sigma$-modules $A_{\Sigma,1},...,A_{\Sigma,n}$, and $\mathbb{V}$-modules $A_{1},...,A_{n}$, let
\begin{align*}
    (A_{\Sigma,1} \circ ... \circ A_{\Sigma,n})(T)&:=  [ \bigotimes_{w \in Vt(T)} A_{\Sigma,ht(w)}(w)]_{Edges(T)}\\
    (A_{1} \circ ... \circ A_{n})(T)&:=  [ \bigotimes_{w \in Vt(T)} A_{ht(w)}(w)]_{Edges(T)}
\end{align*}
We then define the following symmetric, non-symmetric, and shuffle monoidal products as,
\begin{align*}
    (A_{\Sigma,1} \circ ... \circ A_{\Sigma,n})(\vec{v})&:= \coprod_{T\in \symmetricTreesClassesWithEmpty_{n,\vec{v}}} (A_{\Sigma,1} \circ ... \circ A_{\Sigma,n})(T );\\
    (A_{1} \circ ... \circ A_{n})(\vec{v})&:= \coprod_{T \in \strict_{n,\vec{v}}} (A_{1} \circ ... \circ A_{n})(T);\\
    (A_{1} \circ_{sh} ... \circ_{sh} A_{n})(\vec{v})&:= \coprod_{T\in \shuffleTreesClassesWithEmpty_{n,\vec{v}}} (A_{1} \circ ... \circ A_{n})(T).
\end{align*}
\end{definition}

\begin{lemma} \label{Monoidal Product Lemma}
Each of the products above give the underlying category a strict monoidal structure with unit
\begin{align*}
    \mathcal{I}(\vec{v}):=\begin{cases}
        \coprod_{Aut(v)}1_{\mathcal{E}}, & \vec{v} = (v;v) \text{ for some } v\in Ob(\mathbb{V})\\
        \emptyset,& \text{else}
    \end{cases}
\end{align*}
\end{lemma}
This (to be proven) lemma provides the following monoidal definitions of operads.
\begin{definition}\label{monoidal groupoid coloured operads def}
$\quad$
\begin{enumerate}
    \item A \textbf{symmetric groupoid coloured operad} is a monoid in the category of $\mathbb{V}_{\Sigma}$-modules with the monoidal product given by symmetric composition.
    \item A \textbf{non-symmetric groupoid coloured operad} is a monoid in the category of $\mathbb{V}$-modules with the monoidal product given by non-symmetric composition.
    \item A \textbf{shuffle groupoid coloured operad} is a monoid in the category of $\mathbb{V}$-modules with the monoidal product given by shuffle composition.
\end{enumerate}
\end{definition}
We now prove \cref{Monoidal Product Lemma}.

\begin{proof}
For the symmetric case, we refer the reader to Proposition 3.6 of \cite{petersen2013operad} and Lemma 2.22 of \cite{ward2022massey} (accounting for \cref{remark:groupoid coloured bimodules}). We shall only explicitly prove the shuffle case, and the other cases follow similarly. Let $Vt(T)_n = \{v \in Vt(T):ht(v)=n\}$,  then
\begingroup
\allowdisplaybreaks
\begin{align*}
    (A_{1} \circ_{sh}( A_{2} \circ_{sh} A_3))(\vec{v})=& \coprod_{T \in \shuffleTreesClassesWithEmpty_{2,\vec{v}}} (A_{1} \circ (A_{2} \circ_{sh} A_3))(T)\\
    =& \coprod_{T \in \shuffleTreesClassesWithEmpty_{2,\vec{v}}} 
    [ \bigotimes_{w_1 \in Vt(T)_1} A_1(w_1) \bigotimes_{w_2 \in Vt(T)_2} (A_{2} \circ_{sh} A_3)(w_2) ]_{Edges(T)} \\
    =& \coprod_{T \in \shuffleTreesClassesWithEmpty_{2,\vec{v}}} 
    [ \bigotimes_{w_1 \in Vt(T)_1} A_1(w_1) \bigotimes_{w_2 \in Vt(T)_2}
    (\coprod_{T_2 \in \shuffleTreesClassesWithEmpty_{2,\vec{w_2}}} (A_{2} \circ A_3)(T_2) ) ]_{Edges(T)} \\
    =& \coprod_{T \in \shuffleTreesClassesWithEmpty_{2,\vec{v}}} 
    [ \bigotimes_{w_1 \in Vt(T)_1} A_1(w_1) \bigotimes_{w_2 \in Vt(T)_2}
    (\\
    &\quad\quad\coprod_{T_2 \in \shuffleTreesClassesWithEmpty_{2,\vec{w_2}}} [\bigotimes_{w_1' \in Vt(T_2)_1} A_2(w_1') \bigotimes_{w_2' \in Vt(T_2)_2} A_3(w_2') ]_{Edges(T_2)} \\
    &)]_{Edges(T)} \\
    =& \coprod_{T' \in \shuffleTreesClassesWithEmpty_{3,\vec{v}}} 
    [ \bigotimes_{w_1 \in Vt(T')_1} A_1(w_1) \bigotimes_{w_2 \in Vt(T')_2} A_2(w_2)
    \bigotimes_{w_3 \in Vt(T')_3} A_3(w_3)
    ]_{Edges(T')} \\
    =&(A_{1} \circ_{sh} A_{2} \circ_{sh} A_3)(\vec{v})
\end{align*}
\endgroup
The second to last line follows as $\mathcal{E}$ is a closed monoidal category, so we can iteratively commute the colimits and monoidal products. 
The grouping of the two coproducts into the single coproduct $\coprod_{T' \in \shuffleTreesClassesWithEmpty_{3,\vec{v}}}$ is done by the obvious bijection. Similar working shows $((A_{1} \circ_{sh} A_{2}) \circ_{sh} A_3)(\vec{v}) = (A_{1} \circ_{sh} A_{2} \circ_{sh} A_3)(\vec{v})$ yielding the associativity of the product. The required unitality results are then easily verified, giving the result.
\end{proof}

\begin{example}\label{Reducing to Equivalent Definitions for Coloured Operads}
Groupoid coloured operads coloured by a discrete category are coloured operads. This can be seen in the case of $\mathcal{E}=Vect$, the category of vector spaces, by observing that the defined monoidal products are equivalent to those of Definition 1.5 of \cite{kharitonov2021grobner}. This follows from inspection, however we clarify the following bijections to structures used in their definition.
\begin{itemize}
    \item A non-decreasing surjection $f:[m]\to[n]$ is a linear tree of height $2$. In particular if $i_n = \min \{i:f(i)=n\}$ then the corresponding linear tree is 
    
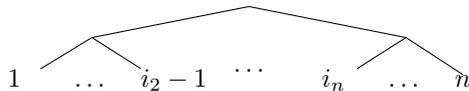
\begin{center}
    \begin{forest}
    fairly nice empty nodes,
    for tree={inner sep=0, l=0}
    [
    [
        [{\makebox[0pt][l]{$1$}\phantom{\dots}}]
        [{\makebox[0pt][l]{$\dots$}\phantom{\dots}},no edge]
        [{\makebox[0pt][l]{$i_2-1$}\phantom{\dots}}]
    ]
    [,no edge[{$\quad \dots\quad$},no edge]]
    [
        [{\makebox[0pt][l]{$i_n$}\phantom{\dots}}]
        [{\makebox[0pt][l]{$\dots$}\phantom{\dots}},no edge]
        [{\makebox[0pt][l]{$n$}\phantom{\dots}}]
    ]
    ]
    \end{forest}
    \end{center}
    \item A shuffling surjection $f:[m]\to[n]$ (i.e. one in which $\min f^{-1}(i)<\min f^{-1}(j)$ for all $i<j$) is a shuffle tree of height $2$. In particular if $i_{1,k}<i_{2,k}<...<i_{|f^{-1}(k)|,k}$ are the elements of $f^{-1}(k)$ ordered smallest to largest then the corresponding shuffle tree is 
    
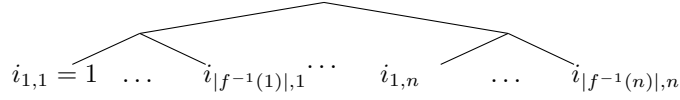
\begin{center}
    \begin{forest}
    fairly nice empty nodes,
    for tree={inner sep=0, l=0}
    [
    [
        [{\makebox[0pt][l]{$i_{1,1}=1$}\phantom{..\dots..}}]
        [{\phantom{..}\dots\phantom{..}},no edge]
        [{\makebox[0pt][l]{$i_{|f^{-1}(1)|,1}$}\phantom{..\dots..}}]
    ]
    [,no edge[{\phantom{.}\dots\phantom{.}},no edge]]
    [
        [{\makebox[0pt][l]{$i_{1,n}$}\phantom{..\dots..}}]
        [{\phantom{..}\dots\phantom{..}},no edge]
        [{\makebox[0pt][l]{$i_{|f^{-1}(n)|,n}$}\phantom{..\dots..}}]
    ]
    ]
    \end{forest}
    \end{center}
\end{itemize}
\end{example}

We make the following definitions in light of this example.
\begin{definition}\label{discrete coloured operads}
We refer to (non-groupoid) coloured operads as \textbf{discrete coloured operads}.
Given a groupoid $\mathbb{V}$, a \textbf{$ob(\mathbb{V})$-coloured operad} is a discrete coloured operad with objects/colours $ob(\mathbb{V})$.
\end{definition}

Another fundamental example of a groupoid coloured operad is the endomorphism operad.
\begin{definition} [Example 2.15 of \cite{ward2022massey}]
\label{Endomorphism Operad}

A functor $X:\mathbb{V}\to \mathcal{E}$ admits a natural extension into a groupoid coloured module $End_X$. For each profile $(v_1,...,v_n;v_0)$
\begin{align*}
    End_X(v_1,...,v_n;v_0):= Hom_{\mathcal{E}}(X(v_1),...,X(v_n);X(v_0))
\end{align*}
with inherited action of $(\times_i^U Hom_{\mathbb{V}}(-,v_i))\times Hom_{\mathbb{V}}(v_0,-)$, and the $\Sigma_n$ action given by permuting the labels. This groupoid coloured module admits the structure of a unital $\mathbb{V}$-coloured operad with composition given by composition of functions.
We call $End_X$ the \textbf{endomorphism operad}.
\end{definition}
The operad structure of $End_X$ is well-defined, as associativity of functions not only yields associativity of the operadic structure maps, but also compatibility of the structure maps with the action of groupoid on the internal edges.
A \textbf{morphism} of (groupoid coloured) operads is a morphism of the (groupoid coloured) modules (\cref{def:morphism of groupoid coloured bimodules}), which is also compatible with the operadic structure maps.

\begin{definition} [2.3.2 of \cite{ward2022massey}] \label{algebra over an operad}
Let $X:\mathbb{V}\to \mathcal{E}$ be a functor and $P$ be a $\mathbb{V}$-coloured operad. 
A \textbf{$P$-algebra structure on $X$} is a morphism of groupoid coloured operads $P\to End_X$.
\end{definition}

If this reader at this stage wants a more sophisticated example of a groupoid coloured operad, then it is possible to skip ahead and read \cref{The Operad Governing Wheeled Props}-\ref{biased presentation of the operad governing wheeled props} which includes a nice example of the equivalence classes induced by the action of the groupoid on internal edges of trees (\cref{Example of Action of the Groupoid}). This later section does use some terminology introduced in the remainder of this section, but should be relatively clear to the reader already familiar with operads. Other examples of groupoid coloured operads include the operad governing props (this is less approachable for technical reasons, but can be found \cref{The Operad Governing Props}), and the operad governing modular operads (see Section 3.1 of \cite{ward2022massey}).
\\\\
We close this section by noting that there is a natural forgetful functor $-^f$ from $\mathbb{V}_\Sigma$-modules to $\mathbb{V}$-modules, which forgets the action of the group of symmetries.
\begin{lemma} \label{Forgetful functor is monoidal}
Given $\mathbb{V}_{\Sigma}$-modules $P$ and $Q$ we have that 
\begin{align*}
    (P\circ Q)^f \cong P^f\circ_{sh} Q^f
\end{align*}
\end{lemma}
\begin{proof}
As the forgetful functor does not touch the underlying coinvariants given by the groupoid, this is an immediate corollary of the corresponding result for discrete coloured operads in \cite{kharitonov2021grobner} (Section 1.7, which is itself a corollary of Proposition 3 of \cite{dotsenko2010grobner}), given we have assumed our modules are reduced (\cref{reduced}). 
\end{proof}

The forgetful functor is not only straightforward to calculate in practice (for instance see \cref{gb assoc example}, and Section \ref{shuffle operad for wheeled props}), but also enables us to characterise when $O$ is Koszul by inspecting $O^f$ (\cref{def:koszul}).

\subsection{Monadic Groupoid Coloured Operads}

In this section, we introduce monadic definitions of our operads and prove their equivalence to the monoidal definitions.
In many cases it is more convenient to work with non-unital variants of operads (e.g. Sections \ref{The Operad Governing Wheeled Props}, \ref{The Operad Governing Props}), and monadic definitions provide a clear path to non-unital variants. We note that our monadic definition of a symmetric groupoid coloured operad is that of \cite{ward2022massey}, and the others are natural alterations.

\begin{definition}
We define the following endofunctors
\begin{itemize}
    \item Let $\mathbb{T}_{\mathbb{V},{\Sigma}},\overline{\mathbb{T}_{\mathbb{V},{\Sigma}}}:\mathbb{V}_{\Sigma}$-modules $\to \mathbb{V}_{\Sigma}$-modules be defined on objects as 
    \begin{align*}
        &\mathbb{T}_{\mathbb{V},\Sigma}(A)(\vec{v}):= \coprod_{T \in \symmetricTreesClasses} A_{\Sigma}(T), &\overline{\mathbb{T}_{\mathbb{V},\Sigma}}(A)(\vec{v}):= \coprod_{T \in \symmetricTreesClassesWithEmpty} A_{\Sigma}(T)
    \end{align*}
    \item Let $\mathbb{T}_{\mathbb{V}},\overline{\mathbb{T}_{\mathbb{V}}}:\mathbb{V}$-modules $\to \mathbb{V}$-modules be defined on objects as 
    \begin{align*}
        &\mathbb{T}_{\mathbb{V}}(A)(\vec{v}):= \coprod_{T \in \regularTreesClasses} A(T), &\overline{\mathbb{T}_{\mathbb{V}}}(A)(\vec{v}):= \coprod_{T \in \strict} A(T)
    \end{align*}
    \item Let $\mathbb{T}_{\mathbb{V},{sh}},\overline{\mathbb{T}_{\mathbb{V},{sh}}}:\mathbb{V}$-modules $\to \mathbb{V}$-modules be defined on objects as 
    \begin{align*}
        &\mathbb{T}_{\mathbb{V},sh}(A)(\vec{v}):= \coprod_{T \in \shuffleTreesClasses} A(T), &\overline{\mathbb{T}_{\mathbb{V},sh}}(A)(\vec{v}):= \coprod_{T \in \shuffleTreesClassesWithEmpty} A(T)
    \end{align*}
\end{itemize}
Morphisms of $\mathbb{V}_{\Sigma}$ (and $\mathbb{V}$)-modules are natural transformations, and so they are given by collections of equivariant maps from $A_{\Sigma}(\vec{v})\to B_{\Sigma}(\vec{v})$, which induce maps from $A_{\Sigma}(T)\to B_{\Sigma}(T)$ (and similarly for $A(-)$). Taking the coproduct of these maps specifies the endofunctors on morphisms.
\end{definition}

\begin{lemma}
Substitution of trees endows these functors with the structure of a monad.
\end{lemma}
\begin{proof}
See Definition 2.1.1 of \cite{ward2022massey} for a definition of substitution of trees, it is straightforward to modify this definition to also track the listings of the internal vertices (as required by \cref{tree definitions}). From here, the proofs for $\mathbb{T}_{\mathbb{V},\Sigma}$ and $\overline{\mathbb{T}_{\mathbb{V},{\Sigma}}}$ are Theorems $2.10$ and $2.11$ of \cite{ward2022massey}. The modifications of these proofs to yield the results for the other functors is straightforward. 
\end{proof}

\begin{definition} \label{monadic groupoid coloured operads def}
These yield the following monadic definitions of groupoid coloured operads,
\begin{itemize}
    \item A symmetric (non-)unital groupoid coloured operad is a algebra over the monad $\overline{\mathbb{T}_{\mathbb{V},\Sigma}}$ (resp. ${\mathbb{T}_{\mathbb{V},\Sigma}})$.
    \item A non-symmetric (non-)unital groupoid coloured operad is a algebra over the monad $\overline{\mathbb{T}_{\mathbb{V}}}$ (resp. ${\mathbb{T}_{\mathbb{V}}})$.
    \item A shuffle (non-)unital groupoid coloured operad is a algebra over the monad $\overline{\mathbb{T}_{\mathbb{V},sh}}$ (resp. ${\mathbb{T}_{\mathbb{V},sh}})$.
\end{itemize}
By this definition, each groupoid coloured operad (of any type) $P$ comes with a morphism $\eta_T:P(T) \to P(\vec{v})$ for every $ob(\mathbb{V})$-tree $T$ of profile $\vec{v}$, this morphism 'contracts' the tree.

\end{definition}

\begin{lemma}\label{equivalence of monodial and monadic defs for groupoid coloured operads}
The unital monadic definitions (\cref{monadic groupoid coloured operads def}) are equivalent to the monoidal definitions (\cref{monoidal groupoid coloured operads def}).
\end{lemma}
\begin{proof}
A sketch of the proof for symmetric operads is given in Theorem 2.2.4 of \cite{ward2022massey}. The proof is modified in the obvious ways for non-symmetric and shuffle operads.
\end{proof}

\begin{definition} [2.19 of \cite{ward2022massey}] \label{Augmented Operad}
An \textbf{augmentation} of a unital groupoid coloured operad $P$ is a morphism $P\to \mathcal{T}$ where $\mathcal{T}$ is the unique unital operad structure on the $\mathbb{V}$-module given by
\begin{align*}
    \mathcal{T}(\vec{v}) := \begin{cases}
        1_{\mathcal{E}}, & \vec{v}=(v;v)\text{ for some $v \in Ob(\mathbb{V})$}\\
        \text{terminal object in $\mathcal{E}$}, & \text{otherwise}
    \end{cases}
\end{align*}
An \textbf{augmented operad} is an operad with an augmentation. If $\mathcal{E}$ is an Abelian category, then the augmentation ideal is the kernel of these maps, and we denote it $\overline{P}$.
\end{definition}
\begin{corollary}\label{Augmented Operad Isomorphic to Kernel}
An augmented operad $P$ in an Abelian category is isomorphic to $\overline{P}$.
\end{corollary}
\begin{proof}
Immediate, but see for instance \cite{markl2008operads}, Proposition 21.
\end{proof}

\subsection{Partial Compositions}

We will use partial operadic compositions throughout this paper (see \cite{markl2008operads} for the uncoloured case). 
Instead of defining partial groupoid coloured operads and proving the equivalence of this definition (which is possible but tedious), we will simply define a partial composition as monadic contractions (\ref{monadic groupoid coloured operads def}) of particular trees.
\begin{definition}
Let $P$ be a groupoid coloured shuffle operad we define the partial composition
\begin{align*}
   P\binom{d}{\underline{c}}\otimes P\binom{c_i}{\underline{b}} \xrightarrow{\circ_{i,\sigma}} P\binom{d}{(\underline{c} \circ_i \underline{b})\sigma}
\end{align*}
by the contraction map applied to a shuffle tree $T$ with $2$ vertices. In particular, for $\alpha \in P\binom{d}{\underline{c}}$ and $\beta \in P\binom{b}{\underline{a}}$ we define $\alpha \circ_{i,\sigma} \beta$ to be the contraction map $\eta_T$ applied to the shuffle tree $T$ with: $2$ vertices; the root vertex decorated by $\alpha$; the lower vertex decorated by $\beta$; the lower vertex connected to the $i$th input flag of the root; and the listing of the legs of the tree specified by $\sigma$. In other words, $\sigma$ is a \textbf{shuffle permutation}, which is to say that 
\begin{itemize}
    \item for $j\leq i$, $\sigma(j)=j$
    \item $\sigma(i+1)<\sigma(i+1)<...<\sigma(i+|\ub|)$
    \item $\sigma(i+|\ub|+1)<...<\sigma(|\uc|+|\ub|-1)$.
\end{itemize}
We will occasionally use the shorthand $\circ_{\varphi}$ to refer to a partial composition with $i$ and $\sigma$ suppressed. We also similarly define partial compositions for symmetric and non-symmetric groupoid coloured operads as the contractions of the obvious trees associated to 
\begin{align*}
   P\binom{d}{\underline{c}}\otimes P\binom{c_i}{\underline{b}} \xrightarrow{\circ_{i}} P\binom{d}{\underline{c} \circ_i \underline{b}}
\end{align*}

\end{definition}

\subsection{Some Assumptions}

We now mirror \cite{ward2022massey}  (Assumption 2.28 and above) and make two assumptions for the rest of this paper. Firstly, we assume that $Aut(v)$ is finite for all $v\in Ob(\mathbb{V})$.
Secondly, we assume that the co-complete and closed symmetric monoidal category used by our operads is $dgVect_{\mathbb{K}}$, the category of differential graded vector spaces over $\mathbb{K}$, a field of characteristic $0$.
The finiteness of $Aut(v)$ not only allows for summation over automorphism groups as in Ward, but also enables an isomorphism between groupoid coloured shuffle operads and discrete coloured shuffle operads developed in \cref{The Groupoid Coloured Extension}. We restrict to characteristic $0$ for the same reasons as Ward, i.e. homotopy transfer theory can be done in characteristic $0$ without resolving automorphisms, but this is not the case in characteristic $p>0$. In particular, the characteristic $0$ assumption is needed for a groupoid coloured module to admit its homology as a deformation retract (\cref{homology as deformation retract}).

\subsection{Tree Monomials and Presentations of Operads}

Tree monomials correspond to the basis elements of free operads, and an understanding of their form is needed for developing Groebner bases for operads. Groupoid coloured tree monomials are equivalence classes of discrete coloured tree monomials under the action of the groupoid, so this work naturally extends Sections 2.2-2.5 of \cite{kharitonov2021grobner}.
\\\\
By the general theory of monads, each of our operads admits a forgetful functor which maps to the underling $\mathbb{V}_{\Sigma}$ (or $\mathbb{V}$)-module, and this forgetful functor admits a left adjoint. We shall denote these left adjoints $\overline{F_{\Sigma}},\overline{F}, \overline{F_{sh}}$ for the unital variants and $F_{\Sigma},F, F_{sh}$ for the non-unital variants. If we wish to highlight the underlying groupoid we will use a superscript e.g. $F^{\mathbb{V}}$. A \textbf{free groupoid coloured operad} is the image of one of these left adjoints applied to an appropriate groupoid coloured module.

\begin{definition}\label{Groupoid Coloured Tree Monomials}
Let $E$ be a $\mathbb{V}_{\Sigma}$ (or $\mathbb{V}$)-module, a \textbf{groupoid coloured tree monomial} of $F_{\Sigma}(E)$ is of the form $[ \bigotimes_{w \in Vt(T)} e_w ]_{Edges(T)}$  (see \cref{Trees decorated with groupoid coloured modules}), where $T$ is a $ob(\mathbb{V})$-tree of suitable type for the free operad, and each $e_w$ is in $E(w)$. 
We can view a groupoid coloured tree monomial as an equivalence class of (discrete) coloured tree monomials (in the sense of Section 2.1 of \cite{kharitonov2021grobner}) under the action of the groupoid on the edges of the tree. 
A tree monomial has several natural degrees (conserved under this action)
\begin{itemize}
    \item Its \textbf{arity} corresponds to its number of leaves.
    \item Its \textbf{weight} corresponds to the number of internal edges $+1$, or equivalently the number of constituent generators from $E$. 
\end{itemize}
We say an element of a groupoid coloured operad is \textbf{homogeneous} with respect to a degree if it is a sum of groupoid coloured tree monomials with non-zero coefficients, all with the same degree. Furthermore, we shall use $F_{\Sigma}(E)^{(k)}$ to denote the weight $k$ subset of $F_{\Sigma}(E)$.

\end{definition}

We close this section by defining quotient operads, and what we mean by a quadratic presentation.

\begin{definition} [2.5.2 \cite{ward2022massey}] \label{operadic ideals}
Given two $\mathbb{V}$ (or $\mathbb{V}_{\Sigma}$)-modules $A$ and $I$, we say $I$ is a \textbf{submodule} of $A$ if it is a subcategory of $A$. If $P$ is a groupoid coloured operad (of any type) and $I \subset P$ a submodule of appropriate type, we say $I$ is an \textbf{ideal} of $P$ if the image of any of the contracting maps $\eta_T$ of the operad (\cref{monadic groupoid coloured operads def}) applied to a tensor with a factor of $I$ lands back in $I$.
\end{definition}

\begin{definition}
If $I$ is an ideal of $P$, we may form a \textbf{quotient operad}
$P/I$ with the quotient $\mathbb{V}$-module and the inherited structure maps. We say an operad $P$ admits a \textbf{quadratic presentation } if $P\cong F_{\Sigma}(E)/R$ and $R\subset F_{\Sigma}(E)^{(2)}$.
\end{definition}

\subsection{Koszul Groupoid Coloured Operads}

We now introduce what it means for a groupoid coloured operad to be Koszul using the theory of Ward \cite{ward2022massey}, before providing an alternate characterisation through shuffle operads. We briefly recall the necessary definitions of groupoid coloured cooperads, and the bar/cobar complexes, before presenting four equivalent characterisations of when a groupoid coloured operad is Koszul (\cref{def:koszul}). Through the theory of shuffle operads we will show that a groupoid coloured symmetric operad $O$ is Koszul when its groupoid coloured shuffle operad $O^f$ is Koszul. This result will enable the Groebner basis machinery developed in \cref{Koszul Groupoid Coloured Operads via Groebner Bases}. 

\subsubsection{Background: Cooperads and Bar/Cobar Complexes}

We first recall some supporting theory of \cite{ward2022massey} regarding cooperads and bar/cobar complexes. All definitions in this section are presented for symmetric groupoid coloured operads, but they can be modified in obvious ways to obtain analogous definitions for non-symmetric and shuffle groupoid coloured operads.

\begin{definition}[2.29 \cite{ward2022massey}]
A symmetric groupoid coloured non-unital conilpotent cooperad is a comonoid in the category of $\mathbb{V}_{\Sigma}$-modules with comonoidal product given in Section 2.4.1 of \cite{ward2022massey}.
\end{definition}

By the general theory of comonads, a cooperad admits a forgetful functor $F_{\Sigma,c}$ which maps to the underling $\mathbb{V}_{\Sigma}$-module, and this forgetful functor admits a left adjoint. This allows us to define cofree cooperads, which leads to the following definition of cooperads with a quadratic presentation.

\begin{definition} [2.5.3 \cite{ward2022massey}]
Let $R$ be a homogeneous quadratic ideal of $F_{\Sigma,c}(E)$, then we define a \textbf{quadratic cooperad} $Q(E,R)$ as the union of all sub-cooperads of $Q\subset F_{\Sigma,c}(E)$ such that the composite $Q \hookrightarrow F_{\Sigma,c}(E)=F_{\Sigma}(E)\twoheadrightarrow F_{\Sigma,c}(E)^{(2)}/R $ vanishes.
\end{definition}

A particularly important quadratic cooperad is the Koszul dual cooperad defined as follows.

\begin{definition} [2.5.4 \cite{ward2022massey}] \label{Koszul Dual Cooperad}
Let $P\cong F_{\Sigma}(E)/R$ be a quadratic $\mathbb{V}_{\Sigma}$ operad, then  $P^{\antishriek}:=Q(sE,s^2R)$.
\end{definition}
where $s$ is the suspension operator in $dgVect$. The suspension operator and its inverse $s^{-1}$ are also used in the bar/cobar constructions. This familiar adjoint pair (see for instance \cite{loday2012algebraic} for the uncoloured case) extends naturally to groupoid coloured operads as follows.

\begin{definition}[2.6.1 \cite{ward2022massey}]
\label{Bar construction}
Let $P$ be an operad, the \textbf{bar} construction of $P$ is $B(P):=(F_c(s P), \partial + d_P)$, i.e. it is the cofree operad generated by the suspension of $P$ whose differential is induced by the operadic composition of $P$ ($\partial$) and the differential of $P$ ($d_P$).
\end{definition}

\begin{definition}[2.6.2 \cite{ward2022massey}] \label{Cobar Construction}
Let $Q$ be a conilpotent cooperad, the \textbf{cobar} construction of $Q$ is $\Omega(Q):=(F(s^{-1} Q), \partial + d_Q)$, i.e. it is the free operad generated by the desuspension of $Q$ whose differential is induced by the cooperadic composition of $Q$ ($\partial$) and the differential of $Q$ ($d_Q$).
\end{definition}

There is a natural grading on these complexes, which simplifies some theory for Koszulness.

\begin{definition}\label{syzygy degree}
There is a natural grading on the bar (and cobar) complex called the \textbf{syzygy degree} induced by the weight gradings of $P$ and $F_c(-)$. In particular, if $a \in F_c(sP)^{(k)} = F(sP)^{(k)}$ then the syzygy degree of $a$ is $w_P(a_1)+...+w_P(a_k)-k$.
\end{definition}

\subsubsection{Koszul Characterisations}

We now present four equivalent definitions for when a quadratic groupoid coloured operad is Koszul, before making some comments and proving their equivalence.

\begin{definition} \label{def:koszul}
A quadratic $\mathbb{V}_{\Sigma}$-operad $P$ is \textbf{Koszul} if any of the following equivalent definitions hold,
\begin{enumerate}
    \item the inclusion $\zeta_{P^{\antishriek}}:P^{\antishriek}\to B(P)$ is a quasi-isomorphism.
    \item the composition $\Omega(P^{\antishriek})\to \Omega(B(P)) \to P$ is a quasi-isomorphism.
    \item using the syzygy degree $H^{n}(B(P))=0$ for $n\geq 1$.
    \item the $\mathbb{V}_{sh}$-operad $P^f$ is Koszul.
\end{enumerate}
Analogous versions of the first three characterisations also hold for non-symmetric and shuffle groupoid coloured operads.
\end{definition}
The first characterisation is Definition 2.45 of \cite{ward2022massey}, and the remaining
characterisations are ordered by their appearance in the proof below. For the purpose of this paper, the second and the fourth characterisation are the most useful. When $P$ is Koszul we will denote $P_\infty := \Omega(P^{\antishriek})$. By the second characterisation, $P_\infty$ is a model for $P$, which by construction is quadratic (and hence minimal). The fourth characterisation is precisely the reason we have developed the theory of groupoid coloured shuffle operads, and this is what will enable the theory of Groebner bases in the subsequent section. We note that \cite{ward2022massey} has further characterisations of groupoid coloured operads, and many characterisations in the uncoloured case (see for instance $7.9.2$ of \cite{loday2012algebraic}) admit natural generalisations. 
\\\\
The equivalence of $(1)$ and $(2)$ in \cref{def:koszul} is a consequence of the bar-cobar adjunction, see Lemma 2.44 of \cite{ward2022massey}. Then, $(1)$ and $(3)$ are equivalent in an obvious extension of the argument of Proposition 7.3.1 of \cite{loday2012algebraic}, i.e. the inclusion $\zeta_{P^{\antishriek}}:P^{\antishriek}\to B(P)$ induces an isomorphism $P^{\antishriek} \cong H^0(B(P))$. Finally, $(3)$ and $(4)$ are equivalent through the following proposition.
\begin{proposition}[Proposition 1.4 \cite{dotsenko2013quillen}]\label{isomorphic bar complexes}
The symmetric bar complex of a groupoid coloured operad $P$ is isomorphic, as a shuffle $dg$ groupoid coloured cooperad, to the shuffle bar complex of $P^f$, i.e. $B(P)^f \cong B(P^f)$.
\end{proposition}
\begin{proof}
This proposition immediately extends to the groupoid coloured case as a consequence of $f$ being monoidal (\cref{Forgetful functor is monoidal}). 
\end{proof}
So,
\begin{align*}
     B(P)^f \cong B(P^f) \implies \forall n, H^n(B(P)^f) \cong H^n(B(P^f))
\end{align*}
which tells us that,
\begin{align*}
    H^n(B(P))=0  \text{ for } n\geq 1  \iff H^n(B(P)^f)=0 \text{ for } n\geq 1 \iff H^n(B(P^f))=0 \text{ for } n\geq 1
\end{align*}
So $P$ is Koszul by \cref{def:koszul}.(3), if and only if, $P^f$ is Koszul by \cref{def:koszul}.(3).

\section{Koszul Groupoid Coloured Operads via Groebner Bases} \label{Koszul Groupoid Coloured Operads via Groebner Bases}

The theory of Groebner bases was first developed for (discrete one-coloured) operads in \cite{dotsenko2010grobner}, extending upon the PBW bases of \cite{hoffbeck2010poincare}, as a useful tool for proving that specific operads are Koszul. In particular, Dotsenko and Khoroshkin proved that if an operad $O$ has a shuffle operad $O^f$ with a quadratic Groebner basis, then $O$ is Koszul. This powerful tool was extended to discrete coloured operads in \cite{kharitonov2021grobner}, and is currently in the process of being extended to dioperads by Khoroshkin. In this section, we develop the theory of Groebner bases for groupoid coloured operads by proving the following theorem.
\begin{theorem*}[\ref{Koszul if corresponding shuffle operad has QGB}]
Let $P$ be a $\mathbb{V}$-coloured operad such that the associated $ob(\mathbb{V})$-coloured shuffle operad $(P^f)_*$ admits a quadratic Groebner basis, then $P$ is Koszul.
\end{theorem*}

We briefly revisit the theory of Groebner bases for discrete coloured operads (\cite{bremner2016algebraic}, \cite{dotsenko2010grobner}, \cite{hoffbeck2010poincare}, and \cite{kharitonov2021grobner}) before extending the theory to groupoid coloured operads in \cref{The Groupoid Coloured Extension}. We then close this section by explicitly establishing a well known equivalence between confluent terminating rewriting systems and Groebner bases. This provides alternate methods for proving that a discrete coloured operad admits a Groebner basis. Later, we use techniques developed in this section to prove that the operads governing wheeled props (\cref{proof The Operad Governing Wheeled Props is Koszul}) and props (\cref{proof The Operad Governing Props is Koszul}) are Koszul.

\subsection{Divisibility, Order and Groebner Bases for Discrete Coloured Shuffle Operads}

We start by recalling some necessary definitions and results.

\begin{definition}
Given a discrete coloured shuffle operad, a partial ordering $\leq$ of its underlying tree monomials is said to be \textbf{admissible} if
\begin{align*}
    (\alpha\leq \alpha' \text{ and } \beta \leq \beta') \implies \alpha \circ_\varphi \beta \leq \alpha' \circ_\varphi \beta'
\end{align*}
where $\alpha,\alpha'$ and $\beta,\beta'$ are pairs of arity homogeneous tree monomials and $\circ_{\varphi}$ is a shuffle composition. If the order is also total (i.e. all tree monomials are comparable) then we will call it a \textbf{total admissible order}.
\end{definition}

For the rest of this section, we assume all tree monomials are ordered by a total admissible order. Path lexicographic orders are an example of a total admissible order, and we will construct two such orders in this paper (\ref{Ordering the Object Coloured Tree Monomials for Wheeled Props}, \ref{Ordering the Object Coloured Tree Monomials for Props}), but see \cite{dotsenko2010grobner} for many additional examples. This total order on tree monomials allows us to identify the largest term with non-zero coefficients of any element of a given shuffle operad. For such an element $f$, we will denote the largest term by $lt(f)$, and it's coefficient by $c_f$.

\begin{definition}\label{divisibility}
Given two tree monomials $\alpha$ and $\beta$, we say that $\alpha$ is \textbf{divisible} by $\beta$ if there exists a subtree $T'$ of $\alpha$ whose corresponding tree monomial is $\beta$. As $\alpha$ is divisible by $\beta$, a composite of elementary shuffle compositions and generators may be used to form $\alpha$ from $\beta$ (Proposition 6 of \cite{dotsenko2010grobner}). This yields an operator in tree monomials $m_{\alpha,\beta}$ such that $m_{\alpha,\beta}(\beta) = \alpha$.
\end{definition}
Note that we can apply the operator $m_{\alpha,\beta}$ to other tree monomials with the same shape as $\beta$.
Furthermore, as we have a total admissible order $\leq$, it must be the case that if $\gamma<\beta$ then $m_{\alpha,\beta}(\gamma)< m_{\alpha,\beta}(\beta)=\alpha$.
Several examples of divisibility, and the induced operators, are provided in \cref{gb assoc example}. With the preceding definitions, we can now define Groebner bases for discrete coloured shuffle operads.
\begin{definition}\label{def Groebner basis} Let $F_{sh}(E)$ be a free discrete coloured shuffle operad, $I$ an operadic ideal of $F_{sh}(E)$, and $G$ a system of arity homogeneous generators for $I$. We say $G$ is \textbf{Groebner basis} for $I$ if, for all ideal elements $f\in I$, the leading term of $f$ is divisible by the leading term of an element of $G$. We say $G$ is \textbf{quadratic} if $G\subset F_{sh}(E)^{(2)}$.
\end{definition}

This is certainly a definition, but we need a practical way to characterise when we have a Groebner basis. To this end, we first discuss reductions and $S$-polynomials.

\begin{definition} \label{reduction}
If $f,g$ are two arity homogeneous elements of $F_{sh}(E)$ such that $lt(f)$ is divisible by $lt(g)$, then
\begin{align*}
    r_g(f) := f - \frac{c_f}{c_g}m_{lt(f),lt(g)}(g)
\end{align*}
is called the \textbf{reduction of $f$ modulo $g$}. 
\end{definition}
The total admissible order on tree monomials implies that $lt(r_g(f))<lt(f)$, and in \cref{Rewriting Systems and Groebner Bases}, we will interpret a reduction $f\to r_g(f)$ as a rewrite of $f$ to a smaller element. These reductions may be chained together to provide a unique normal form for every element of the discrete coloured shuffle operad through the following proposition.

\begin{proposition}[Proposition 7 \cite{dotsenko2010grobner}]\label{Unique normal form proposition}
If $G$ is a Groebner basis for $I$ and $f\in F_{sh}(E)$ then there exists a unique element $\bar{f} \in F_{sh}(E)$ such that 
\begin{enumerate}[label=(\roman*)]
    \item $f-\bar{f}\in I$
    \item $\bar{f}$ is a linear combination of tree monomials which have no non-trivial reductions modulo $G$ (i.e. for all $g$ in $G$, every term of $\bar{f}$ with a non-zero coefficient is not divisible by $lt(g)$). We denote this $f \equiv \bar{f}\ mod\ G$.
\end{enumerate}
We say that $\bar{f}$ is the \textbf{unique normal form} of $f$.
\end{proposition}

We now slightly extend the notion of a reduction by defining a $S$-polynomial.

\begin{definition} 
We say that two tree monomials $\alpha,\beta$ have a \textbf{common multiple} if there exists a discrete coloured tree monomial $\gamma$ such that $\gamma$ is divisible by both $\alpha$ and $\beta$. We say that a common multiple is \textbf{small} if the corresponding subtrees of $\alpha$ and $\beta$ in $\gamma$ overlap.
\end{definition}

\begin{definition}
Let $f$ and $g$ be two arity homogeneous elements of $F_{sh}(E)$ whose leading terms admit a small common multiple $\gamma$. The \textbf{S-polynomial} corresponding to $\gamma$ is defined by
\begin{align*}
    s_{\gamma}(f,g):= m_{\gamma,lt(f))}(f)- \frac{c_f}{c_g} m_{\gamma,lt(g))}(g)
\end{align*}
\end{definition}

We can now give the following equivalent characterisations of Groebner bases for operads.

\begin{proposition}[Theorem 1 \cite{dotsenko2010grobner}] \label{Equivalent defs GB}
The following are equivalent.
\begin{enumerate}
    \item $G$ is a Groebner basis for $I$
    \item for all $f \in I$, $f\equiv 0\ mod\ G$
    \item for all pairs of elements from $ G$, all their $S$-polynomials (if defined) are congruent to $0\ mod\ G$.
\end{enumerate}
\end{proposition}

In particular, the third characterisation means we can check we have a Groebner basis by a generalised version of Buchberger's algorithm \cite{dotsenko2010grobner}. This means we can apply this algorithm to check that a discrete coloured operad is Koszul through the following result.
\begin{proposition}[Theorem 3.12 \cite{kharitonov2021grobner}] 
A discrete coloured shuffle operad $P$ with a quadratic Groebner basis is Koszul.
\end{proposition}

\subsection{The Groupoid Coloured Extension} \label{The Groupoid Coloured Extension}

We now extend these techniques to the groupoid coloured case.
We point out that the constructions of this section will be clarified in \cref{ex:epimorphism example}.

\subsubsection{Forgetting the Action of the Groupoid}

We first show that the Koszulness of a $\mathbb{V}$-coloured operad $P$ is characterised by the Koszulness of an $Ob(\mathbb{V})$-coloured shuffle operad.
We note that the arguments of this section also apply to symmetric and non-symmetric operads with the obvious modifications to modules and operadic compositions.
\\\\
Recall from \cref{Monoidal Product Lemma}, that the forgetful functor from $\mathbb{V}_{\Sigma}$-modules to $\mathbb{V}$-modules is monoidal with respect to the symmetric operadic and shuffle operadic monoidal compositions (\cref{monoidal groupoid coloured operads def}).
One consequence of this is that the bar complex of a groupoid coloured operad $P$, and its corresponding shuffle operad $P^f$ are isomorphic as groupoid coloured shuffle cooperads (\cref{isomorphic bar complexes}),
\begin{align*}
    B(P^f)\cong B(P)^f.
\end{align*}
This result allowed us to infer that $P$ is Koszul, if, and only if, $P^f$ is Koszul (\cref{def:koszul}).
Thus, instead of studying the Koszulness of $\mathbb{V}$-coloured symmetric operads, it is sufficient to study the Koszulness of $\mathbb{V}$-coloured shuffle operads.
\\\\
There also exists a forgetful functor $-^{f_{\mathbb{V}}}$ from $\mathbb{V}$-modules to $ob(\mathbb{V})$-modules, which forgets the action of the groupoid. However, $-^{f_{\mathbb{V}}}$ is not monoidal.
Despite this failure, we will show in \cref{epimorphism of bar constructions}, that there exists an epimorphism $B(P^{f_{\mathbb{V}}}) \twoheadrightarrow B(P)$ of groupoid coloured shuffle cooperads.
This epimorphism lets us infer one of the two directions.
\begin{proposition} \label{groupoid coloured shuffle operad is Koszul if its forgetful image is Koszul}
A $\mathbb{V}$-coloured shuffle operad $P$ is Koszul if $P^{f_{\mathbb{V}}}$ is Koszul.
\end{proposition}
\begin{proof}
The epimorphism $B(P^{f_{\mathbb{V}}}) \twoheadrightarrow B(P)$ implies (using syzygy degree, \cref{syzygy degree}) that,
\begin{align*}
    H^n(B(P^{f_{\mathbb{V}}}))=0 \text{ for }n\geq 1 \quad \implies \quad H^n(B(P))=0 \text{ for }n\geq 1.
\end{align*}
Hence by \cref{def:koszul}, if $P^{f_{\mathbb{V}}}$ is Koszul then $P$ is Koszul.
\end{proof}

The following epimorphism of groupoid coloured shuffle operads will not only let us construct $B(P^{f_{\mathbb{V}}}) \twoheadrightarrow B(P)$, but will eventually lead to a simpler presentation of $P^{f_{\mathbb{V}}}$ in the next subsection.
Recall that a morphism of groupoid coloured operads is a morphism of groupoid coloured modules, \cref{def:morphism of groupoid coloured bimodules}, which is compatible with the operadic structure maps.

\begin{definition} \label{The quotient map}
Let $E$ be a $\mathbb{V}$-coloured module.
We define an epimorphism of groupoid coloured operads,
\begin{align*}
    [-]:F_{sh}^{ob(\mathbb{V})}(E^{f_{\mathbb{V}}})\twoheadrightarrow  F_{sh}^{\mathbb{V}}(E).
\end{align*}
The underlying map of groupoids is the inclusion $[-]_0:ob(\mathbb{V}) \hookrightarrow \mathbb{V}$.
The map of groupoid coloured modules is defined on the basis elements of $F_{sh}^{ob(\mathbb{V})}(E^{f_{\mathbb{V}}})$, and sends each $ob(\mathbb{V})$-coloured tree monomial $T$ to a $\mathbb{V}$-coloured tree monomial by adding in the action of $\mathbb{V}$ on the internal edges of $T$ (\cref{Groupoid Coloured Tree Monomials}),
\begin{align*}
    [T]:= [ \bigotimes_{w \in Vt(T)} e_w ]_{Edges(T)}.
\end{align*}

\end{definition}

\begin{lemma}
The quotient map $[-]$ is indeed an epimorphism of groupoid coloured operads.
\end{lemma}
\begin{proof}
We first observe that $[-]$ (trivially) respects the trivial action of the groupoid $ob(\mathbb{V})$.
Next, we observe that $[\alpha \circ_\varphi \beta] = [\alpha] \circ_\varphi [\beta]$, as both sides of the equation have the same action of $\mathbb{V}$ on the internal edge specified by the partial composition $\circ_\varphi$.
Finally, $[-]$ must be an epimorphism, as the equivalence class corresponding to any groupoid coloured tree monomial in $F_{sh}^{\mathbb{V}}(E)$ is represented by at least one non-groupoid coloured tree monomial in $F_{sh}^{ob(\mathbb{V})}(E^{f_{\mathbb{V}}})$.
(See \cref{lem:unique minimal representative of a groupoid coloured tree monomial} for one method of producing a particular representative.)
\end{proof}

\begin{proposition}\label{epimorphism of bar constructions}
The quotient map $[-]$ induces an epimorphism $B(P^{f_{\mathbb{V}}}) \twoheadrightarrow B(P)$ of groupoid coloured shuffle cooperads. 
\end{proposition}
\begin{proof}
As in the prior definition, the underlying map of groupoids is the inclusion $ob(\mathbb{V}) \hookrightarrow \mathbb{V}$.
Now, recall from \cref{Bar construction}, that given a groupoid coloured shuffle operad $P$, its bar construction is $B(P) = (F_c(s P), \partial+ d_P)$, where the differential $d_P$ is the differential induced by the differential of $P$, and $\partial$ is the differential induced by the operadic composition of $P$.
Furthermore, the underlying groupoid coloured modules of the cofree cooperad $F_c(s P)$ and $F(sP)$ are equal.
Hence,
\begin{align*}
    F_c^{ob(\mathbb{V})}(s P^{f_{\mathbb{V}}}) = F^{ob(\mathbb{V})}(sP^{f_{\mathbb{V}}}) \twoheadrightarrow F^{\mathbb{V}}(sP) = F_c^{\mathbb{V}}(sP).
\end{align*}
This composite, which we also denote $[-]$ in a minor abuse of notation, respects the cooperadic structure, i.e. if we let $\triangle^{ob(\mathbb{V})}$ and $\triangle^{\mathbb{V}}$ denote the partial cooperadic products for $F_c^{ob(\mathbb{V})}(s P^{f_{\mathbb{V}}})$ and $F_c^{\mathbb{V}}(s P)$ respectively, then $([-],[-])\circ \triangle^{ob(\mathbb{V})} = \triangle^{\mathbb{V}} \circ [-]$.
This follows as, for $\alpha \in F_c^{ob(\mathbb{V})}(s P^{f_{\mathbb{V}}})$, we observe that $\triangle^{ob(\mathbb{V})}(\alpha) \in \triangle^{\mathbb{V}} ([\alpha])$.
\\\\
Thus, to conclude the proof, we need only show that $[-]$ commutes with the respective differentials.
This is immediate for the differentials induced by the differential of $P$.
For the $\partial$ (operadic) differentials, we recall from Section 2.6.1 of \cite{ward2022massey} that for a groupoid coloured operad $Q$, the differential $\partial$ of $B(Q)$ is the unique extension of the following cooperad map.
\begin{center}
\begin{tikzcd}
F_c(s Q) \arrow[r] &
F_c(s Q)^{(2)} = 
F(s Q)^{(2)} \cong 
s^2 F(Q)^{(2)} \arrow[r] &
s^2 Q \arrow[r] &
s Q
\end{tikzcd}
\end{center}
Where the first map is projection, the second is the operadic structure map, and the final is a shift in degree. 
Thus if we let $\partial_{ob(\mathbb{V})}$ and  $\partial_{\mathbb{V}}$ denote the $\partial$ differentials of $B(P^{f_{\mathbb{V}}})$ and $B(P)$ respectively.
Then the equation $[-]\circ \partial_{ob(\mathbb{V})} = \partial_{\mathbb{V}} \circ [-]$, is witnessed by the commutativity of their inducing cooperad maps,

\begin{center}
\begin{tikzcd}
F_c(s P^{f_{\mathbb{V}}}) \arrow[r] \arrow[d, twoheadrightarrow, "{[-]}"] &
F_c(s P^{f_{\mathbb{V}}})^{(2)} = F(s P^{f_{\mathbb{V}}})^{(2)} \cong 
s^2 F(P^{f_{\mathbb{V}}})^{(2)} \arrow[r] &
s^2 P^{f_{\mathbb{V}}} \arrow[r] &
s P^{f_{\mathbb{V}}} \arrow[d, twoheadrightarrow, "{[-]}"]\\
F_c(s P) \arrow[r] &
F_c(s P)^{(2)} = 
F(s P)^{(2)} \cong 
s^2 F(P)^{(2)} \arrow[r] &
s^2 P \arrow[r] &
s P
\end{tikzcd}
\end{center}
Thus $[-]$ is a morphism of groupoid coloured modules that also respects the shuffle cooperadic structure, and hence is a morphism of groupoid coloured shuffle cooperads.
\end{proof}

\subsubsection{A Simpler Presentation}

Given a quadratic $\mathbb{V}$-coloured operad $P$ we now seek a simple quadratic presentation for $P^{f_{\mathbb{V}}}$.
Given we have an epimorphism of groupoid coloured operads $[-]:F_{sh}^{ob(\mathbb{V})}(E^{f_{\mathbb{V}}})\twoheadrightarrow  F_{sh}^{\mathbb{V}}(E)$ (\cref{The quotient map}).
An analogue of the fundamental theorem of homomorphisms implies the composite $[-]^{f^{\mathbb{V}}}$ is an isomorphism of $ob(\mathbb{V})$-coloured shuffle operads,
\begin{align}\label{eq:kernel iso}
    F_{sh}^{ob(\mathbb{V})}(E^{f_{\mathbb{V}}}) / \langle  Ker([-]) \rangle \cong  F_{sh}^{\mathbb{V}}(E)^{f_{\mathbb{V}}}.
\end{align}
Thus, we seek to characterise $Ker([-])$.
Firstly, we show the kernel admits a quadratic presentation.
\begin{definition}
Let $[-]_2$ be the restriction to weight two elements (\cref{Groupoid Coloured Tree Monomials}), $[-]_2:= [-]|_{(F_{sh}^{ob(\mathbb{V})})^{(2)}}$.
\end{definition}

By direct inspection,
\begin{align*}
    Ker([-]_2) = \{ \alpha \circ_\phi \beta= \alpha' \circ_{\phi'} \beta' : \alpha,\beta,\alpha',\beta' \in E^{f_{\mathbb{V}}} \text{ such that } [\alpha \circ_\phi \beta] = [\alpha' \circ_{\phi'} \beta'] \}.
\end{align*}

\begin{lemma}\label{lem:kernel in weight two}
\begin{align*}
    \langle Ker([-]_2) \rangle = \langle Ker([-]) \rangle
\end{align*}
\end{lemma}
\begin{proof}
Clearly $\langle Ker([-]_2) \rangle \subseteq \langle Ker([-]) \rangle$.
Conversely, by \cref{Trees decorated with groupoid coloured modules}, the action of $[-]_{Edges(T)}$ is applied to each internal edge via an unordered tensor product.
Thus, as $Ker([-]_2)$ encodes the action of a single internal edge, sequential applications of relations in $Ker([-]_2)$ encode the action across all internal edges.
\end{proof}

We now seek an explicit presentation of $Ker([-]_2)$.

\begin{lemma}\label{lem:unique minimal representative of a groupoid coloured tree monomial}
Let $P$ be a $\mathbb{V}$-coloured operad with a total order on its underlying $ob(\mathbb{V})$-coloured tree monomials.
Every $\mathbb{V}$-coloured tree monomial $T$ admits a unique minimal representative, which we denote $T_*$.
\end{lemma}

\begin{proof}
As $Aut(v)$ is finite for all $v\in ob(\mathbb{V})$, every equivalence class $T:=[ \bigotimes_{w \in Vt(T)} e_w ]_{Edges(T)}$ is finite.
Therefore, as a total order on a finite set is a well-order, a unique minimal element exists.
\end{proof}

\begin{remark}
It is possible to define a total order on groupoid coloured tree monomials using \cref{lem:unique minimal representative of a groupoid coloured tree monomial}.
In particular, given two groupoid coloured tree monomials $\alpha,\beta$ we define $\alpha\leq \beta$ if, and only if, $\alpha_*\leq \beta_*$.
One could then define an appropriate notion of an admissible order for groupoid coloured shuffle operads such that this order is admissible.
From here, one could directly redevelop the entire theory of Groebner bases for groupoid coloured shuffle operads. 
However, the entire point of this section is to circumvent the need to do this.
\end{remark}

\begin{corollary}
This defines an injective map $-_*$ from $\mathbb{V}$-coloured operads to $ob(\mathbb{V})$-coloured operads.
\end{corollary}

\begin{proof}
Given a $\mathbb{V}$-coloured operad $P$, every $x\in P$ is of the form $x = \sum_i k_i T_i $, where each $k_i \in \mathbb{K}$ and $T_i$ is $\mathbb{V}$-coloured tree monomial.
Thus $x_*$ is given via linearisation, $x_* := \sum_i k_i {(T_i)}_*$.
The map is injective as the equivalence class of any groupoid coloured tree monomial can be reconstructed from any element.
\end{proof}

\begin{remark}
In general, $-_*$ will not be a morphism of groupoid coloured operads, as every morphism of the groupoid $\mathbb{V}$ is mapped to an identity morphism.
Thus, if $\alpha$ is an element of a $\mathbb{V}$-coloured operad which admits a non-trivial action of the groupoid $\mathbb{V}$, i.e. $\alpha \neq \alpha \cdot v$,
then $(\alpha \cdot v)_*\neq \alpha_* \cdot (v_*) = \alpha_* \cdot id = \alpha_*$.
\end{remark}

\begin{lemma}\label{lem: E2 presentation of degreee two kernel}
If $F_{sh}^{ob(\mathbb{V})}(E^{f_{\mathbb{V}}})$ is totally ordered, then $\langle Ker([-]_2) \rangle = \langle E_{*}^2 \rangle $  where
\begin{align*}
    E_{*}^2 := \{ \alpha \circ_\varphi \beta - [\alpha \circ_\varphi \beta]_* :\alpha, \beta \in E^{f_{\mathbb{V}}} \}.
\end{align*}
\end{lemma}
\begin{proof}
By their respective definitions, it is clear that $\langle E_{*}^2 \rangle\subseteq \langle Ker([-]_2) \rangle$.
In the other direction, if  $ \alpha \circ_{\varphi} \beta - \alpha' \circ_{\varphi'} \beta' \in Ker([-]_2)$, then it must be the case that $\alpha \circ_{\varphi} \beta - [\alpha \circ_{\varphi} \beta]_*, \alpha' \circ_{\varphi'} \beta' - [\alpha' \circ_{\varphi'} \beta']_* \in E_*^2$ with $[\alpha \circ_{\varphi} \beta]_* = [\alpha' \circ_{\varphi'} \beta']_*$, and hence $\langle Ker([-]_2) \rangle \subseteq \langle E_{*}^2 \rangle $.

\end{proof}

Thus, from 
\cref{eq:kernel iso}, \cref{lem:kernel in weight two} and \cref{lem: E2 presentation of degreee two kernel}, we obtain the following simple presentation for $P^{f_{\mathbb{V}}}$.

\begin{corollary}\label{An alternate presentation after forgetting the groupoid}
Let $P$ be a free $\mathbb{V}$-coloured shuffle operad $P=F_{sh}^{\mathbb{V}}(E)$.
Then $P^{f_{\mathbb{V}}}$ is isomorphic as a $ob(\mathbb{V})$-coloured shuffle operad to, 
\begin{align*}
     P_*:=  F_{sh}^{ob(\mathbb{V})}(E^{f_{\mathbb{V}}}) / \langle E_*^2 \rangle.
\end{align*}
Let $P$ be a $\mathbb{V}$-coloured shuffle operad with presentation $P=F_{sh}^{\mathbb{V}}(E) / \langle R \rangle$.
Then $P^{f_{\mathbb{V}}}$ is isomorphic as a $ob(\mathbb{V})$-coloured shuffle operad to, 
\begin{align*}
    P_*:= F_{sh}^{ob(\mathbb{V})}(E^{f_{\mathbb{V}}}) / \langle E_*^2 \sqcup R_* \rangle, \text{ where } R_*:= \{ [r]_* : r\in R\}.
\end{align*}

\end{corollary}

This corollary tells us we can translate the actions of the groupoid into an explicitly computable set of quadratic relations $E_*^2$ (other non-quadratic encodings of the action of the groupoid are possible such as the invertible unary maps of \cite{batanin2018regular}).

\begin{theorem} \label{Koszul if corresponding shuffle operad has QGB}
Let $P$ be a $\mathbb{V}$-coloured operad such that the associated $ob(\mathbb{V})$-coloured shuffle operad $(P^f)_*$ admits a quadratic Groebner basis, then $P$ is Koszul.
\end{theorem}

\begin{proof}
If the $ob(\mathbb{V})$-coloured operad $(P^f)_*$ admits a quadratic Groebner basis, then it is Koszul by Theorem 3.12 of \cite{kharitonov2021grobner}. 
This implies that $P^f$ is Koszul, through \cref{groupoid coloured shuffle operad is Koszul if its forgetful image is Koszul} and the isomorphism of $(P^f)_*$ with $(P^f)^{f_{\mathbb{V}}}$ given by \cref{An alternate presentation after forgetting the groupoid}.
This implies that $P$ is Koszul by \cref{def:koszul}.(4).
\end{proof}

We will later use \cref{Koszul if corresponding shuffle operad has QGB} to prove the operads governing wheeled props and props are Koszul (Sections \ref{proof The Operad Governing Wheeled Props is Koszul}, and \ref{proof The Operad Governing Props is Koszul}).
For now, we provide a more basic example, clarifying the constructions of this section.

\begin{example} \label{ex:epimorphism example}
Let $\mathbb{V}$ be the groupoid with three objects $a,b,c$ and a single non-identity isomorphism $f:b\to c$, and its inverse $f^{-1}:c\to b$.
Let $N$ be the non-symmetric $ob(\mathbb{V})$-coloured module spanned by a single binary operation
\begin{align*}
    N = N\binom{c}{b,c} = \langle
    \begin{tikzpicture}[node distance = 0.5cm] 
    \node[v] (1) {${}$}; 
    \node[c] (0) [above of=1] {${}$};
    \node[b] (2) [below left of=1] {${}$};
    \node[c] (3) [below right of=1] {${}$};
    \draw (1) to (0);
    \draw (2) to (1);
    \draw (3) to (1);
    \end{tikzpicture}    
    \rangle 
\end{align*}
Where we use a white/grey circle to denote a colouring by the object $b$/$c$ respectively.
Given these identifications, we represent the morphisms of $\mathbb{V}$ graphically,
the isomorphism $f:b \to c$ is denoted  
\begin{tikzpicture}
\node[b] (2) [] {${}$};
\node[c] (3) [right of=2] {${}$};
\draw [->] (2) to (3);
\end{tikzpicture},
and its inverse
$f^{-1}:c \to b$ is denoted  
\begin{tikzpicture}
\node[c] (2) [] {${}$};
\node[b] (3) [right of=2] {${}$};
\draw [->] (2) to (3);
\end{tikzpicture}.
Let $E$ be the free non-symmetric $\mathbb{V}$-module generated by $N$.
In particular, $E$ contains $2^3$ elements corresponding to the $\mathbb{V}$-action on the generator's three flags,
\begin{align*}
    \raisebox{0.6cm}{$E = \langle \ $}
    \begin{tikzpicture}[node distance = 0.5cm] 
    \node[v] (1) {${}$}; 
    \node[c] (0) [above of=1] {${}$};
    \node[b] (2) [below left of=1] {${}$};
    \node[c] (3) [below right of=1] {${}$};
    \node[flag] (o) [above of=0] {${}$};
    \node[flag] (i1) [below of=2] {${}$};
    \node[flag] (i2) [below of=3] {${}$};
    \draw (1) to (0);
    \draw (2) to (1);
    \draw (3) to (1);
    \draw [->, white] (0) to (o);
    \draw [->, white] (i1) to (2);
    \draw [->, white] (i2) to (3);
    \end{tikzpicture}
    \raisebox{0.6cm}{$\ ,\ $}
    \begin{tikzpicture}[node distance = 0.5cm] 
    \node[v] (1) {${}$}; 
    \node[c] (0) [above of=1] {${}$};
    \node[b] (2) [below left of=1] {${}$};
    \node[c] (3) [below right of=1] {${}$};
    \node[b] (o) [above of=0] {${}$};
    \node[flag] (i1) [below of=2] {${}$};
    \node[flag] (i2) [below of=3] {${}$};
    \draw (1) to (0);
    \draw (2) to (1);
    \draw (3) to (1);
    \draw [->] (0) to (o);
    \draw [->, white] (i1) to (2);
    \draw [->, white] (i2) to (3);
    \end{tikzpicture}
    \raisebox{0.6cm}{$\ ,\ $}
    \begin{tikzpicture}[node distance = 0.5cm] 
    \node[v] (1) {${}$}; 
    \node[c] (0) [above of=1] {${}$};
    \node[b] (2) [below left of=1] {${}$};
    \node[c] (3) [below right of=1] {${}$};
    \node[flag] (o) [above of=0] {${}$};
    \node[c] (i1) [below of=2] {${}$};
    \node[flag] (i2) [below of=3] {${}$};
    \draw (1) to (0);
    \draw (2) to (1);
    \draw (3) to (1);
    \draw [->, white] (0) to (o);
    \draw [->] (i1) to (2);
    \draw [->, white] (i2) to (3);
    \end{tikzpicture}
    \raisebox{0.6cm}{$\ ,\ $}
    \begin{tikzpicture}[node distance = 0.5cm] 
    \node[v] (1) {${}$}; 
    \node[c] (0) [above of=1] {${}$};
    \node[b] (2) [below left of=1] {${}$};
    \node[c] (3) [below right of=1] {${}$};
    \node[flag] (o) [above of=0] {${}$};
    \node[flag] (i1) [below of=2] {${}$};
    \node[b] (i2) [below of=3] {${}$};
    \draw (1) to (0);
    \draw (2) to (1);
    \draw (3) to (1);
    \draw [->, white] (0) to (o);
    \draw [->, white] (i1) to (2);
    \draw [->] (i2) to (3);
    \end{tikzpicture}
    \raisebox{0.6cm}{$\ ,\ $}
    \begin{tikzpicture}[node distance = 0.5cm] 
    \node[v] (1) {${}$}; 
    \node[c] (0) [above of=1] {${}$};
    \node[b] (2) [below left of=1] {${}$};
    \node[c] (3) [below right of=1] {${}$};
    \node[b] (o) [above of=0] {${}$};
    \node[c] (i1) [below of=2] {${}$};
    \node[flag] (i2) [below of=3] {${}$};
    \draw (1) to (0);
    \draw (2) to (1);
    \draw (3) to (1);
    \draw [->] (0) to (o);
    \draw [->] (i1) to (2);
    \draw [->, white] (i2) to (3);
    \end{tikzpicture}
    \raisebox{0.6cm}{$\ ,\ $}
    \begin{tikzpicture}[node distance = 0.5cm] 
    \node[v] (1) {${}$}; 
    \node[c] (0) [above of=1] {${}$};
    \node[b] (2) [below left of=1] {${}$};
    \node[c] (3) [below right of=1] {${}$};
    \node[b] (o) [above of=0] {${}$};
    \node[flag] (i1) [below of=2] {${}$};
    \node[b] (i2) [below of=3] {${}$};
    \draw (1) to (0);
    \draw (2) to (1);
    \draw (3) to (1);
    \draw [->] (0) to (o);
    \draw [->, white] (i1) to (2);
    \draw [->] (i2) to (3);
    \end{tikzpicture}
    \raisebox{0.6cm}{$\ ,\ $}
    \begin{tikzpicture}[node distance = 0.5cm] 
    \node[v] (1) {${}$}; 
    \node[c] (0) [above of=1] {${}$};
    \node[b] (2) [below left of=1] {${}$};
    \node[c] (3) [below right of=1] {${}$};
    \node[flag] (o) [above of=0] {${}$};
    \node[c] (i1) [below of=2] {${}$};
    \node[b] (i2) [below of=3] {${}$};
    \draw (1) to (0);
    \draw (2) to (1);
    \draw (3) to (1);
    \draw [->, white] (0) to (o);
    \draw [->] (i1) to (2);
    \draw [->] (i2) to (3);
    \end{tikzpicture}
    \raisebox{0.6cm}{$\ ,\ $}
    \begin{tikzpicture}[node distance = 0.5cm] 
    \node[v] (1) {${}$}; 
    \node[c] (0) [above of=1] {${}$};
    \node[b] (2) [below left of=1] {${}$};
    \node[c] (3) [below right of=1] {${}$};
    \node[b] (o) [above of=0] {${}$};
    \node[c] (i1) [below of=2] {${}$};
    \node[b] (i2) [below of=3] {${}$};
    \draw (1) to (0);
    \draw (2) to (1);
    \draw (3) to (1);
    \draw [->] (0) to (o);
    \draw [->] (i1) to (2);
    \draw [->] (i2) to (3);
    \end{tikzpicture}
    \raisebox{0.6cm}{$\ \rangle$}
\end{align*}
Note that the object $a \in ob(\mathbb{V})$ does not appear in $E$, as $a$ is disconnected from $b$ and $c$ in $\mathbb{V}$.
The $ob(\mathbb{V})$-coloured module $E^{f_{\mathbb{V}}}$ contains the same $2^3$ elements, but forgets the underlying action of $\mathbb{V}$.
We now consider the $ob(\mathbb{V})$-coloured shuffle operad $F_{sh}^{ob(\mathbb{V})}(E^{f_{\mathbb{V}}})$.
In particular, we assume it admits a total admissible order and compute $E_*^2$.
\begin{align*}
    \raisebox{0.6cm}{$E_*^2 =  \langle \ $}
    &
    \begin{tikzpicture}[node distance = 0.5cm] 
    \node[v] (1) {${}$}; 
    \node[c] (0) [above of=1] {${}$};
    \node[b] (2) [below left of=1] {${}$};
    \node[c] (3) [below right of=1] {${}$};
    \node[flag] (o) [above of=0] {${}$};
    \node[flag] (i1) [below of=2] {${}$};
    \node[flag] (i2) [below of=3] {${}$};
    \draw (1) to (0);
    \draw (2) to (1);
    \draw (3) to (1);
    \draw [->, white] (0) to (o);
    \draw [->, white] (i1) to (2);
    \draw [->, white] (i2) to (3);
    \end{tikzpicture}
    \raisebox{0.6cm}{$\ \circ_2 \ $}
    \begin{tikzpicture}[node distance = 0.5cm] 
    \node[v] (1) {${}$}; 
    \node[c] (0) [above of=1] {${}$};
    \node[b] (2) [below left of=1] {${}$};
    \node[c] (3) [below right of=1] {${}$};
    \node[flag] (o) [above of=0] {${}$};
    \node[flag] (i1) [below of=2] {${}$};
    \node[flag] (i2) [below of=3] {${}$};
    \draw (1) to (0);
    \draw (2) to (1);
    \draw (3) to (1);
    \draw [->, white] (0) to (o);
    \draw [->, white] (i1) to (2);
    \draw [->, white] (i2) to (3);
    \end{tikzpicture}
    \raisebox{0.6cm}{$\ = \ $}
    \begin{tikzpicture}[node distance = 0.5cm] 
    \node[v] (1) {${}$}; 
    \node[c] (0) [above of=1] {${}$};
    \node[b] (2) [below left of=1] {${}$};
    \node[c] (3) [below right of=1] {${}$};
    \node[flag] (o) [above of=0] {${}$};
    \node[flag] (i1) [below of=2] {${}$};
    \node[b] (i2) [below of=3] {${}$};
    \draw (1) to (0);
    \draw (2) to (1);
    \draw (3) to (1);
    \draw [->, white] (0) to (o);
    \draw [->, white] (i1) to (2);
    \draw [->] (i2) to (3);
    \end{tikzpicture}
    \raisebox{0.6cm}{$\ \circ_2 \ $}
    \begin{tikzpicture}[node distance = 0.5cm] 
    \node[v] (1) {${}$}; 
    \node[c] (0) [above of=1] {${}$};
    \node[b] (2) [below left of=1] {${}$};
    \node[c] (3) [below right of=1] {${}$};
    \node[b] (o) [above of=0] {${}$};
    \node[flag] (i1) [below of=2] {${}$};
    \node[flag] (i2) [below of=3] {${}$};
    \draw (1) to (0);
    \draw (2) to (1);
    \draw (3) to (1);
    \draw [->] (0) to (o);
    \draw [->, white] (i1) to (2);
    \draw [->, white] (i2) to (3);
    \end{tikzpicture}
    \raisebox{0.6cm}{$\ , \ ... \ , \ $}
    \begin{tikzpicture}[node distance = 0.5cm] 
    \node[v] (1) {${}$}; 
    \node[c] (0) [above of=1] {${}$};
    \node[b] (2) [below left of=1] {${}$};
    \node[c] (3) [below right of=1] {${}$};
    \node[b] (o) [above of=0] {${}$};
    \node[c] (i1) [below of=2] {${}$};
    \node[flag] (i2) [below of=3] {${}$};
    \draw (1) to (0);
    \draw (2) to (1);
    \draw (3) to (1);
    \draw [->] (0) to (o);
    \draw [->] (i1) to (2);
    \draw [->, white] (i2) to (3);
    \end{tikzpicture}
    \raisebox{0.6cm}{$\ \circ_2 \ $}
    \begin{tikzpicture}[node distance = 0.5cm] 
    \node[v] (1) {${}$}; 
    \node[c] (0) [above of=1] {${}$};
    \node[b] (2) [below left of=1] {${}$};
    \node[c] (3) [below right of=1] {${}$};
    \node[flag] (o) [above of=0] {${}$};
    \node[c] (i1) [below of=2] {${}$};
    \node[b] (i2) [below of=3] {${}$};
    \draw (1) to (0);
    \draw (2) to (1);
    \draw (3) to (1);
    \draw [->, white] (0) to (o);
    \draw [->] (i1) to (2);
    \draw [->] (i2) to (3);
    \end{tikzpicture}
    \raisebox{0.6cm}{$\ = \ $}
    \begin{tikzpicture}[node distance = 0.5cm] 
    \node[v] (1) {${}$}; 
    \node[c] (0) [above of=1] {${}$};
    \node[b] (2) [below left of=1] {${}$};
    \node[c] (3) [below right of=1] {${}$};
    \node[b] (o) [above of=0] {${}$};
    \node[c] (i1) [below of=2] {${}$};
    \node[b] (i2) [below of=3] {${}$};
    \draw (1) to (0);
    \draw (2) to (1);
    \draw (3) to (1);
    \draw [->] (0) to (o);
    \draw [->] (i1) to (2);
    \draw [->] (i2) to (3);
    \end{tikzpicture}
    \raisebox{0.6cm}{$\ \circ_2 \ $}
    \begin{tikzpicture}[node distance = 0.5cm] 
    \node[v] (1) {${}$}; 
    \node[c] (0) [above of=1] {${}$};
    \node[b] (2) [below left of=1] {${}$};
    \node[c] (3) [below right of=1] {${}$};
    \node[b] (o) [above of=0] {${}$};
    \node[c] (i1) [below of=2] {${}$};
    \node[b] (i2) [below of=3] {${}$};
    \draw (1) to (0);
    \draw (2) to (1);
    \draw (3) to (1);
    \draw [->] (0) to (o);
    \draw [->] (i1) to (2);
    \draw [->] (i2) to (3);
    \end{tikzpicture}
    \raisebox{0.6cm}{$\ , \ $}\\
    &
    \begin{tikzpicture}[node distance = 0.5cm] 
    \node[v] (1) {${}$}; 
    \node[c] (0) [above of=1] {${}$};
    \node[b] (2) [below left of=1] {${}$};
    \node[c] (3) [below right of=1] {${}$};
    \node[flag] (o) [above of=0] {${}$};
    \node[flag] (i1) [below of=2] {${}$};
    \node[flag] (i2) [below of=3] {${}$};
    \draw (1) to (0);
    \draw (2) to (1);
    \draw (3) to (1);
    \draw [->, white] (0) to (o);
    \draw [->, white] (i1) to (2);
    \draw [->, white] (i2) to (3);
    \end{tikzpicture}
    \raisebox{0.6cm}{$\ \circ_1 \ $}
    \begin{tikzpicture}[node distance = 0.5cm] 
    \node[v] (1) {${}$}; 
    \node[c] (0) [above of=1] {${}$};
    \node[b] (2) [below left of=1] {${}$};
    \node[c] (3) [below right of=1] {${}$};
    \node[b] (o) [above of=0] {${}$};
    \node[flag] (i1) [below of=2] {${}$};
    \node[flag] (i2) [below of=3] {${}$};
    \draw (1) to (0);
    \draw (2) to (1);
    \draw (3) to (1);
    \draw [->] (0) to (o);
    \draw [->, white] (i1) to (2);
    \draw [->, white] (i2) to (3);
    \end{tikzpicture}
    \raisebox{0.6cm}{$\ = \ $}
    \begin{tikzpicture}[node distance = 0.5cm] 
    \node[v] (1) {${}$}; 
    \node[c] (0) [above of=1] {${}$};
    \node[b] (2) [below left of=1] {${}$};
    \node[c] (3) [below right of=1] {${}$};
    \node[flag] (o) [above of=0] {${}$};
    \node[c] (i1) [below of=2] {${}$};
    \node[flag] (i2) [below of=3] {${}$};
    \draw (1) to (0);
    \draw (2) to (1);
    \draw (3) to (1);
    \draw [->, white] (0) to (o);
    \draw [->] (i1) to (2);
    \draw [->, white] (i2) to (3);
    \end{tikzpicture}
    \raisebox{0.6cm}{$\ \circ_1 \ $}
    \begin{tikzpicture}[node distance = 0.5cm] 
    \node[v] (1) {${}$}; 
    \node[c] (0) [above of=1] {${}$};
    \node[b] (2) [below left of=1] {${}$};
    \node[c] (3) [below right of=1] {${}$};
    \node[flag] (o) [above of=0] {${}$};
    \node[flag] (i1) [below of=2] {${}$};
    \node[flag] (i2) [below of=3] {${}$};
    \draw (1) to (0);
    \draw (2) to (1);
    \draw (3) to (1);
    \draw [->, white] (0) to (o);
    \draw [->, white] (i1) to (2);
    \draw [->, white] (i2) to (3);
    \end{tikzpicture}
    \raisebox{0.6cm}{$\ , \ ... \ , \ $}
    \begin{tikzpicture}[node distance = 0.5cm] 
    \node[v] (1) {${}$}; 
    \node[c] (0) [above of=1] {${}$};
    \node[b] (2) [below left of=1] {${}$};
    \node[c] (3) [below right of=1] {${}$};
    \node[b] (o) [above of=0] {${}$};
    \node[flag] (i1) [below of=2] {${}$};
    \node[b] (i2) [below of=3] {${}$};
    \draw (1) to (0);
    \draw (2) to (1);
    \draw (3) to (1);
    \draw [->] (0) to (o);
    \draw [->, white] (i1) to (2);
    \draw [->] (i2) to (3);
    \end{tikzpicture}
    \raisebox{0.6cm}{$\ \circ_1 \ $}
    \begin{tikzpicture}[node distance = 0.5cm] 
    \node[v] (1) {${}$}; 
    \node[c] (0) [above of=1] {${}$};
    \node[b] (2) [below left of=1] {${}$};
    \node[c] (3) [below right of=1] {${}$};
    \node[b] (o) [above of=0] {${}$};
    \node[c] (i1) [below of=2] {${}$};
    \node[b] (i2) [below of=3] {${}$};
    \draw (1) to (0);
    \draw (2) to (1);
    \draw (3) to (1);
    \draw [->] (0) to (o);
    \draw [->] (i1) to (2);
    \draw [->] (i2) to (3);
    \end{tikzpicture}
    \raisebox{0.6cm}{$\ = \ $}
    \begin{tikzpicture}[node distance = 0.5cm] 
    \node[v] (1) {${}$}; 
    \node[c] (0) [above of=1] {${}$};
    \node[b] (2) [below left of=1] {${}$};
    \node[c] (3) [below right of=1] {${}$};
    \node[b] (o) [above of=0] {${}$};
    \node[c] (i1) [below of=2] {${}$};
    \node[b] (i2) [below of=3] {${}$};
    \draw (1) to (0);
    \draw (2) to (1);
    \draw (3) to (1);
    \draw [->] (0) to (o);
    \draw [->] (i1) to (2);
    \draw [->] (i2) to (3);
    \end{tikzpicture}
    \raisebox{0.6cm}{$\ \circ_1 \ $}
    \begin{tikzpicture}[node distance = 0.5cm] 
    \node[v] (1) {${}$}; 
    \node[c] (0) [above of=1] {${}$};
    \node[b] (2) [below left of=1] {${}$};
    \node[c] (3) [below right of=1] {${}$};
    \node[flag] (o) [above of=0] {${}$};
    \node[c] (i1) [below of=2] {${}$};
    \node[b] (i2) [below of=3] {${}$};
    \draw (1) to (0);
    \draw (2) to (1);
    \draw (3) to (1);
    \draw [->, white] (0) to (o);
    \draw [->] (i1) to (2);
    \draw [->] (i2) to (3);
    \end{tikzpicture}
    \raisebox{0.6cm}{$\ , \ $}\\
    &
    \begin{tikzpicture}[node distance = 0.5cm] 
    \node[v] (1) {${}$}; 
    \node[c] (0) [above of=1] {${}$};
    \node[b] (2) [below left of=1] {${}$};
    \node[c] (3) [below right of=1] {${}$};
    \node[flag] (o) [above of=0] {${}$};
    \node[flag] (i1) [below of=2] {${}$};
    \node[flag] (i2) [below of=3] {${}$};
    \draw (1) to (0);
    \draw (2) to (1);
    \draw (3) to (1);
    \draw [->, white] (0) to (o);
    \draw [->, white] (i1) to (2);
    \draw [->, white] (i2) to (3);
    \end{tikzpicture}
    \raisebox{0.6cm}{$\ \circ_{1,(132)} \ $}
    \begin{tikzpicture}[node distance = 0.5cm] 
    \node[v] (1) {${}$}; 
    \node[c] (0) [above of=1] {${}$};
    \node[b] (2) [below left of=1] {${}$};
    \node[c] (3) [below right of=1] {${}$};
    \node[b] (o) [above of=0] {${}$};
    \node[flag] (i1) [below of=2] {${}$};
    \node[flag] (i2) [below of=3] {${}$};
    \draw (1) to (0);
    \draw (2) to (1);
    \draw (3) to (1);
    \draw [->] (0) to (o);
    \draw [->, white] (i1) to (2);
    \draw [->, white] (i2) to (3);
    \end{tikzpicture}
    \raisebox{0.6cm}{$\ = \ $}
    \begin{tikzpicture}[node distance = 0.5cm] 
    \node[v] (1) {${}$}; 
    \node[c] (0) [above of=1] {${}$};
    \node[b] (2) [below left of=1] {${}$};
    \node[c] (3) [below right of=1] {${}$};
    \node[flag] (o) [above of=0] {${}$};
    \node[c] (i1) [below of=2] {${}$};
    \node[flag] (i2) [below of=3] {${}$};
    \draw (1) to (0);
    \draw (2) to (1);
    \draw (3) to (1);
    \draw [->, white] (0) to (o);
    \draw [->] (i1) to (2);
    \draw [->, white] (i2) to (3);
    \end{tikzpicture}
    \raisebox{0.6cm}{$\ \circ_{1,(132)} \ $}
    \begin{tikzpicture}[node distance = 0.5cm] 
    \node[v] (1) {${}$}; 
    \node[c] (0) [above of=1] {${}$};
    \node[b] (2) [below left of=1] {${}$};
    \node[c] (3) [below right of=1] {${}$};
    \node[flag] (o) [above of=0] {${}$};
    \node[flag] (i1) [below of=2] {${}$};
    \node[flag] (i2) [below of=3] {${}$};
    \draw (1) to (0);
    \draw (2) to (1);
    \draw (3) to (1);
    \draw [->, white] (0) to (o);
    \draw [->, white] (i1) to (2);
    \draw [->, white] (i2) to (3);
    \end{tikzpicture}
    \raisebox{0.6cm}{$\ , \ ... \ , \ $}
    \begin{tikzpicture}[node distance = 0.5cm] 
    \node[v] (1) {${}$}; 
    \node[c] (0) [above of=1] {${}$};
    \node[b] (2) [below left of=1] {${}$};
    \node[c] (3) [below right of=1] {${}$};
    \node[b] (o) [above of=0] {${}$};
    \node[flag] (i1) [below of=2] {${}$};
    \node[b] (i2) [below of=3] {${}$};
    \draw (1) to (0);
    \draw (2) to (1);
    \draw (3) to (1);
    \draw [->] (0) to (o);
    \draw [->, white] (i1) to (2);
    \draw [->] (i2) to (3);
    \end{tikzpicture}
    \raisebox{0.6cm}{$\ \circ_{1,(132)} \ $}
    \begin{tikzpicture}[node distance = 0.5cm] 
    \node[v] (1) {${}$}; 
    \node[c] (0) [above of=1] {${}$};
    \node[b] (2) [below left of=1] {${}$};
    \node[c] (3) [below right of=1] {${}$};
    \node[b] (o) [above of=0] {${}$};
    \node[c] (i1) [below of=2] {${}$};
    \node[b] (i2) [below of=3] {${}$};
    \draw (1) to (0);
    \draw (2) to (1);
    \draw (3) to (1);
    \draw [->] (0) to (o);
    \draw [->] (i1) to (2);
    \draw [->] (i2) to (3);
    \end{tikzpicture}
    \raisebox{0.6cm}{$\ = \ $}
    \begin{tikzpicture}[node distance = 0.5cm] 
    \node[v] (1) {${}$}; 
    \node[c] (0) [above of=1] {${}$};
    \node[b] (2) [below left of=1] {${}$};
    \node[c] (3) [below right of=1] {${}$};
    \node[b] (o) [above of=0] {${}$};
    \node[c] (i1) [below of=2] {${}$};
    \node[b] (i2) [below of=3] {${}$};
    \draw (1) to (0);
    \draw (2) to (1);
    \draw (3) to (1);
    \draw [->] (0) to (o);
    \draw [->] (i1) to (2);
    \draw [->] (i2) to (3);
    \end{tikzpicture}
    \raisebox{0.6cm}{$\ \circ_{1,(132)} \ $}
    \begin{tikzpicture}[node distance = 0.5cm] 
    \node[v] (1) {${}$}; 
    \node[c] (0) [above of=1] {${}$};
    \node[b] (2) [below left of=1] {${}$};
    \node[c] (3) [below right of=1] {${}$};
    \node[flag] (o) [above of=0] {${}$};
    \node[c] (i1) [below of=2] {${}$};
    \node[b] (i2) [below of=3] {${}$};
    \draw (1) to (0);
    \draw (2) to (1);
    \draw (3) to (1);
    \draw [->, white] (0) to (o);
    \draw [->] (i1) to (2);
    \draw [->] (i2) to (3);
    \end{tikzpicture}
    \raisebox{0.6cm}{$\ \rangle \ $}
\end{align*}
Each row above contains $2^4$ equalities corresponding to the (prior) free action of the groupoid on the external flags.
Each equality corresponds to an action of $\mathbb{V}$ along an internal edge.
We note that this example is particularly simple to compute, because the equivalence class corresponding to each groupoid coloured tree monomial with a single internal edge, contains exactly two elements.
Thus, each equation in $E_*^2$ is independent of the rest, and contains the minimal representative $[-]_*$ with respect to any total admissible order on $F_{sh}^{ob(\mathbb{V})}(E^{f_{\mathbb{V}}})$.
If we assume that, 
\begin{align}\label{eq:ordering rels in E_*^2 example}
    \begin{tikzpicture}[node distance = 0.5cm] 
    \node[v] (1) {${}$}; 
    \node[c] (0) [above of=1] {${}$};
    \node[b] (2) [below left of=1] {${}$};
    \node[c] (3) [below right of=1] {${}$};
    \node[flag] (o) [above of=0] {${}$};
    \node[flag] (i1) [below of=2] {${}$};
    \node[flag] (i2) [below of=3] {${}$};
    \draw (1) to (0);
    \draw (2) to (1);
    \draw (3) to (1);
    \draw [->, white] (0) to (o);
    \draw [->, white] (i1) to (2);
    \draw [->, white] (i2) to (3);
    \end{tikzpicture}
    \raisebox{0.6cm}{$\ \circ_1 \ $}
    \begin{tikzpicture}[node distance = 0.5cm] 
    \node[v] (1) {${}$}; 
    \node[c] (0) [above of=1] {${}$};
    \node[b] (2) [below left of=1] {${}$};
    \node[c] (3) [below right of=1] {${}$};
    \node[b] (o) [above of=0] {${}$};
    \node[flag] (i1) [below of=2] {${}$};
    \node[flag] (i2) [below of=3] {${}$};
    \draw (1) to (0);
    \draw (2) to (1);
    \draw (3) to (1);
    \draw [->] (0) to (o);
    \draw [->, white] (i1) to (2);
    \draw [->, white] (i2) to (3);
    \end{tikzpicture}
    \raisebox{0.6cm}{$\ \leq \ $}
    \begin{tikzpicture}[node distance = 0.5cm] 
    \node[v] (1) {${}$}; 
    \node[c] (0) [above of=1] {${}$};
    \node[b] (2) [below left of=1] {${}$};
    \node[c] (3) [below right of=1] {${}$};
    \node[flag] (o) [above of=0] {${}$};
    \node[c] (i1) [below of=2] {${}$};
    \node[flag] (i2) [below of=3] {${}$};
    \draw (1) to (0);
    \draw (2) to (1);
    \draw (3) to (1);
    \draw [->, white] (0) to (o);
    \draw [->] (i1) to (2);
    \draw [->, white] (i2) to (3);
    \end{tikzpicture}
    \raisebox{0.6cm}{$\ \circ_1 \ $}
    \begin{tikzpicture}[node distance = 0.5cm] 
    \node[v] (1) {${}$}; 
    \node[c] (0) [above of=1] {${}$};
    \node[b] (2) [below left of=1] {${}$};
    \node[c] (3) [below right of=1] {${}$};
    \node[flag] (o) [above of=0] {${}$};
    \node[flag] (i1) [below of=2] {${}$};
    \node[flag] (i2) [below of=3] {${}$};
    \draw (1) to (0);
    \draw (2) to (1);
    \draw (3) to (1);
    \draw [->, white] (0) to (o);
    \draw [->, white] (i1) to (2);
    \draw [->, white] (i2) to (3);
    \end{tikzpicture}
    \raisebox{0.6cm}{$\quad$ and $\quad$}
    \begin{tikzpicture}[node distance = 0.5cm] 
    \node[v] (1) {${}$}; 
    \node[c] (0) [above of=1] {${}$};
    \node[b] (2) [below left of=1] {${}$};
    \node[c] (3) [below right of=1] {${}$};
    \node[flag] (o) [above of=0] {${}$};
    \node[flag] (i1) [below of=2] {${}$};
    \node[flag] (i2) [below of=3] {${}$};
    \draw (1) to (0);
    \draw (2) to (1);
    \draw (3) to (1);
    \draw [->, white] (0) to (o);
    \draw [->, white] (i1) to (2);
    \draw [->, white] (i2) to (3);
    \end{tikzpicture}
    \raisebox{0.6cm}{$\ \circ_2 \ $}
    \begin{tikzpicture}[node distance = 0.5cm] 
    \node[v] (1) {${}$}; 
    \node[c] (0) [above of=1] {${}$};
    \node[b] (2) [below left of=1] {${}$};
    \node[c] (3) [below right of=1] {${}$};
    \node[flag] (o) [above of=0] {${}$};
    \node[c] (i1) [below of=2] {${}$};
    \node[flag] (i2) [below of=3] {${}$};
    \draw (1) to (0);
    \draw (2) to (1);
    \draw (3) to (1);
    \draw [->, white] (0) to (o);
    \draw [->] (i1) to (2);
    \draw [->, white] (i2) to (3);
    \end{tikzpicture}
    \raisebox{0.6cm}{$\  \leq  \ $}
    \begin{tikzpicture}[node distance = 0.5cm] 
    \node[v] (1) {${}$}; 
    \node[c] (0) [above of=1] {${}$};
    \node[b] (2) [below left of=1] {${}$};
    \node[c] (3) [below right of=1] {${}$};
    \node[flag] (o) [above of=0] {${}$};
    \node[flag] (i1) [below of=2] {${}$};
    \node[b] (i2) [below of=3] {${}$};
    \draw (1) to (0);
    \draw (2) to (1);
    \draw (3) to (1);
    \draw [->, white] (0) to (o);
    \draw [->, white] (i1) to (2);
    \draw [->] (i2) to (3);
    \end{tikzpicture}
    \raisebox{0.6cm}{$\ \circ_2 \ $}
    \begin{tikzpicture}[node distance = 0.5cm] 
    \node[v] (1) {${}$}; 
    \node[c] (0) [above of=1] {${}$};
    \node[b] (2) [below left of=1] {${}$};
    \node[c] (3) [below right of=1] {${}$};
    \node[b] (o) [above of=0] {${}$};
    \node[c] (i1) [below of=2] {${}$};
    \node[flag] (i2) [below of=3] {${}$};
    \draw (1) to (0);
    \draw (2) to (1);
    \draw (3) to (1);
    \draw [->] (0) to (o);
    \draw [->] (i1) to (2);
    \draw [->, white] (i2) to (3);
    \end{tikzpicture}
\end{align}
Then the following relation $r$ between groupoid coloured tree monomials in $F_{sh}^{\mathbb{V}}(E)\binom{c}{b,c,c}$, 
\begin{align*}
    \raisebox{0.6cm}{$\ r:\quad \quad \ $}
    \begin{tikzpicture}[node distance = 0.5cm] 
    \node[v] (1) {${}$}; 
    \node[c] (0) [above of=1] {${}$};
    \node[b] (2) [below left of=1] {${}$};
    \node[c] (3) [below right of=1] {${}$};
    \node[flag] (o) [above of=0] {${}$};
    \node[flag] (i1) [below of=2] {${}$};
    \node[flag] (i2) [below of=3] {${}$};
    \draw (1) to (0);
    \draw (2) to (1);
    \draw (3) to (1);
    \draw [->, white] (0) to (o);
    \draw [->, white] (i1) to (2);
    \draw [->, white] (i2) to (3);
    \end{tikzpicture}
    \raisebox{0.6cm}{$\ \circ_1 \ $}
    \begin{tikzpicture}[node distance = 0.5cm] 
    \node[v] (1) {${}$}; 
    \node[c] (0) [above of=1] {${}$};
    \node[b] (2) [below left of=1] {${}$};
    \node[c] (3) [below right of=1] {${}$};
    \node[b] (o) [above of=0] {${}$};
    \node[flag] (i1) [below of=2] {${}$};
    \node[flag] (i2) [below of=3] {${}$};
    \draw (1) to (0);
    \draw (2) to (1);
    \draw (3) to (1);
    \draw [->] (0) to (o);
    \draw [->, white] (i1) to (2);
    \draw [->, white] (i2) to (3);
    \end{tikzpicture}
    \raisebox{0.6cm}{$\ = \ $}
    \begin{tikzpicture}[node distance = 0.5cm] 
    \node[v] (1) {${}$}; 
    \node[c] (0) [above of=1] {${}$};
    \node[b] (2) [below left of=1] {${}$};
    \node[c] (3) [below right of=1] {${}$};
    \node[flag] (o) [above of=0] {${}$};
    \node[flag] (i1) [below of=2] {${}$};
    \node[flag] (i2) [below of=3] {${}$};
    \draw (1) to (0);
    \draw (2) to (1);
    \draw (3) to (1);
    \draw [->, white] (0) to (o);
    \draw [->, white] (i1) to (2);
    \draw [->, white] (i2) to (3);
    \end{tikzpicture}
    \raisebox{0.6cm}{$\ \circ_2 \ $}
    \begin{tikzpicture}[node distance = 0.5cm] 
    \node[v] (1) {${}$}; 
    \node[c] (0) [above of=1] {${}$};
    \node[b] (2) [below left of=1] {${}$};
    \node[c] (3) [below right of=1] {${}$};
    \node[flag] (o) [above of=0] {${}$};
    \node[c] (i1) [below of=2] {${}$};
    \node[flag] (i2) [below of=3] {${}$};
    \draw (1) to (0);
    \draw (2) to (1);
    \draw (3) to (1);
    \draw [->, white] (0) to (o);
    \draw [->] (i1) to (2);
    \draw [->, white] (i2) to (3);
    \end{tikzpicture}
\end{align*}
can be represented by the same pair of elements in $F_{sh}^{ob(\mathbb{V})}(E^{f_{\mathbb{V}}})\binom{c}{b,c,c}$ (i.e. in an abuse of notation $r_* = r$).
So, if $R$ is the $\mathbb{V}$-coloured module generated by $r$, i.e. it contains $2^4$ basis elements corresponding to the free action of $\mathbb{V}$ on the external flags of $r$, then $R_*$ will also contain $2^4$ elements computed similarly.
So we have calculated the constituents of \cref{An alternate presentation after forgetting the groupoid}.
In particular, if  $P:=F_{sh}^{\mathbb{V}}(E)/ \langle R \rangle $ then $P^{f_{\mathbb{V}}}\cong P_*:=  F_{sh}^{ob(\mathbb{V})}(E^{f_{\mathbb{V}}}) / \langle E_*^2  \sqcup R_* \rangle$.
Note in particular, that whilst $P_*$ has no notion of external actions of $\mathbb{V}$, that all actions of $\mathbb{V}$ along internal edges are encoded through $E_*^2$.
\\\\
Finally, consider the following four elements of the bar constructions $B(P^{f_\mathbb{V}})$.
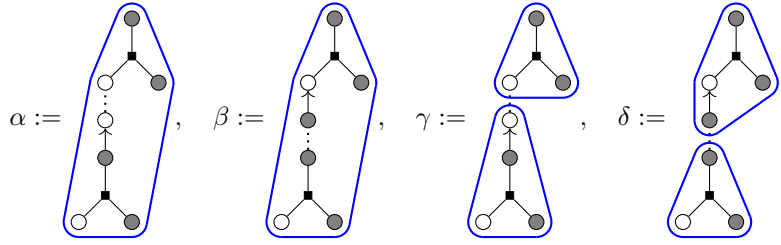
\begin{align*}
    \raisebox{1.5cm}{$\alpha:=$}
    \begin{tikzpicture}[node distance = 0.5cm]
    \node[v] (1) {${}$}; 
    \node[c] (0) [above of=1] {${}$};
    \node[b] (2) [below left of=1] {${}$};
    \node[c] (3) [below right of=1] {${}$};
    \node[b] (i1) [below of=2] {${}$};
    \node[flag] (i2) [below of=3] {${}$};
    \draw (1) to (0);
    \draw (2) to (1);
    \draw (3) to (1);
    \draw [dotted, thick] (i1) to (2);
    \draw [->, white] (i2) to (3);
    \node[c] (l0) [below of=i1] {${}$};
    \node[v] (l1) [below of=l0] {${}$}; 
    \node[b] (l2) [below left of=l1] {${}$};
    \node[c] (l3) [below right of=l1] {${}$};
    \node[flag] (lo) [above of=l0] {${}$};
    \draw (l1) to (l0);
    \draw (l2) to (l1);
    \draw (l3) to (l1);
    \draw [->] (l0) to (i1);
    \myroundpoly[blue, thick]{0,3,l3,l2,2}{0.2cm};
    \end{tikzpicture}
    \raisebox{1.5cm}{$,\quad \beta:=$}
    \begin{tikzpicture}[node distance = 0.5cm]
    \node[v] (1) {${}$}; 
    \node[c] (0) [above of=1] {${}$};
    \node[b] (2) [below left of=1] {${}$};
    \node[c] (3) [below right of=1] {${}$};
    \node[c] (i1) [below of=2] {${}$};
    \node[flag] (i2) [below of=3] {${}$};
    \draw (1) to (0);
    \draw (2) to (1);
    \draw (3) to (1);
    \draw [->] (i1) to (2);
    \draw [->, white] (i2) to (3);
    \node[c] (l0) [below of=i1] {${}$};
    \node[v] (l1) [below of=l0] {${}$}; 
    \node[b] (l2) [below left of=l1] {${}$};
    \node[c] (l3) [below right of=l1] {${}$};
    \node[flag] (lo) [above of=l0] {${}$};
    \draw (l1) to (l0);
    \draw (l2) to (l1);
    \draw (l3) to (l1);
    \draw [-, dotted, thick] (l0) to (i1);
    \myroundpoly[blue, thick]{0,3,l3,l2,2}{0.2cm};
    \end{tikzpicture}
    \raisebox{1.5cm}{$,\quad \gamma:=$}
    \begin{tikzpicture}[node distance = 0.5cm]
    \node[v] (1) {${}$}; 
    \node[c] (0) [above of=1] {${}$};
    \node[b] (2) [below left of=1] {${}$};
    \node[c] (3) [below right of=1] {${}$};
    \node[b] (i1) [below of=2] {${}$};
    \node[flag] (i2) [below of=3] {${}$};
    \draw (1) to (0);
    \draw (2) to (1);
    \draw (3) to (1);
    \draw [dotted, thick] (i1) to (2);
    \draw [->, white] (i2) to (3);
    \node[c] (l0) [below of=i1] {${}$};
    \node[v] (l1) [below of=l0] {${}$}; 
    \node[b] (l2) [below left of=l1] {${}$};
    \node[c] (l3) [below right of=l1] {${}$};
    \node[flag] (lo) [above of=l0] {${}$};
    \draw (l1) to (l0);
    \draw (l2) to (l1);
    \draw (l3) to (l1);
    \draw [->] (l0) to (i1);
    \myroundpoly[blue, thick]{0,3,2}{0.2cm};
    \myroundpoly[blue, thick]{i1,l3,l2}{0.2cm};
    \end{tikzpicture}
    \raisebox{1.5cm}{$,\quad \delta:=$}
    \begin{tikzpicture}[node distance = 0.5cm]
    \node[v] (1) {${}$}; 
    \node[c] (0) [above of=1] {${}$};
    \node[b] (2) [below left of=1] {${}$};
    \node[c] (3) [below right of=1] {${}$};
    \node[c] (i1) [below of=2] {${}$};
    \node[flag] (i2) [below of=3] {${}$};
    \draw (1) to (0);
    \draw (2) to (1);
    \draw (3) to (1);
    \draw [->] (i1) to (2);
    \draw [->, white] (i2) to (3);
    \node[c] (l0) [below of=i1] {${}$};
    \node[v] (l1) [below of=l0] {${}$}; 
    \node[b] (l2) [below left of=l1] {${}$};
    \node[c] (l3) [below right of=l1] {${}$};
    \node[flag] (lo) [above of=l0] {${}$};
    \draw (l1) to (l0);
    \draw (l2) to (l1);
    \draw (l3) to (l1);
    \draw [-, dotted, thick] (l0) to (i1);
    \myroundpoly[blue, thick]{0,3,i1,2}{0.2cm};
    \myroundpoly[blue, thick]{l0,l3,l2}{0.2cm};
    \end{tikzpicture}
\end{align*}
Here we denote,
\begin{itemize}
    \item vertices of bar constructions via blue brackets, which are labelled with elements of $P^{f_\mathbb{V}}$, 
    \item and internal edges of the operad $P^{f_\mathbb{V}}$, and the cooperad $B(P^{f_\mathbb{V}})$, with
    \begin{tikzpicture}
    \node[b] (2) [] {${}$};
    \node[b] (3) [right of=2] {${}$};
    \draw [dotted, thick] (2) to (3);
    \end{tikzpicture}, 
    \begin{tikzpicture}
    \node[c] (2) [] {${}$};
    \node[c] (3) [right of=2] {${}$};
    \draw [dotted, thick] (2) to (3);
    \end{tikzpicture}
    .
\end{itemize}
We observe through $P^{f_{\mathbb{V}}}\cong P_*$ and the relations of $E_*$, i.e. the left (ordered) relation of \cref{eq:ordering rels in E_*^2 example}, that $\alpha$ and $\beta$ are equal.
We also observe that $\gamma \neq \delta$ despite $\gamma$ and $\delta$ just being refinements of $\alpha$ and $\beta$ respectively.
We now consider the epimorphism of groupoid coloured shuffle cooperads $B(P^{f_{\mathbb{V}}}) \twoheadrightarrow B(P)$ of \cref{epimorphism of bar constructions}.
We observe that not only do we have $[\alpha] = [\beta]$, but we also have $[\gamma] = [\delta]$,
i.e. as $P$ is a groupoid coloured module, we can push and pull on internal edges of $B(-)$ through the action of $\mathbb{V}$ on $P$.
So, it is straightforward to see in this particular example that $\triangle(\alpha) \in \triangle([\alpha])$ (as needed by \cref{epimorphism of bar constructions}). 

\end{example}

\begin{remark}
We note it seems very likely that Groebner bases for dioperads could be extended to groupoid coloured dioperads by similar techniques to this section.
In addition, it seems likely that the argument of this section could be adapted to prove that a (discrete) coloured operad is Koszul if its corresponding (one coloured discrete or) uncoloured operad is Koszul (see 
Proposition 1.15 of \cite{kharitonov2021grobner}).
In particular, suppose we have a coloured operad $P$, and the forgetful functor $-^u$ from coloured operads to uncoloured operads, which forgets the underlying colouring.
Then there seems to be an epimorphism of the bar complexes $B(P^{u}) \twoheadrightarrow B(P)^u$.
However, whilst this moderately interesting, it is not computationally useful, as the Koszulness of coloured operads is already well understood.
\end{remark}

\subsection{Rewriting Systems and Groebner Bases} \label{Rewriting Systems and Groebner Bases}

We close this section by showing how a basis $G$ of an ideal of a discrete coloured shuffle operad defines a terminating rewriting system (RS), and relate the confluence of this system to $G$ being a Groebner basis (\cref{RS confluence iff GB}). 
This connection is known, and for instance, can be found for operads in Section 8 of \cite{loday2012algebraic}, we simply state it concisely here, so we may use these confluence ideas to prove the operads governing props and wheeled props are Koszul.
The following definitions are standard, and for instance can be found in \cite{dershowitz1990rewrite}.
\begin{definition}
A \textbf{rewriting system (RS)} is a binary relation $\rightarrow$ on a set $T$. The reflexive transitive closure of this relation is denoted $\rightarrow_*$. We say a given RS is \textbf{terminating}, if there exists no endless chain $t_1\rightarrow t_2 \rightarrow t_3 \rightarrow ...$ of elements in the set $T$. We say that a point $w\in T$ is,
\begin{itemize}
    \item \textbf{confluent}, if for all $x,y \in T$ such that $w\rightarrow_* x$ and $w\rightarrow_* y$, there exists a $z\in T$ such that $x\rightarrow_* z$ and $y\rightarrow_* z$. 
    \item \textbf{locally confluent}, if for all $x,y \in T$ such that $w\rightarrow x$ and $w\rightarrow y$, there exists a $z\in T$ such that $x\rightarrow_* z$ and $y\rightarrow_* z$.
\end{itemize}
If either of these properties are true for all $w \in T$, then we say the $RS$ itself is confluent, or locally confluent.
\end{definition}
These properties admit the following renowned equivalence. 

\begin{proposition}[Diamond Lemma, \cite{newman1942theories}]
A terminating RS is confluent if, and only if, it is locally confluent.
\end{proposition}

The confluence of the following terminating rewriting system associated to a basis of an ideal of a discrete coloured shuffle operad coincides with the basis being a Groebner basis.

\begin{definition}
Let $G$ be a basis for an ideal of a discrete coloured shuffle operad which admits an admissible order on its shuffle tree monomials. We define the RS associated to $G$, say $RS(G)$, to be the subset of $F_{sh}(E)^2$,
\begin{align*}
    RS(G):=\{(f,r_g(f)): f\in  F_{sh}(E), g\in G \text{ and $lt(g)$ divides $lt(f)$} \}
\end{align*}
\end{definition}

\begin{proposition}\label{RS confluence iff GB}
The following characterisations are equivalent.
\begin{enumerate}
    \item $RS(G)$ is confluent. 
    \item $RS(G)$ is confluent on small common multiples of $G$.
    \item $G$ is a Groebner basis.
\end{enumerate}
\end{proposition}
\begin{proof}
The first condition clearly implies the second. The second condition implies the third, as if $RS(G)$ is confluent on the small common multiples of $G$, then for any $S$-polynomial of $G$, say $s_{\gamma}(f,g)$, the confluence forces $s_{\gamma}(f,g) \equiv 0\ mod\ G$. In particular, we have that,
\begin{align*}
    s_{\gamma}(f,g)&= m_{\gamma,lt(f)}(f)- \frac{c_f}{c_g} m_{\gamma,lt(g)}(g)\\
    &=\frac{c_f}{c_{\gamma}}((\gamma- \frac{c_{\gamma}}{c_g} m_{\gamma,lt(g)}(g)) -(\gamma-\frac{c_{\gamma}}{c_f}m_{\gamma,lt(f)}(f))\\
    &=\frac{c_f}{c_{\gamma}}(r_g(\gamma)- r_f(\gamma))\\
    &\rightarrow_*\frac{c_f}{c_{\gamma}}(\gamma'-\gamma')=0
\end{align*}
where, $\gamma'$ is given by the confluence of small common multiples.
Finally, the third condition implies the first through the unique minimal form given by the Groebner basis. In particular, if $G$ is a Groebner basis, then $RS(G)$ is confluent as $f\rightarrow_* f'$ and $f\rightarrow_* f''$, implies $f',f''\rightarrow_*\bar{f}$.

\end{proof}
We note that the small common multiples must be a subset of the shuffle tree monomials, so showing that all shuffle tree monomials are confluent is sufficient to show we have a Groebner basis. 
We will use this technique to prove that the operads governing props and wheeled props are Koszul.
The more common technique for showing you have a Groebner basis (Buchberger's algorithm, see for instance \cite{dotsenko2010grobner}), amounts to checking the local confluence of the small common multiples (through $S$-polynomials).
Here is a closing example that illustrates the link between rewriting systems and Groebner bases.

\begin{example} [See Example 10 of \cite{dotsenko2010grobner}] \label{gb assoc example}

The associative operad admits the presentation,
{\small
\begin{align*}
    P=F_{\Sigma}(
\begin{forest}
fairly nice empty nodes,
for tree={inner sep=0, l=0}
[
    [{$1$}]
    [{$2$}] 
]
\end{forest}
)/
\langle 
\begin{forest}
fairly nice empty nodes,
for tree={inner sep=0, l=0}
[
    [
        [{$1$}]
        [{$2$}] 
    ]
    [{$3$}]
]
\end{forest}
-
\begin{forest}
fairly nice empty nodes,
for tree={inner sep=0, l=0}
[
    [{$1$}]
    [
        [{$2$}]
        [{$3$}] 
    ]
]
\end{forest}
\rangle
\end{align*}
}
The shuffle operad $P^f$ is calculated by taking the orbit of the generators and the relation.
In particular, if we let  
$f(1,2):=$
{\tiny
\begin{forest}
fairly nice empty nodes,
for tree={inner sep=0, l=0}
[
    [{$1$}]
    [{$2$}] 
]
\end{forest}
}
and 
$f(2,1):=$
{\tiny
\begin{forest}
fairly nice empty nodes,
for tree={inner sep=0, l=0}
[
    [{$2$}]
    [{$1$}] 
]
\end{forest}
}
then because we are forgetting $\Sigma$ we need to introduce a new generator to our shuffle operad
$h(1,2):= f(2,1)$. We then proceed through and orient the $3!$ orbit elements of our sole relation using $h(1,2)$ to replace $f(2,1)$.
\begin{align*}
    f(f(1,2),3) - f(1,f(2,3)) &\mapsto f(f(1,2),3) - f(1,f(2,3))\\
    f(f(1,3),2) - f(1,f(3,2)) &\mapsto f(f(1,3),2) - f(1,h(2,3))\\
    f(f(2,1),3) - f(2,f(1,3)) &\mapsto f(h(1,2),3) - h(f(1,3),2)\\
    f(f(2,3),1) - f(2,f(3,1)) &\mapsto h(1,f(2,3)) - h(h(1,3),2)\\
    f(f(3,1),2) - f(3,f(1,2)) &\mapsto f(h(1,3),2) - h(f(1,2),3)\\
    f(f(3,2),1) - f(3,f(2,1)) &\mapsto h(1,h(2,3)) - h(h(1,2),3)
\end{align*}
Hence, our shuffle operad is
$P^f = F_{sh}(f,h)/\langle G \rangle$ where $G$ is the six relations on the right. It turns out that there exists a total admissible order such that $G$ is a (quadratic) Groebner basis. This Groebner basis admits many small common multiples with $S$-polynomials witnessing their reduction. Here is one small common multiple whose two possible reduction paths generate the associative coherence pentagon.

{\footnotesize
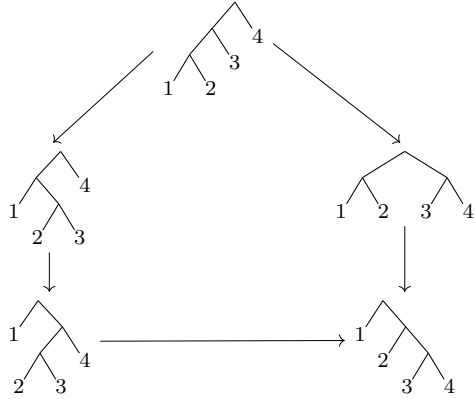
\begin{center}
\begin{tikzcd}
&
\begin{forest}
fairly nice empty nodes,
for tree={inner sep=0, l=0}
[
    [
        [
            [{$1$}]
            [{$2$}] 
        ]
        [{$3$}]
    ]
[{$4$}]
]
\end{forest}
\arrow[rd]
\arrow[ld]
\\
\begin{forest}
fairly nice empty nodes,
for tree={inner sep=0, l=0}
[
    [
        [{$1$}]
        [
            [{$2$}]
            [{$3$}] 
        ]
    ]
[{$4$}]
]
\end{forest}
\arrow[d]
&
&
\begin{forest}
fairly nice empty nodes,
for tree={inner sep=0, l=0}
[
    [
        [{$1$}]
        [{$2$}] 
    ]
    [
        [{$3$}]
        [{$4$}] 
    ]
]
\end{forest}
\arrow[d]
\\
\begin{forest}
fairly nice empty nodes,
for tree={inner sep=0, l=0}
[
    [{$1$}]
    [
        [
            [{$2$}]
            [{$3$}] 
        ]
        [{$4$}]
    ]
]
\end{forest}
\arrow[rr, to path= (\tikztostart.318) -- (\tikztotarget.217)]
&
&
\begin{forest}
fairly nice empty nodes,
for tree={inner sep=0, l=0}
[
    [{$1$}]
    [
        [{$2$}]
        [
            [{$3$}]
            [{$4$}] 
        ]
    ]
]
\end{forest}
\end{tikzcd} 
\end{center}
}
In this example every single tree monomial except the bottom right corner is divisible by
{\small
\begin{center}
$lt(g):=lt($
\begin{forest}
fairly nice empty nodes,
for tree={inner sep=0, l=0}
[
    [
        [{$1$}]
        [{$2$}] 
    ]
    [{$3$}]
]
\end{forest}
-
\begin{forest}
fairly nice empty nodes,
for tree={inner sep=0, l=0}
[
    [{$1$}]
    [
        [{$2$}]
        [{$3$}] 
    ]
]
\end{forest}
$
)
=
$
\begin{forest}
fairly nice empty nodes,
for tree={inner sep=0, l=0}
[
    [
        [{$1$}]
        [{$2$}] 
    ]
    [{$3$}]
]
\end{forest}
, e.g.
$\quad$
\begin{forest}
fairly nice empty nodes,
for tree={inner sep=0, l=0}
[
    [{$1$}]
    [
        [
            [{$2$}]
            [{$3$}] 
        ]
        [{$4$}]
    ]
]
\end{forest}
$=$
\begin{forest}
fairly nice empty nodes,
for tree={inner sep=0, l=0}
[
    [{$1$}]
    [{$2$}] 
]
\end{forest}
$\circ_{2,id}$
\begin{forest}
fairly nice empty nodes,
for tree={inner sep=0, l=0}
[
    [
        [{$1$}]
        [{$2$}] 
    ]
    [{$3$}]
]
\end{forest}

\end{center}
}
Furthermore, each arrow corresponds to a reduction $f \mapsto r_g(f)$. The top tree monomial $\gamma$ is a small common multiple as it is divisible by $lt(g)$ in two different ways and the corresponding subtrees of the divisors overlap. We can calculate
{\small
\begin{center}
$s_{\gamma}=$
\begin{forest}
fairly nice empty nodes,
for tree={inner sep=0, l=0}
[
    [
        [{$1$}]
        [
            [{$2$}]
            [{$3$}] 
        ]
    ]
[{$4$}]
]
\end{forest}
$-$
\begin{forest}
fairly nice empty nodes,
for tree={inner sep=0, l=0}
[
    [
        [{$1$}]
        [{$2$}] 
    ]
    [
        [{$3$}]
        [{$4$}] 
    ]
]
\end{forest}
\end{center}
}
and the congruence of this $S$-polynomial to $0$ witnesses that the two reductions of $\gamma$ are (locally) confluent. Note, the unique normal form of all tree monomials in the diagram is the bottom right tree monomial.
\end{example}

\section{The Operad Governing Wheeled Props}
\label{The Operad Governing Wheeled Props}

We now seek to define a groupoid coloured operad governing wheeled props, which we can show is Koszul, through use of the preceding theory. 
The main ingredient needed is a new biased definition of a wheeled prop (\cref{Alternate biased wheeled prop}), whose axioms, apart from the unital and biequivariance axioms, are all quadratic.
By translating the biequivariance axioms into actions of a groupoid, we will define $\mathbb{W}$ to be a quadratic non-unital groupoid coloured operad whose algebras are non-unital wheeled props. 
We then prove this operad is Koszul, by constructing its shuffle operad, and applying the Groebner basis machinery.

\subsection{Wheeled Props} \label{section:An Alternate Biased Definition of a Wheeled Prop}

We introduce our alternate definition of a wheeled prop, which is precisely Definition 11.33 of \cite{yau2015foundation} with all compositions and axioms extended via actions of the bimodule (see \cref{Bimodules}). We then prove its equivalence to the original definition, and define augmented and non-unital versions of these structures.
A key property of the new definition is that its equivariance axioms yield the following result.
\begin{proposition} \label{simple canonical form of alternate wheeled props}
Every valid composite of operations in \cref{Alternate biased wheeled prop} admits a simple canonical form, where all non-identity permutations are pushed to the outermost operation.
\end{proposition}
\begin{proof}
While we have an inner operation with non-identity permutations, use right compatibility to push the permutations off it and left compatibility to push them back on to the next operation.
\end{proof}
This not only allows us to simplify the remaining axioms, but also when we translate the equivariance axioms into actions of the  groupoid, it will induce a simple canonical representative of every groupoid coloured tree monomial in $\mathbb{W}$ (see \cref{Example of Action of the Groupoid}).
We note that all definitions and diagrams referenced in the following definition are sourced from \cite{yau2015foundation}. 

\begin{definition}\label{Alternate biased wheeled prop}
Let $\mathcal{E}$ be a symmetric monoidal category, a \textbf{wheeled prop} consists of the following.
\begin{enumerate}
    \item A bimodule $P:\ptwoc \to \mathcal{E}$.
    \item An extended horizontal composition,
    \begin{align*}
        P\dc \otimes P\ba \xrightarrow{\hcomp,\binom{\sigma}{\tau}}P\binom{\sigma(\ud,\ub)}{(\uc , \ua)\tau},
    \end{align*}
    where $(\hcomp,\binom{\sigma}{\tau})(\alpha,\beta) := \hcomp(\alpha,\beta)\cdot\binom{\sigma}{\tau}$, with $\hcomp(-,-)$ being the horizontal composition of 11.33.
    \item An extended contraction,
    \begin{align*}
        P\dc \xrightarrow{(\contra{j}{i},\binom{\sigma}{\tau})} P\binom{\sigma(\ud \setminus d_i)}{(\uc \setminus c_j)\tau},
    \end{align*}
    where $(\contra{j}{i},\binom{\sigma}{\tau})(\alpha) := \contra{j}{i}(\alpha)\cdot\binom{\sigma}{\tau}$, with $\contra{j}{i}(-)$ being the contraction of 11.33.
    \item The same units as 11.33 (or to be consistent, they may be extended with trivial permutations).
\end{enumerate}
In presenting the axioms, we assume all profiles, permutations and contraction indices are such that the diagrams make sense.
Before each diagram, we cite the corresponding diagram in \cite{yau2015foundation}.
We will not express the various extended unity diagrams (as we will quickly move to non-unital wheeled props) but the obvious extensions to ($11.21$ and $11.33$) are also present.
The extended horizontal composition operation satisfies the following right, left and switch compatibility axioms.
The right compatibility axiom is an immediate consequence of being a horizontal composition extended by an action of the bimodule.
The left and switch compatibility axioms correspond to diagrams 11.19 and 11.20.
\begin{center}
\begin{tikzcd}
P\dc \otimes P\ba
\arrow[d,"{(\hcomp,\binom{\sigma}{\tau})}"']
\arrow[rd,"{(\otimes_h , \binom{\sigma'\sigma}{\tau \tau'})}"]\\
P\binom{\sigma(\ud,\ub)}{(\uc , \ua)\tau}  \arrow[r,"{\binom{\sigma'}{\tau'}}"']& P\binom{\sigma'\sigma(\ud,\ub)}{(\uc , \ua)\tau\tau'}
\end{tikzcd}
$\quad$
\begin{tikzcd}
P\dc \otimes P\ba
\arrow[d,"{\binom{\sigma_1}{\tau_1}\otimes\binom{\sigma_2}{\tau_2} }"']
\arrow[rd,"{(\otimes_h , \binom{\sigma(\sigma_1\times\sigma_2)}{(\tau_1\times\tau_2)\tau})}"]\\
P\binom{\sigma_1\ud}{\uc \tau_1} \otimes P\binom{\sigma_2\ub}{\ua \tau_2}  \arrow[r,"{(\otimes_h , \binom{\sigma}{\tau})}"']& P\binom{\sigma(\sigma_1\ud,\sigma_2 \ub)}{(\uc \tau_1,\ua \tau_2)\tau)}
\end{tikzcd}
$\quad$
\begin{tikzcd}
P\ba \otimes P\dc
\arrow[d,"switch"']
\arrow[rd,"{(\otimes_h , \binom{\sigma'}{\tau'})}"]\\
P\dc \otimes P \ba \arrow[r,"{(\otimes_h , \binom{\sigma}{\tau})}"']& P\binom{\sigma(\ud,\ub)}{(\uc,\ua)\tau }
\end{tikzcd}
\end{center}
In the switch compatibility diagram, $\binom{\sigma'}{\tau'}$ is the unique permutation such that $\binom{\sigma(\ud,\ub)}{(\uc,\ua)\tau )} = \binom{\sigma'(\ub,\ud)}{ (\ua,\uc)\tau'}$. The extended contraction operator satisfies the following left and right compatibility axioms. The right compatibility axiom is an immediate consequence of being a contraction operation extended by an action of the bimodule, and the left compatibility axiom corresponds to (11.34). 
\begin{center}
\begin{tikzcd}
P\dc
\arrow[d,"{(\contra{j}{i},\binom{\sigma}{\tau})}"']
\arrow[rd,"{(\contra{j}{i},\binom{\sigma'\sigma}{\tau\tau'})}"]\\
P\binom{\sigma(\ud \setminus d_i)}{(\uc \setminus c_j)\tau} \arrow[r,"{\binom{\sigma'}{\tau'}}"']& P\binom{\sigma'\sigma(\ud \setminus d_i)}{(\uc \setminus c_j)\tau\tau'}
\end{tikzcd}
$\quad$
\begin{tikzcd}
P\dc
\arrow[d,"{\binom{\sigma'}{\tau'}}"']
\arrow[rd,"{(\contra{j}{i},\binom{\sigma(\sigma'^{(i)})}{(\tau'^{(j)})\tau})}"]\\
P\binom{\sigma'\ud}{\uc \tau'}
\arrow[r,"{(\contra{\tau'(j)}{\sigma'^{-1}(i)},\binom{\sigma}{\tau})}"']& P\binom{\sigma(\sigma'^{(i)})(\ud\setminus d_i)}{(\uc \setminus c_j)(\tau'^{(j)})\tau }
\end{tikzcd}
\end{center}
In the left compatibility diagram $\sigma'^{(i)}$ and $\tau'^{(j)}$ are the obvious permutations acting on $\ud \setminus d_i$ and $\uc \setminus c_j$ which are induced by the permutations
$\sigma'$ and $\tau'$ (acting on $\ud$, $\uc$).
\\\\
The language of
\cref{simple canonical form of alternate wheeled props} should now be clear, and we can use it to simplify the remaining diagrams (by pushing all permutations from the top/left arrows to the bottom/right arrows). We note that it is possible to encode the remaining diagrams entirely graphically (see \cref{The Relations of the Operad Governing Wheeled Props}). We will however write out the remaining axioms in full to make it clear they are precisely the axioms of \cite{yau2015foundation} extended by an action of the bimodule. This helps to clarify that the definitions are indeed equivalent.
\\\\
The associativity diagram of $\hcomp$ ($11.18$)
\begin{center}
\begin{tikzcd} [column sep=1in]
P\fe \otimes P\dc \otimes P\ba
\arrow[r,"{(\hcomp,\binom{id}{id})\otimes id}"]
\arrow[d,"{id\otimes(\hcomp,\binom{id}{id})}"']
&P \binom{\uf,\ud}{\ue,\uc}\otimes P\ba
\arrow[d,"{(\hcomp,\binom{\sigma}{\tau})}"]
\\
P \fe \otimes P\binom{\ud,\ub}{\uc,\ua}
\arrow[r,"{(\hcomp,\binom{\sigma}{\tau})}"']&
P\binom{\sigma(\uf,\ud,\ub)}{(\ue,\uc,\ua)\tau}
\end{tikzcd}
\end{center}
The commutativity diagrams of $\contra{}{}$ ($11.35, 11.36$). Suppose that $|d|,|c|\geq 2$, $d_i = c_j$ and $d_{i'} = c_{j'}$ for some $i\neq i' \leq |d|$ and $j\neq j' \leq |c|$. In the left diagram $i<i',j<j'$, and in the right diagram $i>i',j<j'$.
\begin{center}
\begin{tikzcd}
P\dc
\arrow[r,"{(\contra{j}{i},\binom{id}{id})}"]
\arrow[d,"{(\contra{j'}{i'},\binom{id}{id})}"']
&P\binom{\ud \setminus d_i}{\uc \setminus c_j}
\arrow[d,"{(\contra{j'-1}{i'-1},\binom{\sigma}{\tau})}"]
\\
P\binom{\ud \setminus d_{i'}}{\uc \setminus c_{j'}}
\arrow[r,"{(\contra{j}{i},\binom{\sigma}{\tau})}"']&
P\binom{\sigma(\ud \setminus \{d_{i},d_{i'}\})}{(\uc \setminus \{c_j,c_{j'}\})\tau}
\end{tikzcd}
$\quad\quad\quad$
\begin{tikzcd}
P\dc
\arrow[r,"{(\contra{j}{i},\binom{id}{id})}"]
\arrow[d,"{(\contra{j'}{i'},\binom{id}{id})}"']
&P\binom{\ud \setminus d_i}{\uc \setminus c_j}
\arrow[d,"{(\contra{j'-1}{i'},\binom{\sigma}{\tau})}"]
\\
P\binom{\ud \setminus d_{i'}}{\uc \setminus c_{j'}}
\arrow[r,"{(\contra{j}{i-1},\binom{\sigma}{\tau})}"']&
P\binom{\sigma(\ud \setminus \{d_{i},d_{i'}\})}{(\uc \setminus \{c_j,c_{j'}\})\tau}
\end{tikzcd}
\end{center}
Note the right diagram has slightly altered indexing from $11.36$, which uses $i<i'$ and $j>j'$ instead. This alteration arises from a later choice in \cref{Total Admissible Order for Shuffle Tree Monomials of Wheeled Props} where we choose to prioritise inputs over outputs, and making this (equivalent) alteration here slightly cleans up the relations. 
\\\\
The compatibility diagrams of $\hcomp, \contra{}{}$ ($11.37, 11.38$). In the left diagram $i\leq |d|,j\leq |c|$, and in the right diagram $i\leq |b|,j\leq |a|$

\begin{center}
\begin{tikzcd}
P\dc \otimes P\ba
\arrow[r,"{(\hcomp,\binom{id}{id})}"]
\arrow[d,"{(\contra{j}{i},\binom{id}{id})\otimes id}"']
&P\binom{\ud,\ub}{\uc,\ua}
\arrow[d,"{(\contra{j}{i},\binom{\sigma}{\tau})}"]
\\
P\binom{\ud \setminus d_{i}}{\uc \setminus c_{j}} \otimes P\ba
\arrow[r,"{(\hcomp,\binom{\sigma}{\tau})}"']&
P\binom{\sigma(\ud \setminus d_{i},\ub)}{(\uc \setminus c_{j},\ua)\tau}
\end{tikzcd}
$\quad\quad\quad$
\begin{tikzcd}
P\dc \otimes P\ba
\arrow[r,"{(\hcomp,\binom{id}{id})}"]
\arrow[d,"{id\otimes (\contra{j}{i},\binom{id}{id})}"']
&P\binom{\ud,\ub}{\uc,\ua}
\arrow[d,"{(\contra{|\uc|+j}{|\ud|+i},\binom{\sigma}{\tau})}"]
\\
P\dc \otimes P\binom{\ub \setminus b_{i}}{\ua \setminus a_{j}}
\arrow[r,"{(\hcomp,\binom{\sigma}{\tau})}"']&
P\binom{\sigma(\ud,\ub \setminus b_{i})}{(\uc ,\ua \setminus a_j)\tau}
\end{tikzcd}
\end{center}

\end{definition}

\begin{lemma}
\cref{Alternate biased wheeled prop}, and Definition 11.33 of \cite{yau2015foundation} are equivalent definitions of wheeled props.
\end{lemma}
\begin{proof}
In presenting the alternate definition, we already established how the extended operations and axioms can be obtained from the original operations and axioms. We can similarly re-obtain the original operations and axioms from the alternate operations and axioms by taking specific cases of them.
\end{proof}
We will now refer to alternate wheeled props as wheeled props, but will continue to use alternate if we need to emphasise the use of this particular presentation.
We now define augmented and non-unital versions of this definition.

\begin{definition}\label{Trivial Wheeled Prop}

Let $\mathcal{K}$ be the trivial wheeled prop in $Vect_{\mathbb{K}}$ whose only constituents are the horizontal and vertical units, and the contractions of the vertical units (this is a technical requirement for wheeled props). 
\end{definition}

\begin{definition}\label{Augmented Wheeled Prop}
An \textbf{augmentation} of a wheeled prop $P$ in $dgVect_{\mathbb{K}}$ is a morphism (of wheeled props, see Corollary 11.36 of \cite{yau2015foundation}) $\epsilon:P\to \mathcal{K}$. Wheeled props with an augmentation are called \textbf{augmented wheeled props}.
\end{definition}

\begin{definition}\label{non-unital wheeled prop}
A \textbf{non-unital wheeled prop} consists of all the data of a wheeled prop, excluding the units and the corresponding axioms.
\end{definition}

As one would expect, we have the following isomorphism.

\begin{proposition}
Augmented wheeled props and non-unital wheeled props are isomorphic.
\end{proposition}
\begin{proof}
This is a clear analogue of the classical result for partial operads, see Proposition 21 of \cite{markl2008operads}.
\end{proof}

\subsection{A Groupoid Coloured Quadratic Operad}
\label{biased presentation of the operad governing wheeled props}

We now produce a quadratic non-unital groupoid coloured operad $\mathbb{W}:=F_{\Sigma}^{\SC}(E)/ \langle R \rangle$ whose algebras are non-unital wheeled props.
After we have presented the generators of our operad, we will explore in \cref{Example of Action of the Groupoid} how the groupoid actions enable a simple canonical representative of each groupoid coloured tree monomial of $F_{\Sigma}^{\SC}(E)$.
Finally, we use this canonical form to provide a diagrammatic representation of the quadratic ideal $R$.
\\\\
The underlying groupoid of our operad is $\SC:=\ptwoc$, and the groupoid coloured module $E$ (\cref{Groupoid Coloured Bimodule}) is constructed from the compositions of the wheeled prop as follows. Let $\ba,\dc,\fe \in \ptwoc$, then,
\begin{align*}
    E(\dc;\ba):=& \{(\contra{j}{i},\binom{\sigma}{\tau})(-): d_i = c_j \text{, and } \binom{\sigma(\ud \setminus \{d_i\})}{(\uc \setminus \{c_j\})\tau} =  \ba \}\\
    E(\dc,\ba; \fe):=&\{(\hcomp,\binom{\sigma}{\tau})(-,-): \binom{\sigma(\ud,\ub)}{(\uc , \ua)\tau} = \fe\}
\end{align*}
where $i\in \{1,...,|\ud|\},j\in \{1,...,|\uc|\}$ and the $\binom{\tau}{\sigma}$'s are morphisms in $\ptwoc$. 
The necessary left, right and $\Sigma_2$ actions on $\hcomp$ are directly given by the left, right and switch biequivariance/compatibility axioms of $\hcomp$ (\cref{Alternate biased wheeled prop}).
Similarly, the necessary left and right actions on $\contra{}{}$ are directly given by the left and right compatibility axioms of $\contra{}{}$.
The necessary $\Sigma_1$ action on $\contra{}{}$ is the trivial action.
To clarify by example, we now unpack the right action of $\hcomp$.
Let $(\hcomp,\binom{\sigma}{\tau}) \in E(\dc,\ba; \fe)$ and $\binom{\sigma'}{\tau'} \in Iso_{\SC}(\fe,\binom{\sigma' \uf}{\ue \tau'})$ then,
\begin{align*}
    (\hcomp,\binom{\sigma}{\tau})\binom{\sigma'}{\tau'} = (\hcomp,\binom{\sigma'\sigma}{\tau\tau'})\in E(\dc,\ba; \binom{\sigma' \uf}{\ue \tau'})
\end{align*}
i.e. the bottom-left and diagonal path for the right compatibility diagram of $\hcomp$ are equal (\cref{Alternate biased wheeled prop}).

We now provide an example of composition in this operad, and these actions.
\begin{example}\label{Example of Action of the Groupoid}
Suppose we have the following profiles,
\begin{align*}
    p_x = \binom{d_1,d_2}{c_1,c_2,c_3}, p_y = \binom{c_3}{\emptyset}, p=\binom{d_2,d_1,c_3}{c_2,c_3,c_1},p_{\gamma} = \binom{d_2,d_1}{c_2,c_1}
\end{align*}
then using the inline notation for permutations (here and in the rest of the paper)
\begin{align*}
    (\hcomp,\binom{213}{231}) \in E\binom{p}{p_x,p_y} \quad \text{and} \quad (\contra{2}{3},\binom{id}{id}) \in E\binom{p_{\gamma}}{p} 
\end{align*}
We can compose these two generators to form a groupoid coloured tree monomial with at least two distinct representatives.
{\small
\begin{center}
$(\contra{2}{3},\binom{id}{id}) \circ_1 (\hcomp,\binom{213}{231})=[$
\begin{forest}
for tree={parent anchor=south}
[{$(\contra{2}{3},\binom{id}{id})$}
    [{$(\hcomp,\binom{213}{231})$}
        [$p_x$]
        [$p_y$] 
    ]
]
\end{forest}
$]=[$
\begin{forest}
for tree={parent anchor=south}
[{$(\contra{3}{3},\binom{21}{21})$}
    [{$(\hcomp,\binom{id}{id})$}
        [$p_x$]
        [$p_y$] 
    ]
]
\end{forest}
$]\in E\binom{p_\gamma}{p_x,p_y}$
\end{center}
}
We can explicitly compute that these two tree monomials are part of the same equivalence class as follows. The output edge of $(\hcomp,\binom{213}{231})$
is coloured by the profile $p$ (in this example). 
The identity morphism is an automorphism of $p$ and splits as $id = p \xrightarrow{\binom{213}{231}^{-1}} p'  \xrightarrow{\binom{213}{231}} p \in Aut(p)$ where $p' = \binom{d_1,d_2,c_3}{c_1,c_2,c_3}$. Hence by \cref{Decorated Symmetric Tree}, contraction left action, and horizontal right action we have the following equalities.
{\small
\begin{center}
$[$
\begin{forest}
for tree={parent anchor=south}
[{$(\contra{2}{3},\binom{id}{id})$}
    [{$(\hcomp,\binom{213}{231})$}
        [$p_x$]
        [$p_y$] 
    ]
]
\end{forest}
$]\stackrel{id}{=}[$
\begin{forest}
for tree={parent anchor=south}
[{$\binom{213}{231}\cdot(\contra{2}{3},\binom{id}{id})$}
    [{$(\hcomp,\binom{213}{231})\cdot \binom{213}{231}^{-1}$}
        [$p_x$]
        [$p_y$] 
    ]
]
\end{forest}
$]\stackrel{\text{left+right actions}}{=}[$
\begin{forest}
for tree={parent anchor=south}
[{$(\contra{3}{3},\binom{21}{21})$}
    [{$(\hcomp,\binom{id}{id})$}
        [$p_x$]
        [$p_y$] 
    ]
]
\end{forest}
$]$
\end{center}
}
By construction, every shuffle tree monomial of the operad can be seen as providing instructions for forming a wheeled graph via horizontal compositions and contractions. 
The fact that these two tree monomials are in the same equivalence class can be visually verified by seeing they provide two different ways of forming the following graph $\gamma$ from two corollas $*_{x}$ and $*_{y}$ through different horizontal compositions and contractions.
(Recall that all directed graphs in this paper are drawn with inputs on the bottom, and outputs on the top).
{\small
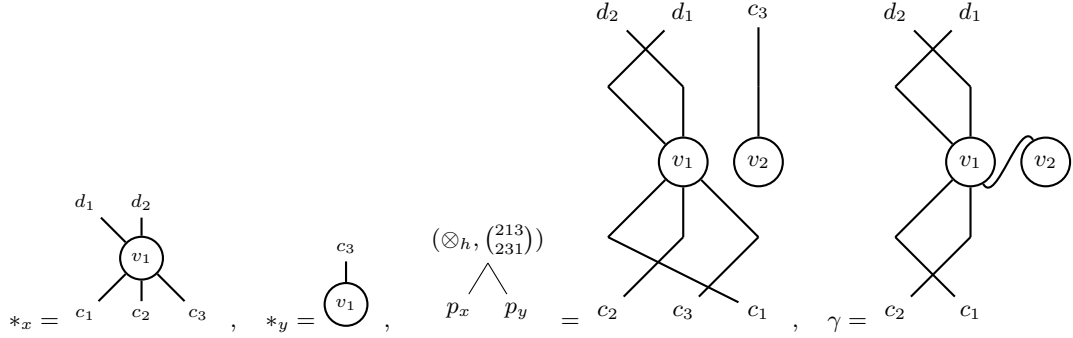
\begin{center}
$*_{x}=$
{\scriptsize
\begin{tikzpicture}[node distance={7.5mm},thick] 
\node[flag] (v1o1) {$d_1$};
\node[flag][right of=v1o1] (v1o2) {$d_2$};
\node[vertex] (v1) [below of=v1o2]{$v_1$};
\node[flag] [below of=v1] (v1i2) {$c_2$};
\node[flag] [left of=v1i2] (v1i1) {$c_1$};
\node[flag] [right of=v1i2] (v1i3) {$c_3$};
\draw [black] (v1) to (v1o1);
\draw [black] (v1) to (v1o2);
\draw [black] (v1) to (v1i1);
\draw [black] (v1) to (v1i2);
\draw [black] (v1) to (v1i3);
\end{tikzpicture}
}
$,\quad *_{y}=$
{\scriptsize
\begin{tikzpicture}[node distance={7.5mm},thick] 
\node[flag] (v2o1) {$c_3$};
\node[vertex] (v2) [below of=v2o1] {$v_1$};
\draw [thick] (v2) to (v2o1);
\end{tikzpicture}
}
$,\quad$ 
\begin{forest}
for tree={parent anchor=south}
[{$(\hcomp,\binom{213}{231})$}
    [$p_x$]
    [$p_y$] 
]
\end{forest}
=
\begin{tikzpicture}[node distance={10mm},thick] 
\node[flag] (go1) {$d_2$};
\node[flag] (go2) [right of=go1] {$d_1$};
\node[flag] (go3) [right of=go2] {$c_3$};
\coordinate[below of=go1] (v1o1);
\coordinate[right of=v1o1] (v1o2);
\coordinate[right of=v1o2] (v2o1);
\node[vertex] (v1) [below of=v1o2]{$v_1$}; 
\node[vertex] (v2) [right of=v1] {$v_2$};
\coordinate[below of=v1] (v1i2) ;
\coordinate[left of=v1i2] (v1i1) ;
\coordinate[right of=v1i2] (v1i3) ;
\node[flag] (gi1) [below of=v1i1] {$c_2$};
\node[flag] (gi2) [right of=gi1] {$c_3$};
\node[flag] (gi3) [right of=gi2] {$c_1$};
\draw [black] (go1) to (v1o2);
\draw [black] (go2) to (v1o1);
\draw [black] (go3) to (v2o1);
\draw [black] (v1) to (v1o1);
\draw [black] (v1) to (v1o2);
\draw [black] (v2) to (v2o1);
\draw [black] (v1) to (v1i1);
\draw [black] (v1) to (v1i2);
\draw [black] (v1) to (v1i3);
\draw [black] (v1i1) to (gi3);
\draw [black] (v1i2) to (gi1);
\draw [black] (v1i3) to (gi2);
\end{tikzpicture}
$,\quad \gamma=$
\begin{tikzpicture}[node distance={10mm},thick] 
\node[flag] (go1) {$d_2$};
\node[flag] (go2) [right of=go1] {$d_1$};
\coordinate[below of=go1] (v1o1);
\coordinate[right of=v1o1] (v1o2);
\coordinate[right of=v1o2] (v2o1);
\node[vertex] (v1) [below of=v1o2]{$v_1$}; 
\node[vertex] (v2) [right of=v1] {$v_2$};
\coordinate[below of=v1] (v1i2) ;
\coordinate[left of=v1i2] (v1i1) ;
\coordinate[right of=v1i2] (v1i3) ;
\node[flag] (gi1) [below of=v1i1] {$c_2$};
\node[flag] (gi3) [right of=gi1] {$c_1$};
\draw [black] (go1) to (v1o2);
\draw [black] (go2) to (v1o1);
\draw [black] (v1) to (v1o1);
\draw [black] (v1) to (v1o2);
\draw [thick] (v2) to [out=120,in=-60] (v1);
\draw [black] (v1) to (v1i1);
\draw [black] (v1) to (v1i2);
\draw [black] (v1i1) to (gi3);
\draw [black] (v1i2) to (gi1);
\end{tikzpicture}
\end{center}
}
We note that the vertex labelling of $\gamma$ (i.e. $v_1,v_2$), is encoded by the order of the inputs of the tree monomial, i.e. the corolla $*_x$ is labelled $v_1$ and the corolla $*_y$ is labelled $v_2$ because $(\contra{2}{3},\binom{id}{id}) \circ_1 (\hcomp,\binom{213}{231}) \in E\binom{p_\gamma}{p_x,p_y}$.
Thus, we can also interpret the symmetric action of the operad as switching the labels of the vertices.
For example, here is the graphical interpretation of the $\Sigma_2$ action on $(\hcomp,\binom{213}{231}) \in E\binom{p}{p_x,p_y}$.
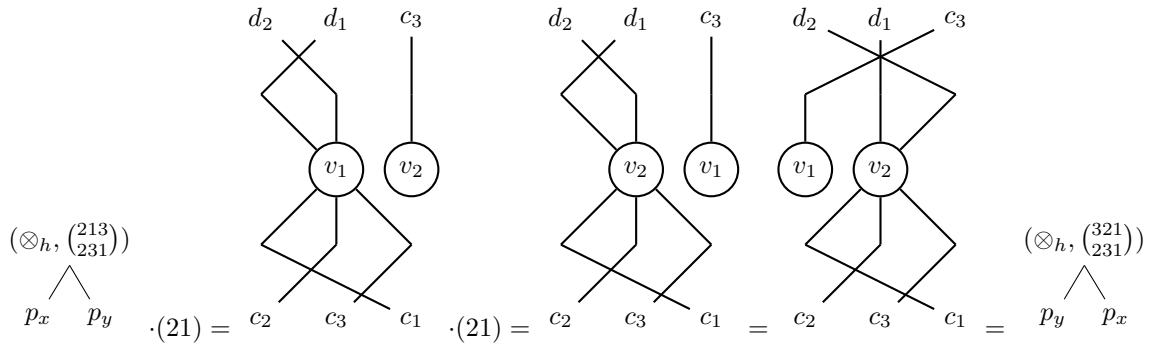
\begin{center}
\begin{forest}
for tree={parent anchor=south}
[{$(\hcomp,\binom{213}{231})$}
    [$p_x$]
    [$p_y$] 
]
\end{forest}
$\cdot (21)
=$
\begin{tikzpicture}[node distance={10mm},thick] 
\node[flag] (go1) {$d_2$};
\node[flag] (go2) [right of=go1] {$d_1$};
\node[flag] (go3) [right of=go2] {$c_3$};
\coordinate[below of=go1] (v1o1);
\coordinate[right of=v1o1] (v1o2);
\coordinate[right of=v1o2] (v2o1);
\node[vertex] (v1) [below of=v1o2]{$v_1$}; 
\node[vertex] (v2) [right of=v1] {$v_2$};
\coordinate[below of=v1] (v1i2) ;
\coordinate[left of=v1i2] (v1i1) ;
\coordinate[right of=v1i2] (v1i3) ;
\node[flag] (gi1) [below of=v1i1] {$c_2$};
\node[flag] (gi2) [right of=gi1] {$c_3$};
\node[flag] (gi3) [right of=gi2] {$c_1$};
\draw [black] (go1) to (v1o2);
\draw [black] (go2) to (v1o1);
\draw [black] (go3) to (v2o1);
\draw [black] (v1) to (v1o1);
\draw [black] (v1) to (v1o2);
\draw [black] (v2) to (v2o1);
\draw [black] (v1) to (v1i1);
\draw [black] (v1) to (v1i2);
\draw [black] (v1) to (v1i3);
\draw [black] (v1i1) to (gi3);
\draw [black] (v1i2) to (gi1);
\draw [black] (v1i3) to (gi2);
\end{tikzpicture}
$\cdot (21)
=$
\begin{tikzpicture}[node distance={10mm},thick] 
\node[flag] (go1) {$d_2$};
\node[flag] (go2) [right of=go1] {$d_1$};
\node[flag] (go3) [right of=go2] {$c_3$};
\coordinate[below of=go1] (v1o1);
\coordinate[right of=v1o1] (v1o2);
\coordinate[right of=v1o2] (v2o1);
\node[vertex] (v1) [below of=v1o2]{$v_2$}; 
\node[vertex] (v2) [right of=v1] {$v_1$};
\coordinate[below of=v1] (v1i2) ;
\coordinate[left of=v1i2] (v1i1) ;
\coordinate[right of=v1i2] (v1i3) ;
\node[flag] (gi1) [below of=v1i1] {$c_2$};
\node[flag] (gi2) [right of=gi1] {$c_3$};
\node[flag] (gi3) [right of=gi2] {$c_1$};
\draw [black] (go1) to (v1o2);
\draw [black] (go2) to (v1o1);
\draw [black] (go3) to (v2o1);
\draw [black] (v1) to (v1o1);
\draw [black] (v1) to (v1o2);
\draw [black] (v2) to (v2o1);
\draw [black] (v1) to (v1i1);
\draw [black] (v1) to (v1i2);
\draw [black] (v1) to (v1i3);
\draw [black] (v1i1) to (gi3);
\draw [black] (v1i2) to (gi1);
\draw [black] (v1i3) to (gi2);
\end{tikzpicture}
$=$
\begin{tikzpicture}[node distance={10mm},thick] 
\node[flag] (go1) {$d_2$};
\node[flag] (go2) [right of=go1] {$d_1$};
\node[flag] (go3) [right of=go2] {$c_3$};
\coordinate[below of=go1] (v2o1);
\coordinate[below of=go2] (v1o1);
\coordinate[below of=go3] (v1o2);
\node[vertex] (v2) [below of=v2o1] {$v_1$};
\node[vertex] (v1) [right of=v2]{$v_2$}; 
\coordinate[below of=v1] (v1i2) ;
\coordinate[left of=v1i2] (v1i1) ;
\coordinate[right of=v1i2] (v1i3) ;
\node[flag] (gi1) [below of=v1i1] {$c_2$};
\node[flag] (gi2) [right of=gi1] {$c_3$};
\node[flag] (gi3) [right of=gi2] {$c_1$};
\draw [black] (go1) to (v1o2);
\draw [black] (go2) to (v1o1);
\draw [black] (go3) to (v2o1);
\draw [black] (v1) to (v1o1);
\draw [black] (v1) to (v1o2);
\draw [black] (v2) to (v2o1);
\draw [black] (v1) to (v1i1);
\draw [black] (v1) to (v1i2);
\draw [black] (v1) to (v1i3);
\draw [black] (v1i1) to (gi3);
\draw [black] (v1i2) to (gi1);
\draw [black] (v1i3) to (gi2);
\end{tikzpicture}
$=$
\begin{forest}
for tree={parent anchor=south}
[{$(\hcomp,\binom{321}{231})$}
    [$p_y$]
    [$p_x$] 
]
\end{forest}
\end{center}
Here the outermost tree monomials are equal under the definition of the $\Sigma_2$ action.
The inner equalities provide a graphical interpretation of this result, where the second equality uses the $\Sigma_2$ action to switch the vertex labels, and the third equality is an equality of graphs (these graphs are indistinguishable under \cref{Wheeled Graph}).
This example is also sufficient to illustrate that horizontal composition is not necessarily strictly commutative (see \cref{inclusion of com} for an important case when it is commutative), as 
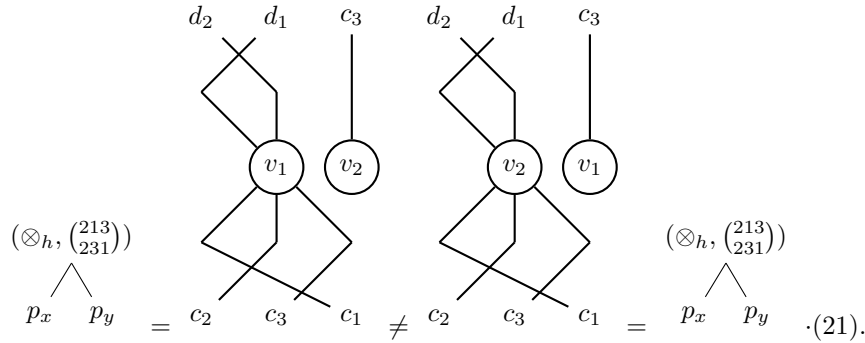
\begin{center}
\begin{forest}
for tree={parent anchor=south}
[{$(\hcomp,\binom{213}{231})$}
    [$p_x$]
    [$p_y$] 
]
\end{forest}
$=$
\begin{tikzpicture}[node distance={10mm},thick] 
\node[flag] (go1) {$d_2$};
\node[flag] (go2) [right of=go1] {$d_1$};
\node[flag] (go3) [right of=go2] {$c_3$};
\coordinate[below of=go1] (v1o1);
\coordinate[right of=v1o1] (v1o2);
\coordinate[right of=v1o2] (v2o1);
\node[vertex] (v1) [below of=v1o2]{$v_1$}; 
\node[vertex] (v2) [right of=v1] {$v_2$};
\coordinate[below of=v1] (v1i2) ;
\coordinate[left of=v1i2] (v1i1) ;
\coordinate[right of=v1i2] (v1i3) ;
\node[flag] (gi1) [below of=v1i1] {$c_2$};
\node[flag] (gi2) [right of=gi1] {$c_3$};
\node[flag] (gi3) [right of=gi2] {$c_1$};
\draw [black] (go1) to (v1o2);
\draw [black] (go2) to (v1o1);
\draw [black] (go3) to (v2o1);
\draw [black] (v1) to (v1o1);
\draw [black] (v1) to (v1o2);
\draw [black] (v2) to (v2o1);
\draw [black] (v1) to (v1i1);
\draw [black] (v1) to (v1i2);
\draw [black] (v1) to (v1i3);
\draw [black] (v1i1) to (gi3);
\draw [black] (v1i2) to (gi1);
\draw [black] (v1i3) to (gi2);
\end{tikzpicture}
$\neq$
\begin{tikzpicture}[node distance={10mm},thick] 
\node[flag] (go1) {$d_2$};
\node[flag] (go2) [right of=go1] {$d_1$};
\node[flag] (go3) [right of=go2] {$c_3$};
\coordinate[below of=go1] (v1o1);
\coordinate[right of=v1o1] (v1o2);
\coordinate[right of=v1o2] (v2o1);
\node[vertex] (v1) [below of=v1o2]{$v_2$}; 
\node[vertex] (v2) [right of=v1] {$v_1$};
\coordinate[below of=v1] (v1i2) ;
\coordinate[left of=v1i2] (v1i1) ;
\coordinate[right of=v1i2] (v1i3) ;
\node[flag] (gi1) [below of=v1i1] {$c_2$};
\node[flag] (gi2) [right of=gi1] {$c_3$};
\node[flag] (gi3) [right of=gi2] {$c_1$};
\draw [black] (go1) to (v1o2);
\draw [black] (go2) to (v1o1);
\draw [black] (go3) to (v2o1);
\draw [black] (v1) to (v1o1);
\draw [black] (v1) to (v1o2);
\draw [black] (v2) to (v2o1);
\draw [black] (v1) to (v1i1);
\draw [black] (v1) to (v1i2);
\draw [black] (v1) to (v1i3);
\draw [black] (v1i1) to (gi3);
\draw [black] (v1i2) to (gi1);
\draw [black] (v1i3) to (gi2);
\end{tikzpicture}
$=$
\begin{forest}
for tree={parent anchor=south}
[{$(\hcomp,\binom{213}{231})$}
    [$p_x$]
    [$p_y$] 
]
\end{forest}
$\cdot (21)$.
\end{center}

\end{example}

Observe that in the prior example, the action of the groupoid was used to push all permutations to the top (i.e. the only non-identity permutation is that used by the generator at the root). This trick can be exploited to produce a canonical element of the equivalence class of any groupoid coloured tree monomial, i.e. whilst there is a non-identity permutation on a generator below the root, we use the action of the groupoid via the internal edge above it to push the permutation up the tree. So, the canonical element of any $\mathbb{V}$-coloured tree monomial can be taken as the $ob(\mathbb{V})$-coloured tree monomial whose (potentially) only non-identity permutations are those at the root most generator. 
\\\\
The relations $R$ of our operad $\mathbb{W}$ are given by the five remaining quadratic axioms of the alternate definition of a wheeled prop. These can be explicitly written out if desired, but we will use the canonical form to provide a concise graphical representation of the relations in \cref{The Relations of the Operad Governing Wheeled Props} (see the souls of \cite{batanin2021koszul} for a similar graphical encoding of relations). Note, by construction this table precisely corresponds to the non-unital and non-biequivariance axioms of \cref{Alternate biased wheeled prop}. The table makes use of a suppressed notation, which we clarify by example. On the second row,
\begin{itemize}
    \item The diagram (of the graph) represents a graph with two vertices in which one of the output flags of the first vertex is connected to one of the input flags of the first vertex, forming an edge. There are no other edges present in the graph, but there may be other flags (which are suppressed) and these flags may have an arbitrary listing in the graph (i.e. the graph may have any permutations of its inputs and outputs).
    \item The groupoid coloured tree monomials to the right are then taken under the implicit understanding that they output the graph on the left. 
    Consequently, the profile of the $i$th input of the tree monomial monomial must correspond to the profile of the $i$th vertex of the graph, see \cref{Example of Action of the Groupoid}.
    (We are actually using a slightly abusive notation that we will clarify in \cref{Succinct Notation for Shuffle Permutations}).
    The unary vertices of the tree monomials correspond to contractions, and the binary vertices correspond to horizontal compositions. 
    There is no need to explicitly include the permutations or the indices used by these compositions. 
    A canonical element of the equivalence class of each groupoid coloured tree monomial can be recovered by pushing all permutations to the top.
    From here, the necessary indices to contract by to form the edge are uniquely specified by the graph; similarly, the topmost permutations $\binom{\sigma}{\tau}$ are uniquely determined by needing to match the listings of the open flags of the graph.
\end{itemize}

\begin{table}[ht]
\centering
\begin{tabular}{ |c|c| } 
\hline
Graphs & Relations \\
\hline
\noindent\parbox[c]{3cm}{
\begin{tikzpicture}[node distance={10mm}, thick,
v/.style={circle, draw}] 
\node[v] (1) {$v_1$}; 
\node[v] (2) [right of=1] {$v_2$};
\node[v] (3) [right of=2] {$v_3$};
\end{tikzpicture}
}
&
\noindent\parbox[c]{3.3cm}{
\begin{forest}
fairly nice empty nodes,
for tree={inner sep=0, l=0}
[
    [{$v_1$}]
    [
        [{$v_2$}]
        [{$v_3$}] 
    ]
]
\end{forest}
$=$
\begin{forest}
fairly nice empty nodes,
for tree={inner sep=0, l=0}
[
    [
        [{$v_1$}]
        [{$v_2$}] 
    ]
    [{$v_3$}]
]
\end{forest}
}
\\
\hline
\noindent\parbox[c]{2cm}{
\vspace{-0.25cm}
\begin{tikzpicture}[node distance={10mm}, thick,
v/.style={circle, draw}] 
\node[v] (1) {$v_1$}; 
\node[v] (2) [right of=1] {$v_2$};
\draw [thick] (1) to [out=120,in=240,looseness =3] (1);
\end{tikzpicture}
\vspace{-0.75cm}
}
& 
\noindent\parbox[c]{2.4cm}{
\begin{forest}
fairly nice empty nodes,
for tree={inner sep=0, l=0}
[
[[{$v_1$}]]
[{$v_2$}] 
]
\end{forest}
$=$
\begin{forest}
fairly nice empty nodes,
for tree={inner sep=0, l=0}
[
[
    [{$v_1$}]
    [{$v_2$}] 
]
]
\end{forest}
}
\\
\hline
\noindent\parbox[c]{2cm}{
\vspace{-0.25cm}
\begin{tikzpicture}[node distance={10mm}, thick,
v/.style={circle, draw}] 
\node[v] (1) {$v_1$}; 
\node[v] (2) [right of=1] {$v_2$};
\draw [thick] (2) to [out=60,in=-60,looseness =3] (2);
\end{tikzpicture}
\vspace{-0.75cm}
}
& 
\noindent\parbox[c]{2.4cm}{
\begin{forest}
fairly nice empty nodes,
for tree={inner sep=0, l=0}
[
[{$v_1$}]
[[{$v_2$}]]
]
\end{forest}
$=$
\begin{forest}
fairly nice empty nodes,
for tree={inner sep=0, l=0}
[
[
    [{$v_1$}]
    [{$v_2$}] 
]
]
\end{forest}
}
\\
\hline
\noindent\parbox[c]{1.6cm}{
\vspace{-0.25cm}
\begin{tikzpicture}[node distance={10mm}, thick,
v/.style={circle, draw}] 
\node[v] (1) {$v_1$};
\draw [thick,red] (1) to [out=120,in=240,looseness =3] (1);
\draw [thick,blue] (1) to [out=60,in=-60,looseness =3] (1);
\end{tikzpicture}
\vspace{-0.75cm}
}
,
\noindent\parbox[c]{1.6cm}{
\vspace{-0.25cm}
\begin{tikzpicture}[node distance={10mm}, thick,
v/.style={circle, draw}] 
\node[v] (1) {$v_1$};
\draw [thick,red] (1) to [out=75,in=220,looseness=4] (1);
\draw [thick,blue] (1) to [out=105,in=-40,looseness=4] (1);
\end{tikzpicture}
\vspace{-0.75cm}
}
& 
\noindent\parbox[c]{1.2cm}{
\begin{forest}
fairly nice empty nodes,
for tree={inner sep=0, l=0}
[ 
[ , edge={thick,blue}
[{$v_1$}, edge={thick,red}
]
]
]
\end{forest}
$=$
\begin{forest}
fairly nice empty nodes,
for tree={inner sep=0, l=0}
[ 
[ , edge={thick,red}
[{$v_1$}, edge={thick,blue}
]
]
]
\end{forest}
}
\\
\hline
\end{tabular}    
\caption{The relations of the operad governing wheeled props, and dually, the non-unital and non-biequivariance relations of an alternate wheeled prop.}
\label{The Relations of the Operad Governing Wheeled Props}
\end{table}

By the construction of this section, the following is immediate.

\begin{corollary}
A algebra over $\mathbb{W}$ is a non-unital wheeled prop.
\end{corollary}

\subsection{The Shuffle Operad} \label{shuffle operad for wheeled props}

So having defined a quadratic groupoid coloured operad governing wheeled props $\mathbb{W}$, we must now prove that it is Koszul. We will do this by showing $(\mathbb{W}^f)_*$ admits a quadratic Groebner basis, but first we must calculate $\mathbb{W}^f=( F_{\Sigma}^{\SC}(E)/\langle R \rangle)^f \cong F_{sh}^{\SC}(E^f)/\langle R^f \rangle$. It may be helpful to review the calculation of $-^f$ in the (discrete) coloured case prior to stepping through this section, see for instance Section $2.6$ of \cite{kharitonov2021grobner} or \cref{gb assoc example}.
\\\\
We first observe how $-^f$ acts on the generators. Every unary generator has only a trivial orbit under $\Sigma_1$. The binary generators have non-trivial orbits, however the $\Sigma_2$ action on $\hcomp$ sends every generator to another generator (see \cref{Example of Action of the Groupoid}), so there is no need to introduce further generators to account for the orbits (this is comparable to the shuffle Lie operad only needing a single generator, see Example 9 of \cite{dotsenko2010grobner}). This tells us that $E^f = E$.
\\\\
We now consider how $-^f$ acts on the relations, we do so by calculating the orbit or each relation of $\mathbb{W}$ then orienting each element of the orbit to use shuffle compositions. Let $A$ be the first relation of \cref{The Relations of the Operad Governing Wheeled Props} written down for a particular graph via the following notation, $\hcomp(\hcomp(*_{v_1},*_{v_2}),*_{v_3}) =  \hcomp(*_{v_1},\hcomp(*_{v_2},*_{v_3}))$, where $*_{v_i}$ is the profile of the $i$th vertex of the graph.
We now orient each element of the $\Sigma_3$ orbit of $A$,
\begin{align*}
    A \cdot (123)&\mapsto \quad  \hcomp(\hcomp(*_{v_1},*_{v_2}),*_{v_3})\cdot (123) = \hcomp(*_{v_1},\hcomp(*_{v_2},*_{v_3}))\cdot (123)\\
    A \cdot (132)&\mapsto \quad \hcomp(\hcomp(*_{v_1},*_{v_2}),*_{v_3})\cdot (132) = \hcomp(*_{v_1},\hcomp(*_{v_3},*_{v_2})) \cdot (123)\\
    A \cdot (213)&\mapsto \quad \hcomp(\hcomp(*_{v_2},*_{v_1}),*_{v_3})\cdot (123) = \hcomp(\hcomp(*_{v_2},*_{v_2}),*_{v_3}) \cdot (132)\\
    A \cdot (231)&\mapsto \quad \hcomp(*_{v_3},\hcomp(*_{v_1},*_{v_2}))\cdot (123) = \hcomp(\hcomp(*_{v_3},*_{v_2}),*_{v_1}) \cdot (132)\\
    A \cdot (312)&\mapsto \quad \hcomp(\hcomp(*_{v_2},*_{v_1}),*_{v_3})\cdot (132)= \hcomp(\hcomp(*_{v_2},*_{v_3}),*_{v_1}) \cdot (123)\\
    A \cdot (321)&\mapsto \quad \hcomp(*_{v_3},\hcomp(*_{v_2},*_{v_1}))\cdot (123) = \hcomp(\hcomp(*_{v_3},*_{v_2}),*_{v_1}) \cdot (123)
\end{align*}
One might expect that every time we orient an element of the orbit we would obtain a new relation in the shuffle operad, however observe that the orientation of the orbit $A\cdot (321)$ is precisely $A'=\hcomp(\hcomp(*_{v'_1},*_{v'_2}),*_{v'_3}) =  \hcomp(*_{v'_1},\hcomp(*_{v'_2},*_{v'_3}))$ where $*_{v'_1} = *_{v_3}$, $*_{v'_2} = *_{v_2}$ and $*_{v'_3}=*_{v_1}$, so it is a redundant relation. In fact, it is straightforward to show that each orientation arises as a consequence of the following two families of equations (the families are all graphs of this form with all possible graph and vertex profiles)
\begin{center}
\begin{tikzpicture}[node distance={10mm}, thick,
v/.style={circle, draw}] 
\node[v] (1) {$v_1$}; 
\node[v] (2) [right of=1] {$v_2$};
\node[v] (3) [right of=2] {$v_3$};
\end{tikzpicture}
,\quad
\begin{forest}
fairly nice empty nodes,
for tree={inner sep=0, l=0}
[
    [{$v_1$}]
    [
        [{$v_2$}]
        [{$v_3$}] 
    ]
]
\end{forest}
$=$
\begin{forest}
fairly nice empty nodes,
for tree={inner sep=0, l=0}
[
    [
        [{$v_1$}]
        [{$v_2$}] 
    ]
    [{$v_3$}]
]
\end{forest}
,\quad
\begin{forest}
fairly nice empty nodes,
for tree={inner sep=0, l=0}
[
    [
        [{$v_1$}]
        [{$v_3$}] 
    ]
    [{$v_2$}]
]
\end{forest}
$=$
\begin{forest}
fairly nice empty nodes,
for tree={inner sep=0, l=0}
[
    [
        [{$v_1$}]
        [{$v_2$}] 
    ]
    [{$v_3$}]
]
\end{forest}

\end{center}
Here we have used the notation,
\begin{align}\label{Succinct Notation for Shuffle Permutations}
\begin{forest}
fairly nice empty nodes,
for tree={inner sep=0, l=0}
[
    [
        [{$v_1$}]
        [{$v_2$}] 
    ]
    [{$v_3$}]
]
\end{forest}
:=
\begin{forest}
fairly nice empty nodes,
for tree={inner sep=0, l=0}
[
    [
        [{$*_{v_1}$}]
        [{$*_{v_2}$}] 
    ]
    [{$*_{v_3}$}]
]
\end{forest}
\cdot (123)
,\quad \quad
\begin{forest}
fairly nice empty nodes,
for tree={inner sep=0, l=0}
[
    [
        [{$v_1$}]
        [{$v_3$}] 
    ]
    [{$v_2$}]
]
\end{forest}
:=
\begin{forest}
fairly nice empty nodes,
for tree={inner sep=0, l=0}
[
    [
        [{$*_{v_1}$}]
        [{$*_{v_3}$}] 
    ]
    [{$*_{v_2}$}]
]
\end{forest}
\cdot(132)
\end{align}
The permutation $(132)$ is ensuring that the profile of the corolla $*_{v_3}$ (which is just a element of $\ptwoc$ and has no notion of order with other profiles, see the end of \cref{Example of Action of the Groupoid}) is the third input of the shuffle tree monomial.
So, this orientation of the orbit has revealed a new way to form this graph with $3$ vertices via shuffle tree monomials (given we have broken the symmetry). If we proceed to calculate the orbit of the remaining relations of the operad governing wheeled props (each with arity $\leq 2$) then we will find that no new relations are required. 
\\\\
We once again encode the relations of the shuffle operad graphically in \cref{The Relations of the Shuffle Operad Governing Wheeled Props}. This table uses the same conventions as \cref{The Relations of the Operad Governing Wheeled Props}, and the notation regarding shuffle permutations defined in \cref{Succinct Notation for Shuffle Permutations}. Furthermore, each equation is now numbered for referencing and directed using a total order given in the next section, the larger element is on the left-hand side of each equation.

\begin{table}[ht]
\begin{tabular}{ |c|c| } 
\hline
Graphs & Relations \\
\hline
\noindent\parbox[c]{3cm}{
\begin{tikzpicture}[node distance={10mm}, thick,
v/.style={circle, draw}] 
\node[v] (1) {$v_1$}; 
\node[v] (2) [right of=1] {$v_2$};
\node[v] (3) [right of=2] {$v_3$};
\end{tikzpicture}
}
&
\noindent\parbox[c]{7.4cm}{
\begin{forest}
fairly nice empty nodes,
for tree={inner sep=0, l=0}
[
    [{$v_1$}]
    [
        [{$v_2$}]
        [{$v_3$}] 
    ]
]
\end{forest}
$\xrightarrow{1}$
\begin{forest}
fairly nice empty nodes,
for tree={inner sep=0, l=0}
[
    [
        [{$v_1$}]
        [{$v_2$}] 
    ]
    [{$v_3$}]
]
\end{forest}
,
\begin{forest}
fairly nice empty nodes,
for tree={inner sep=0, l=0}
[
    [
        [{$v_1$}]
        [{$v_3$}] 
    ]
    [{$v_2$}]
]
\end{forest}
$\xrightarrow{2}$
\begin{forest}
fairly nice empty nodes,
for tree={inner sep=0, l=0}
[
    [
        [{$v_1$}]
        [{$v_2$}] 
    ]
    [{$v_3$}]
]
\end{forest}
}
\\
\hline
\noindent\parbox[c]{2cm}{
\vspace{-0.25cm}
\begin{tikzpicture}[node distance={10mm}, thick,
v/.style={circle, draw}] 
\node[v] (1) {$v_1$}; 
\node[v] (2) [right of=1] {$v_2$};
\draw [thick] (1) to [out=120,in=240,looseness =3] (1);
\end{tikzpicture}
\vspace{-0.75cm}
}
& 
\noindent\parbox[c]{3cm}{
\begin{forest}
fairly nice empty nodes,
for tree={inner sep=0, l=0}
[
[[{$v_1$}]]
[{$v_2$}] 
]
\end{forest}
$\xrightarrow{3}$
\begin{forest}
fairly nice empty nodes,
for tree={inner sep=0, l=0}
[
[
    [{$v_1$}]
    [{$v_2$}] 
]
]
\end{forest}
}
\\
\hline
\noindent\parbox[c]{2cm}{
\vspace{-0.25cm}
\begin{tikzpicture}[node distance={10mm}, thick,
v/.style={circle, draw}] 
\node[v] (1) {$v_1$}; 
\node[v] (2) [right of=1] {$v_2$};
\draw [thick] (2) to [out=60,in=-60,looseness =3] (2);
\end{tikzpicture}
\vspace{-0.75cm}
}
& 
\noindent\parbox[c]{3cm}{
\begin{forest}
fairly nice empty nodes,
for tree={inner sep=0, l=0}
[
[{$v_1$}]
[[{$v_2$}]]
]
\end{forest}
$\xrightarrow{4}$
\begin{forest}
fairly nice empty nodes,
for tree={inner sep=0, l=0}
[
[
    [{$v_1$}]
    [{$v_2$}] 
]
]
\end{forest}
}
\\
\hline
\noindent\parbox[c]{1.6cm}{
\vspace{-0.25cm}
\begin{tikzpicture}[node distance={10mm}, thick,
v/.style={circle, draw}] 
\node[v] (1) {$v_1$};
\draw [thick,red] (1) to [out=120,in=240,looseness =3] (1);
\draw [thick,blue] (1) to [out=60,in=-60,looseness =3] (1);
\end{tikzpicture}
\vspace{-0.75cm}
}
,
\noindent\parbox[c]{1.6cm}{
\vspace{-0.25cm}
\begin{tikzpicture}[node distance={10mm}, thick,
v/.style={circle, draw}] 
\node[v] (1) {$v_1$};
\draw [thick,red] (1) to [out=75,in=220,looseness=4] (1);
\draw [thick,blue] (1) to [out=105,in=-40,looseness=4] (1);
\end{tikzpicture}
\vspace{-0.75cm}
}
& 
\noindent\parbox[c]{2cm}{
\begin{forest}
fairly nice empty nodes,
for tree={inner sep=0, l=0}
[ 
[ , edge={thick,blue}
[{$v_1$}, edge={thick,red}
]
]
]
\end{forest}
$\xrightarrow{5,6}$
\begin{forest}
fairly nice empty nodes,
for tree={inner sep=0, l=0}
[ 
[ , edge={thick,red}
[{$v_1$}, edge={thick,blue}
]
]
]
\end{forest}
}
\\
\hline
\end{tabular}
\captionsetup{justification=centering}
\caption{The relations of the shuffle operad governing wheeled props}
\label{The Relations of the Shuffle Operad Governing Wheeled Props}
\end{table}

\subsection{Ordering the Object Coloured Tree Monomials} \label{Ordering the Object Coloured Tree Monomials for Wheeled Props}

In order to show that $(\mathbb{W}_f)^*$ admits a Groebner basis, we must define a total admissible order on the underlying $ob(\SC)$-coloured tree monomials of the shuffle operad.
In this section we define an order on the $ob(\SC)$-coloured tree monomials of $F_{sh}(E)$ and show that this order is total and admissible. The order is a variant of a path lexicographic order of Section 3.2 of \cite{dotsenko2010grobner} (for extending path lexicographic orders to discrete coloured operads see \cite{kharitonov2021grobner}), and although it looks complicated, we will show that it is a composite of simpler admissible orders. The underlying motivation for the order is that it provides a simple unique minimal shuffle tree monomial encoding any wheeled graph (\cref{UMF Algorithm for Wheeled Props}). 
As such, the key features of the order that the reader should observe are the following, the order prefers
\begin{itemize}
    \item contractions being above horizontal compositions.
    \item horizontal compositions being composed to the left, i.e. $\hcomp(\hcomp(v_1,v_2),v_3) < \hcomp(v_1,\hcomp(v_2,v_3))$.
    \item tree monomials whose inputs are ordered linearly, i.e. $\hcomp(\hcomp(v_1,v_2),v_3)<\hcomp(\hcomp(v_1,v_3),v_2)$.
    \item tree monomials where the only non-identity permutations used by a generator are at the root.
    \item contractions where larger input indices are used.
\end{itemize}
Here is the formal definition of our needed order, and we prove it is total and admissible in \cref{Nested Admissible is Admissible}.
\begin{definition}[Total Admissible Order] \label{Total Admissible Order for Shuffle Tree Monomials of Wheeled Props}
Let $\alpha,\beta$ be two tree monomials of $F_{sh}(E)$, we define $\alpha\leq \beta$ if:
\begin{enumerate}
    \item The arity of $\alpha < \beta$.
    \item Or, if all the prior are equal, then compare $P^{\alpha}<P^{\beta}$ where:
    \begin{itemize}
        \item if $\alpha$ and $\beta$ have arity $n$ then $P^{\alpha} := (P^{\alpha}_1,...,P^{\alpha}_n)$, where $P^{\alpha}_i$ is the word formed out of the generators when stepping from the $i$th input to the root in the tree monomial $\alpha$.
        \item $P^{\alpha}$ and $P^{\beta}$ are compared lexicographically. Two paths from the same input are compared 
        \begin{itemize}
            \item First by degree, the \textbf{longer} path is smaller.
            \item Next by a partial lexicographic order where two generators are compared using the partial order $\hcomp<\contra{}{}$ (i.e. indices and permutations are ignored).
        \end{itemize}
    \end{itemize}
    \item Or, if all the prior are equal, then compare the \textbf{input permutations} of the shuffle tree monomials (not the permutations of the generators) via lexicographic order.
    \item Or, if all the prior are equal, then compare the permutations of all the generators as follows
    \begin{itemize}
        \item For a given generator $s$ (e.g. $(\hcomp,\binom{\sigma}{\tau})$ let $p(s)$ be its permutations (e.g. $p((\hcomp,\binom{\sigma}{\tau})) = \binom{\sigma}{\tau}$).
        \item For any given tree monomial there is a natural total order of its generators/vertices, where any vertex with $k$ vertices between it and the root is smaller than any vertex with $j<k$ vertices between it and the root, and any two vertices at the same depth are ordered left to right using the planar embedding of the tree monomial.
        \item We use this order and the map $p$ to write out the permutations of all generators of each tree monomial as a single word (the letters correspond to $\binom{\sigma}{\tau}$).
        \item We compare the words lexicographically, where $\binom{\sigma}{\tau}\leq \binom{\sigma'}{\tau'}$ if $\sigma < \sigma'$, or $\sigma=\sigma'$ and $\tau<\tau'$. Here, two individual permutations are compared lexicographically.
    \end{itemize}
    \item Or, if all the prior are equal, then we again compare $P^{\alpha}<P^{\beta}$ lexicographically, this time with a total order on generators.
    \begin{itemize}
        \item $\hcomp<\contra{}{}$
        \item $\contra{j}{i}\leq\contra{j'}{i'}$ if $j>j'$, or $j=j'$ and $i>i'$.
        \item If the composition type is the same then the permutations are compared with, where we define $\binom{\sigma}{\tau}\leq \binom{\sigma'}{\tau'}$ if $\sigma < \sigma'$, or $\sigma=\sigma'$ and $\tau<\tau'$
        \item if $s$ and $s'$ are identical on all prior checks then compare their input colour(s) (left then right input if $\hcomp$) then their output colour. Input and output colours are profiles $\dc,\ba \in \ptwoc$, we say $\dc\leq \ba$ if $\ud< \ub$, or $\ud=\ub$ and $\uc\leq \ua$. We compare two sequences of colours using a degree lexicographic order induced by a total order on $\mathfrak{C}$.
    \end{itemize}
\end{enumerate}
\end{definition}
\begin{remark}
The assumption of having a total order on the underlying set of colours $\mathfrak{C}$ (which is implied by the axiom of choice) could be relaxed. We can define the order of \cref{Total Admissible Order for Shuffle Tree Monomials of Wheeled Props} as a partial order, where we do not compare generators via their input and output colours. This order will still be total on the equivalence class corresponding to each groupoid coloured tree monomial, and will still enable the confluence of the rewriting system which is equivalent to the Groebner basis to be computed.
\end{remark}

We note that several choices were made in the order, such as favouring inputs over outputs or larger indices over smaller indices.
Alternate orders exist, and this particular order was constructed as it provides a simple unique minimal form.

\begin{example} \label{Total Order for Wheeled Props Example}
The following ordered tree monomials all output the graph depicted on the left.
{\small
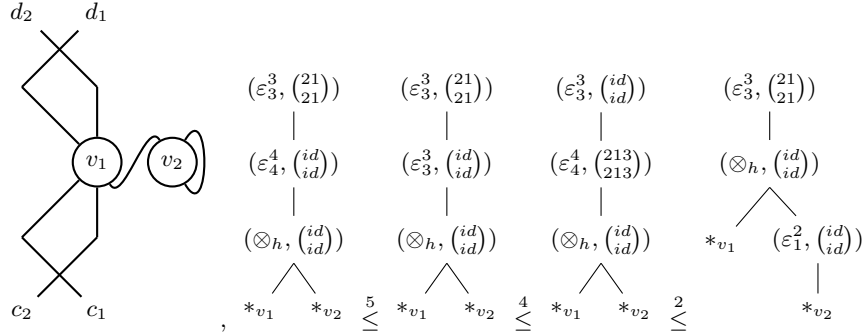
\begin{center}
\begin{tikzpicture}[node distance={10mm},thick] 
\node[flag] (go1) {$d_2$};
\node[flag] (go2) [right of=go1] {$d_1$};
\coordinate[below of=go1] (v1o1);
\coordinate[right of=v1o1] (v1o2);
\coordinate[right of=v1o2] (v2o1);
\node[vertex] (v1) [below of=v1o2]{$v_1$}; 
\node[vertex] (v2) [right of=v1] {$v_2$};
\coordinate[below of=v1] (v1i2) ;
\coordinate[left of=v1i2] (v1i1) ;
\coordinate[right of=v1i2] (v1i3) ;
\node[flag] (gi1) [below of=v1i1] {$c_2$};
\node[flag] (gi3) [right of=gi1] {$c_1$};
\draw [black] (go1) to (v1o2);
\draw [black] (go2) to (v1o1);
\draw [black] (v1) to (v1o1);
\draw [black] (v1) to (v1o2);
\draw [thick] (v2) to [out=120,in=-60] (v1);
\draw [black] (v1) to (v1i1);
\draw [black] (v1) to (v1i2);
\draw [black] (v1i1) to (gi3);
\draw [black] (v1i2) to (gi1);
\draw [thick] (v2) to [out=60,in=-60,looseness=3] (v2);
\end{tikzpicture}
,
\begin{forest} for tree={parent anchor=south}
[{$(\contra{3}{3},\binom{21}{21})$}
    [{$(\contra{4}{4},\binom{id}{id})$}
        [{$(\hcomp,\binom{id}{id})$}
            [$*_{v_1}$]
            [$*_{v_2}$] 
        ]
    ]
]
\end{forest}
$\stackrel{5}{\leq}$
\begin{forest} for tree={parent anchor=south}
[{$(\contra{3}{3},\binom{21}{21})$}
    [{$(\contra{3}{3},\binom{id}{id})$}
        [{$(\hcomp,\binom{id}{id})$}
            [$*_{v_1}$]
            [$*_{v_2}$] 
        ]
    ]
]
\end{forest}
$\stackrel{4}{\leq}$
\begin{forest} for tree={parent anchor=south}
[{$(\contra{3}{3},\binom{id}{id})$}
    [{$(\contra{4}{4},\binom{213}{213})$}
        [{$(\hcomp,\binom{id}{id})$}
            [$*_{v_1}$]
            [$*_{v_2}$] 
        ]
    ]
]
\end{forest}
$\stackrel{2}{\leq}$
\begin{forest} for tree={parent anchor=south}
[{$(\contra{3}{3},\binom{21}{21})$}
    [{$(\hcomp,\binom{id}{id})$}
        [$*_{v_1}$]
        [{$(\contra{1}{2},\binom{id}{id})$}
            [$*_{v_2}$] 
        ]
    ]
]
\end{forest}
\end{center}
}
We have labelled each $\leq$ with the step of \cref{Total Admissible Order for Shuffle Tree Monomials of Wheeled Props} that orders each pair of tree monomials.
\end{example}

Before we prove this order is admissible, we need a simple helper lemma. 
\begin{lemma} \label{Nested Admissible is Admissible}
If $\leq_1,\leq_2$ are admissible orders, then the order $\leq$ defined by $\alpha \leq \beta$ if $\alpha<_1 \beta$ or $\alpha =_1 \beta$ and $\alpha\leq_2 \beta$, is also admissible.
\end{lemma}
\begin{proof}
For all $\alpha,\alpha',\beta,\beta'$ if $\alpha \leq \alpha'$ and $\beta \leq \beta'$ then each of the four possibilities (in terms of $\leq_1,\leq_2$) lead to $\alpha \circ_\varphi \beta \leq \alpha' \circ_\varphi \beta'$.
\end{proof}
With this result in hand, we now prove the following.
\begin{lemma}\label{The order for shuffle tree monomials of wheeled props is admissible}
The order of \cref{Total Admissible Order for Shuffle Tree Monomials of Wheeled Props} is total and admissible.
\end{lemma}
\begin{proof}
It is straightforward to see the order is total, as eventually the order compares all the data of the two shuffle tree monomials. It suffices to consider its admissibility. The comparisons of the order $\leq$ of steps $1-3$ is a path lexicographic order of \cite{dotsenko2010grobner} and hence is an admissible (partial) order. We now argue that tacking on step $4$ still results in admissible order. It is slightly easier to work with an alternate characterisation of admissibility,
\begin{align*}
    \forall \alpha,\alpha',\beta,\beta',\quad  (\alpha\leq \alpha' \text{ and } \beta \leq \beta' )&\implies \alpha \circ_\varphi \beta \leq \alpha' \circ_\varphi \beta'\\
    &\iff\\
    (\forall \alpha,\alpha',\beta,\quad \alpha\leq \alpha' \implies  \alpha \circ_\varphi \beta \leq \alpha' \circ_\varphi \beta )\quad &\text{ and } \quad(\forall \alpha,\beta,\beta'\quad
    \beta \leq \beta ' \implies \alpha \circ_\varphi \beta \leq \alpha \circ_\varphi \beta').
\end{align*}
This variant of the definition allows us to vary one side of the composition at a time.
Suppose that $\beta\leq\beta'$, and we wish to show that $\alpha \circ_\varphi \beta \leq \alpha \circ_\varphi \beta'$. Suppose that the first difference between $\beta$ and $\beta'$ occurs as a result of step $4$. That is to say the tree monomials $\beta$ and $\beta'$ have the same underlying 'shape' (as unary binary trees) but when comparing the permutations of the generators (bottom to top, left to right), all permutations are equal until we reach some particular vertex pair which satisfies $\binom{\sigma}{\tau}<\binom{\sigma'}{\tau'}$. 
\\\\
As $\beta$ and $\beta'$ have the same underlying 'shape' (as unary binary trees) this implies that $\alpha \circ_\varphi \beta$ and $\alpha \circ_\varphi \beta'$ have the same underlying shape. As such, we can see that the order on the generators of $\beta,\beta'$ (bottom to top, left to right) will naturally inject into the order on the generators of $\alpha \circ_\varphi \beta$ and $\alpha \circ_\varphi \beta'$. This means that the vertex pair of $\beta,\beta'$ causing $\binom{\sigma}{\tau}<\binom{\sigma'}{\tau'}$ will also induce $\alpha \circ_\varphi \beta<\alpha \circ_\varphi \beta'$ via step $4$ via the same vertex pair. A similar argument yields $\alpha\leq \alpha' \implies  \alpha \circ_\varphi \beta \leq \alpha' \circ_\varphi \beta$.
\\\\
Finally, to add in condition $5$ of \cref{Total Admissible Order for Shuffle Tree Monomials of Wheeled Props}, we first observe that performing the steps $1,5,3$ in order is a path lexicographic order of \cite{dotsenko2010grobner}, and hence is admissible. We then notice that performing steps $1-4$ followed by steps $1,5,3$ is just equivalent to performing steps $1-5$ (i.e. performing steps $1$ and $3$ again provides no new order information). As such steps $1-5$ are the composite of two admissible orders, and hence by \cref{Nested Admissible is Admissible}, is also an admissible order.
\end{proof}

We close this section by noting that, as observed in \cref{The Groupoid Coloured Extension}, we may use this total order on the underlying $ob(\SC)$-coloured tree monomials to order the relations of the groupoid coloured shuffle operad governing wheeled props. This induced order is precisely what has been used in \cref{The Relations of the Shuffle Operad Governing Wheeled Props}.

\subsection{The Operad Governing Wheeled Props Is Koszul} 
\label{proof The Operad Governing Wheeled Props is Koszul}

This section is devoted to the proof that the groupoid coloured operad $\mathbb{W}$ is Koszul.
In the prior section, we calculated an explicit presentation of the shuffle operad $\mathbb{W}^f= F_{sh}^{\SC}(E) / \langle R^f \rangle $.
We now apply \cref{Koszul if corresponding shuffle operad has QGB}, to show that $\mathbb{W}$ is Koszul, if $G$ is a quadratic Groebner basis for the $ob(\SC)$-coloured shuffle operad $F_{sh}^{ob(\SC)}(E^{f_{\SC}}) / \langle G \rangle$, where $G:=E_*^2 \sqcup (R^f)_*$.
We note that in our particular case that the relations of $(R^f)_*$ are precisely the pairs of canonical elements of the corresponding relations for the shuffle operad, and we will use $i_*$ to refer to the 'groupoid minimised' version of the $i$th relation in \cref{The Relations of the Shuffle Operad Governing Wheeled Props}.
We note that the relations of $E_*^2=\{ \alpha \circ_\varphi \beta \to [\alpha \circ_\varphi \beta]_* :\alpha, \beta \in E^{f_{\mathbb{V}}} \}$ are all by definition directed towards the minimal element. As an explicit example, from \cref{Example of Action of the Groupoid}, we see that $(\contra{2}{3},\binom{id}{id}) \circ_1 (\hcomp,\binom{213}{231}) \to (\contra{3}{3},\binom{21}{21}) \circ_1 (\hcomp,\binom{id}{id}) \in E_*^2$.

\begin{lemma} \label{Groebner Basis for Wheeled Props}
Under the order of \cref{Total Admissible Order for Shuffle Tree Monomials of Wheeled Props}, $G$ is a quadratic Groebner basis for $F_{sh}^{ob(\SC)}(E^{f_{\SC}}) / \langle G \rangle$.
\end{lemma}

\begin{proof}
We prove this lemma as follows.
\begin{itemize}
    \item In \cref{UMF Algorithm for Wheeled Props} we describe an algorithm that produces a unique minimal shuffle tree monomial encoding every graph. We show this algorithm is well-defined in \cref{UMF alg for wheeled props is well-defined}.
    \item We then prove every shuffle tree monomial which is not the unique minimal shuffle tree monomial of a wheeled graph admits a rewrite using $RS(G)$ (see \cref{Rewriting Systems and Groebner Bases}), establishing the confluence of the rewriting system on all shuffle tree monomials. This is done in \cref{Non Minimal Wheeled Prop Shuffle Tree Monomials Admit Rewrites}.
    \item Hence, through \cref{RS confluence iff GB}, this proves $G$ is a Groebner basis.
\end{itemize}
\end{proof}
The remainder of this section outlines the construction and the two needed lemmas.

\begin{construction} \label{UMF Algorithm for Wheeled Props}
The following algorithm produces the unique smallest shuffle tree monomial encoding a wheeled graph $\gamma$ with at least one vertex, we denote the outputted tree $UMF_{\wheel}(\gamma)$.
\begin{itemize}
    \item If $\gamma$ has just one vertex then let $T=v_1$, otherwise if $\gamma$ has $n>1$ vertices then let
    
    {\small
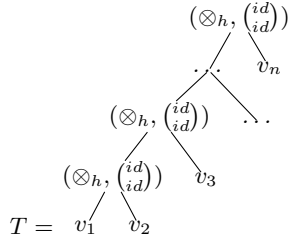
    \begin{center}
    $T =$ 
    \begin{forest}
    fairly nice empty nodes,
    for tree={inner sep=0, l=0}
    [{$(\hcomp,\binom{id}{id})$}
        [\dots
            [{$(\hcomp,\binom{id}{id})$}
                [{$(\hcomp,\binom{id}{id})$}
                    [{$v_1$}]
                    [{$v_2$}] 
                ]
                [{$v_3$}]
            ]
            [{\dots}]
        ]
        [{$v_n$}]
    ]
    \end{forest}    
    \end{center}
    }
     
    \item Let $\gamma_R = \gamma$.
    \item While there is an edge in $\gamma_R$:
    \begin{itemize}
        \item Pick the edge $e \in \gamma_R$ that uses the largest input in the graph outputted by $T$, and identify the necessary contraction indices $(i,j)$ to form this edge (in the graph outputted by $T$).
        \item Update $T = (\contra{j}{i},\binom{id}{id})(T)$.
        \item Remove the edge $e$ from $\gamma_R$.
    \end{itemize}
    \item If $T$ is not a vertex, update the permutations of the root-most operation of $T$ so that the corresponding graph outputted by $T$ has the same ordering on its open flags as $\gamma$.
    \item Return $T$.
\end{itemize}
\end{construction}

\begin{example} \label{UMF Wheeled Prop Example}
For the graph $\gamma$ in \cref{Total Order for Wheeled Props Example}, $UMF_{\wheel}(\gamma)$ is the smallest tree monomial of that example.
\end{example}

\begin{lemma} \label{UMF alg for wheeled props is well-defined}
The unique minimal form algorithm of \cref{UMF Algorithm for Wheeled Props} is well-defined.
\end{lemma}
\begin{proof}
It is straightforward to verify that this algorithm produces a (well-formed) shuffle tree monomial that produces the graph $\gamma$. We now verify that the algorithm produces the minimal shuffle tree encoding $\gamma$ by verifying each of the enumerated conditions of \cref{Total Admissible Order for Shuffle Tree Monomials of Wheeled Props}.
\begin{enumerate}
    \item Any tree monomial encoding a graph $\gamma$ with $n$ vertices must have arity $n$, so there is nothing to minimise here.
    \item The path from the $1$st input to the root is minimal as it is of maximal degree (all generators are above it and there cannot be another tree forming the same graph with a different number of generators) and all contraction generators come after horizontal compositions. Similarly, we can verify that the path from the $n$th input to the root is minimal, conditional on all paths from smaller inputs being fixed.
    \item The inputs of the shuffle tree monomials have already been minimised by the prior condition.
    \item Having all identity permutations except for the final generator is clearly minimal.
    \item The contractions needed to form the graph $\gamma$ were greedily added to the tree $T$ by the algorithm by their largest inputs first, and $\contra{j}{i}<\contra{j'}{i'}$ if $j>j'$.
\end{enumerate}
\end{proof}

\begin{lemma} \label{Non Minimal Wheeled Prop Shuffle Tree Monomials Admit Rewrites}
Let $T$ be a shuffle tree monomial for a directed graph $\gamma$ such that $T\neq UMF_{\wheel}(\gamma)$. 
Then, $T$ is rewritable by $G$.
\end{lemma}
\begin{proof}
Suppose that $T$ is larger than $UMF_{\wheel}(\gamma)$ as a result of not being normalised with respect to the action of the groupoid. Then there exists an internal edge of $T$ which sits above a non-identity permutation. This internal edge defines a corresponding action of the groupoid which we can translate into an element of $E^2_*$. This defines a corresponding rewrite (which will push the permutation up the tree). So given we have access to this rewrite for the remainder of this proof, we suppose that $T$ is normalised with respect to the action of the groupoid.
\\\\
Suppose that $T$ is larger than $UMF_{\wheel}(\gamma)$ as a result of having a contraction below a horizontal composition. Then $T$ must have the form of the left-hand side of one of the following equations, and hence must admit the corresponding rewrite.
{\small
\begin{center}

\begin{forest}
fairly nice empty nodes,
for tree={inner sep=0, l=0}
[{$T_u$}
    [{$(\hcomp,\binom{\sigma}{\tau})$}
        [{$(\contra{j}{i},\binom{id}{id})$}
            [{$T_1$}]
        ]
        [{$T_2$}]
    ]
]
\end{forest}
$\xrightarrow{3_*}$
\begin{forest}
fairly nice empty nodes,
for tree={inner sep=0, l=0}
[{$T_u$}
    [{$(\contra{j}{i},\binom{\sigma}{\tau} )$}
        [{$(\hcomp,\binom{id}{id})$}
            [{$T_1$}]
            [{$T_2$}]
        ]
    ]
]
\end{forest}
$\quad\quad$ or$\quad \quad$
\begin{forest}
fairly nice empty nodes,
for tree={inner sep=0, l=0}
[{$T_u$}
    [{$(\hcomp,\binom{\sigma}{\tau})$}
        [{$T_1$}]
        [{$(\contra{j}{i},\binom{id}{id})$}
            [{$T_2$}]
        ]
    ]
]
\end{forest}
$\xrightarrow{4_*}$
\begin{forest}
fairly nice empty nodes,
for tree={inner sep=0, l=0}
[{$T_u$}
    [{$(\contra{j+|\uc|}{i+|\ud|},\binom{\sigma}{\tau})$}
        [{$(\hcomp,\binom{id}{id})$}
            [{$T_1$}]
            [{$T_2$}]
        ]
    ]
]
\end{forest}
\end{center}
}
Here $T_1$ and $T_2$ are subtrees of $T$, $T_u$ is the remainder of the tree $T$ (close to root), in the case the second equation we also suppose that $T_1$ outputs a graph with profile $\dc$. We note that if $T_u$ is not empty, then $\binom{\sigma}{\tau}$ must be the identity permutations.
Suppose that $T$ is larger than $UMF_{\wheel}(\gamma)$ as a result of having a horizontal composition appear in the right terminal of another horizontal composition, then $T$ must admit the following form and rewrite.
{\small
\begin{center}
\begin{forest}
fairly nice empty nodes,
for tree={inner sep=0, l=0}
[{$T_u$}
    [{$(\hcomp,\binom{\sigma}{\tau})$}
        [{$T_1$}]
        [{$(\hcomp,\binom{id}{id})$}
            [{$T_2$}]
            [{$T_3$}] 
        ]
    ]
]
\end{forest}
$\xrightarrow{1_*}$
\begin{forest}
fairly nice empty nodes,
for tree={inner sep=0, l=0}
[{$T_u$}
    [{$(\hcomp,\binom{\sigma}{\tau})$}
        [{$(\hcomp,\binom{id}{id})$}
            [{$T_1$}]
            [{$T_2$}] 
        ]
        [{$T_3$}]
    ]
]
\end{forest}
\end{center}
}
Suppose that $T$ has all generators on the path from the minimal vertex to the root, but has a different order of its inputs than $UMF_\wheel(\gamma)$. Then $T$ must admit the following form and rewrite.
{\small
\begin{center}
$T=$
\begin{forest}
fairly nice empty nodes,
for tree={inner sep=0, l=0}
[{$T_u$}
    [{$(\hcomp,\binom{\sigma}{\tau})$}
        [{$(\hcomp,\binom{id}{id})$}
            [{$T_1$}]
            [{$c_j$}] 
        ]
        [{$c_i$}]
    ]
]
\end{forest}
$\xrightarrow{2_*}$
\begin{forest}
fairly nice empty nodes,
for tree={inner sep=0, l=0}
[{$T_u$}
    [{$(\hcomp,\binom{\sigma'}{\tau'})$}
        [{$(\hcomp,\binom{id}{id})$}
            [{$T_1$}]
            [{$c_i$}] 
        ]
        [{$c_j$}]
    ]
]
\end{forest}
\end{center}
}
Here $i<j$, and the alteration of $\binom{\sigma}{\tau}$ to  $\binom{\sigma'}{\tau'}$ is needed in general with this interchange of vertices. As $T$ is a shuffle tree monomial, $T_1$ must contain $c_1$. Finally, suppose that $T$ is larger than $UMF_\wheel(\gamma)$ as a result of using a different contraction sequence in the 'topmost' vertical segment. This implies there must exist two adjacent contractions such that the lower contraction uses input index $j$, the higher contraction uses input index $j'-1$ and $j< j'$. Then $T$ must admit either of the following forms and rewrites.

{\small
\begin{center}
If $i<i'$ \quad
\begin{forest}
for tree={parent anchor=south}
[$T_u$ [{$(\contra{j'-1}{i'-1},\binom{\sigma}{\tau})$}[{$(\contra{j}{i},\binom{id}{id})$} [$T_1$]]]]
\end{forest}
$\xrightarrow{5_*}$
\begin{forest}
for tree={parent anchor=south}
[$T_u$[{$(\contra{j}{i},\binom{\sigma}{\tau})$}[{$(\contra{j'}{i'},\binom{id}{id})$} [$T_1$]]]]
\end{forest}
, if $i>i'$ \quad
\begin{forest}
for tree={parent anchor=south}
[$T_u$[{$(\contra{j'-1}{i'},\binom{\sigma}{\tau})$}[{$(\contra{j}{i},\binom{id}{id})$} [$T_1$]]]]
\end{forest}
$\xrightarrow{6_*}$
\begin{forest}
for tree={parent anchor=south}
[$T_u$[{$(\contra{j}{i-1},\binom{\sigma}{\tau})$}[{$(\contra{j'}{i'},\binom{id}{id})$} [$T_1$]]]]
\end{forest}
\end{center}
}
We conclude this proof by noting that there are no further ways for $T$ to be greater than $UMF_{\wheel}(\gamma)$, hence either $T=UMF_{\wheel}(\gamma)$ or $T$ admits a rewrite.
\end{proof}

\section{The Operad Governing Props}\label{The Operad Governing Props}

We will now mirror the methodology of the preceding section to define a groupoid coloured operad $\mathbb{P}$ governing props.
We then show this operad is Koszul (\cref{proof The Operad Governing Props is Koszul}). We will assume that the reader has read (or is reading in parallel) \cref{The Operad Governing Wheeled Props}, so that we may streamline identical proofs and focus on technical difficulties arising from the nature of props.

\subsection{Props}

We wish to define an alternate biased prop which will induce a nice definition of a quadratic groupoid coloured operad governing props. Such a definition is not as straightforward to produce as the alternate definition for wheeled props (\cref{Alternate biased wheeled prop}). As such, we first discuss the necessary features of an 'alternate biased prop', before presenting it (\cref{Alternate biased prop}) and proving it is equivalent to a standard definition of a prop (Definition 11.30 of \cite{yau2015foundation}, through \cref{definitions of props are equivalent}). Finally, we will discuss non-unital and augmented variants of this definition.
\\\\
Our alternate definition of a prop should use the same extended horizontal composition as the alternate wheeled prop (\cref{Alternate biased wheeled prop}), and then have some way to connect flags of graphs without the possibility of forming wheels. The standard vertical composition of props which connects all outputs of one graph to all inputs of another (see for instance Definition 11.30 of \cite{yau2015foundation}) is not a good candidate for at least two reasons. Firstly, a non-unital version of such a prop would not be able to form graphs like the Bow graph.
\begin{center}\label{non-unital standard prop counter example}
\begin{tikzpicture}[thick,
v/.style={circle, draw}] 
\node[v] (1) {$v_1$}; 
\node[v] (2) [below right of=1] {$v_2$};
\node[v] (3) [below left of=2] {$v_3$};
\draw (2) to (1);
\draw (3) to (1);
\draw (3) to (2);
\end{tikzpicture}  
\end{center}
This is similar to how a non-unital May operad, defined via $\gamma:O(n)\times O(k_1) \times ... O(k_n) \to O(k_1+...+k_n)$, is incapable of forming a partial operadic composition $\circ_i:O(n)\times O(k)\to O(n+k-1)$ (see \cite{markl2008operads}).
Secondly, props with the standard vertical composition satisfy the interchange axiom (11.26 of \cite{yau2015foundation}) which is a ternary relation, and hence would not induce a quadratic presentation of our operad.
As such, instead of seeking to modify vertical composition, we will instead define an extended properadic composition, which we now discuss the motivation for. 
\\\\
In \cite{yau2015foundation} two biased definitions of properads are provided. The first, Definition 11.25, uses a properadic composition of the form 
\begin{align} \label{contiguous properadic compsoition}
    P\dc \otimes P\ba \xrightarrow{\properadic{\cbprime}} P\binom{\ub\circ_{\ub'}\ud}{\uc \circ_{\uc'}\ua}
\end{align}
This composition connects two graphs via contiguous $k$-segments $\uc'\subset \uc$ and $\ub'\subset \ub$.
By $\uc'$ is a contiguous $k$-segment of $\uc$ (Definition 1.5 of \cite{yau2015foundation}), we mean that $\uc = (c_1,...,c_{i-1},\uc',c_{i+k+1},...,c_{|\uc|})$, and $\uc \circ_{\uc'}\ua$ is syntactic sugar denoting $\uc \circ_{\uc'}\ua:= (c_1,...,c_{i-1},\ua,c_{i+k+1},...,c_{|\uc|})$.
The second definition of a properad in \cite{yau2015foundation}, Definition 11.27, uses a properadic composition which is an extension of the prior.
\begin{align}\label{discontiguous properadic compsoition}
    P\dc \otimes P\ba \xrightarrow{\properadic{\cbprime,\binom{\sigma}{\tau}}{}} P\binom{\sigma\ub\circ_{\ub'}\ud}{\uc\tau \circ_{\uc'}\ua}, \quad\text{ satisfying }\quad
    \alpha \properadic{\cbprime,\binom{\sigma}{\tau}}{} \beta := (\alpha \cdot \binom{id}{\tau}) \properadic{\binom{\uc'}{\ub'}}{}{} (\beta \cdot \binom{\sigma}{id})
\end{align}
This composition connects two graphs via the contiguous $k$-segments $\uc'\subset \uc \tau$ and $\ub'\subset \sigma \ub$. 
This in effect allows the composition of two graphs via potentially dis-contiguous segments via first permuting the dis-contiguous segments into contiguous segments. 
In particular we could form any connected graph with two vertices $\gamma$, by first forming the internal edges of $\gamma$ via \cref{discontiguous properadic compsoition}, and then producing any listing of the remaining flags via an action of the bimodule.
We will use the following diagram to represent the graph $\gamma$.
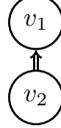
\begin{center}
\begin{tikzpicture}[node distance={10mm}, thick,
v/.style={circle, draw}] 
\node[v] (1) {$v_1$}; 
\node[v] (2) [below of=1] {$v_2$};
\draw [-{Implies},double] (2) to [out=90,in=-90,looseness =3] (1);
\end{tikzpicture}   
\end{center}
Unfortunately, $\properadic{\cbprime,\binom{\sigma}{\tau}}$ satisfies inner biequivariance axioms (see Definition 11.27 of \cite{yau2015foundation}) such as
\begin{center}
\begin{tikzpicture}[thick,
v/.style={circle, draw}] 
\node[v] (1) {$v_1$}; 
\node[flag] (2) [below left =0.3cm of 1] {$c_1$};
\node[flag] (3) [below right =0.3cm of 1] {$c_2$};
\draw (1) to (2);
\draw (1) to (3);
\end{tikzpicture}
\raisebox{0.75cm}{
$\properadic{\binom{c_1,c_2}{c_1,c_2},\binom{id}{id}}{}$
}
\raisebox{0.2cm}{
\begin{tikzpicture}[thick,
v/.style={circle, draw}] 
\node[v] (1) {$v_1$}; 
\node[flag] (2) [above left =0.3cm of 1] {$c_1$};
\node[flag] (3) [above right =0.3cm of 1] {$c_2$};
\draw (1) to (2);
\draw (1) to (3);
\end{tikzpicture}
}
\raisebox{0.75cm}{$=$}
\begin{tikzpicture}[node distance={10mm}, thick,
v/.style={circle, draw}] 
\node[v] (1) {$v_1$}; 
\node[v] (2) [below of=1] {$v_2$};
\draw (2) to [bend left] (1);
\draw (2) to [bend right] (1);
\end{tikzpicture}
\raisebox{0.75cm}{$=$}
\begin{tikzpicture}[thick,
v/.style={circle, draw}] 
\node[v] (1) {$v_1$}; 
\node[flag] (2) [below left =0.3cm of 1] {$c_1$};
\node[flag] (3) [below right =0.3cm of 1] {$c_2$};
\draw (1) to (2);
\draw (1) to (3);
\end{tikzpicture}
\raisebox{0.75cm}{
$\properadic{\binom{c_2,c_1}{c_2,c_1},\binom{21}{21}}{}$
}
\raisebox{0.2cm}{
\begin{tikzpicture}[thick,
v/.style={circle, draw}] 
\node[v] (1) {$v_1$}; 
\node[flag] (2) [above left =0.3cm of 1] {$c_1$};
\node[flag] (3) [above right =0.3cm of 1] {$c_2$};
\draw (1) to (2);
\draw (1) to (3);
\end{tikzpicture}
}
\end{center}
Hence, if we were to take this as our definition of properadic composition, then the operad governing props would satisfy unary relations.
Thus, for any connected graph with two vertices $\gamma$, we wish to identify a unique way of forming $\gamma$ from these vertices (with their specific listings, \cref{Wheeled Graph}).
Or combinatorially, we wish to uniquely identify: an instance of \cref{discontiguous properadic compsoition} which forms the $k$ connected edges of $\gamma$ by connecting the potentially dis-contiguous $k$-segments $\uc'\subseteq \uc$ and $\ub'\subseteq \ub$, subject to $\uc' = \ub'$; 
and a pair of permutations which produces the specific listing of the remaining open inputs and outputs of $\gamma$.
\begin{itemize}
    \item We assume $\uc'\subseteq \uc$ admits the reduced order of $\uc$, i.e. if $c_i,c_j\in \uc'$ and $c_i<c_j$ in $\uc$, then $c_i<c_j$ in $\uc'$.
    If this was not the case, then there exists a permutation $\sigma$ such that $\uc' \sigma$ admits the reduced order, and we could proceed with this pair instead (clearly $\uc' \sigma = \ub' \sigma$).
    \item The properadic composition of \cref{contiguous properadic compsoition}, is defined whenever both $\uc'$ and $\ub'$ are contiguous, thus we could apply any pair of permutations $\tau, \sigma$ such that $\uc\tau = (\underline{c}_l, \uc', \underline{c_r})$ and $\sigma \ub = (\underline{b}_l, \ub', \underline{b_r})$.
    Of these possibilities, we choose the unique permutations $\tau^{\uc}_{\uc'},\sigma^{\ub}_{\ub'}$ such that $\uc\tau^{\uc}_{\uc'} = (\uc',\uc\setminus \uc')$ and $\sigma^{\ub}_{\ub'} \ub = (\ub',\ub\setminus \ub')$,
    i.e. we choose the unique permutations which pull the connecting segments to the left, and conserve the order of the remaining elements.
    
    \item Hence by the definition of $\circ_{\uc'}$ below \cref{contiguous properadic compsoition}, $\uc\tau^{\uc}_{\uc'} \circ_{\uc'} \ua = (\ua,\uc\setminus \uc')$ and $\sigma^{\ub}_{\ub'} \ub \circ_{\ub'} \ud = (\ud, \ub\setminus \ub')$.
    \item Finally, there exist unique permutations $\binom{\sigma}{\tau}$ producing any listing of the remaining inputs+outputs.
\end{itemize}
That is to say we identify our properadic composition as the composite,
\begin{center}
\begin{tikzcd}[column sep = 2.5cm]\label{extended properadic as a composite}
P\dc \otimes P\ba \arrow[r, "{\binom{id}{\tau^{\uc}_{\uc'}}\otimes \binom{\sigma^{\ub}_{\ub'}}{id}}"]
\arrow[rdd,"{\eproperadic{\cbprime}{\binom{\sigma}{\tau}}}"']
& P\binom{\ud}{\uc',\uc\setminus \uc'}\otimes P\binom{\ub',\ub\setminus \ub'}{\ua}
\arrow[d, "{\properadic{\cbprime}}"]\\
& P\binom{\ud, \ub\setminus \ub'}{\ua,\uc\setminus \uc'}
\arrow[d, "{\binom{\sigma}{\tau}}"]\\
& P\binom{\sigma(\ud, \ub\setminus \ub')}{(\ua,\uc\setminus \uc')\tau}
\end{tikzcd}
\end{center}

Note, in \cref{Alternate biased prop}, we do not define our composition to be equal to this composite, but rather alter the underlying biequivariance axioms so it behaves in this fashion (see \cref{restriction is an alternate definition of a properad}).
We also note that other ways of producing a representative are possible, such as giving $\ub'$ the reduced order of $\ub$, or pulling the connecting segments to the right.
We only care that there is a unique way to form each connected graph with two vertices via our defined composition.
We illustrate this with the following example.

\begin{example} \label{extended properadic composition example}
Recalling that we use the line permutation notation, we compute
{\tiny
\begin{align*}
\begin{tikzpicture}[node distance={6mm},thick] 
\node[flag] (v1o1) {$d_1$};
\node[vertex] (v1) [below of=v1o1]{$v_1$};
\node[flag] [below of=v1] (v1i2) {$c_2$};
\node[flag] [left of=v1i2] (v1i1) {$c_1$};
\node[flag] [right of=v1i2] (v1i3) {$c_3$};
\draw [black] (v1) to (v1o1);
\draw [black] (v1) to (v1i1);
\draw [black] (v1) to (v1i2);
\draw [black] (v1) to (v1i3);
\end{tikzpicture}
\eproperadic{\binom{c_1,c_3}{b_3,b_2}}{\binom{21}{id}}
\begin{tikzpicture}[node distance={6mm},thick] 
\node[flag] (v1o1) {$b_1$};
\node[flag][right of=v1o1] (v1o2) {$b_2$};
\node[flag][right of=v1o2] (v1o3) {$b_3$};
\node[vertex] (v1) [below of=v1o2]{$v_1$};
\node[flag] [below of=v1] (v1i1) {$a_1$};
\draw [black] (v1) to (v1o1);
\draw [black] (v1) to (v1o2);
\draw [black] (v1) to (v1o3);
\draw [black] (v1) to (v1i1);
\end{tikzpicture}
=
(
\begin{tikzpicture}[node distance={6mm},thick] 
\node[flag] (v1o1) {$d_1$};
\node[vertex] (v1) [below of=v1o1]{$v_1$};
\node[flag] [below of=v1] (v1i2) {$c_2$};
\node[flag] [left of=v1i2] (v1i1) {$c_1$};
\node[flag] [right of=v1i2] (v1i3) {$c_3$};
\draw [black] (v1) to (v1o1);
\draw [black] (v1) to (v1i1);
\draw [black] (v1) to (v1i2);
\draw [black] (v1) to (v1i3);
\end{tikzpicture}
\cdot\binom{id}{132})
\properadic{\binom{c_1,c_3}{b_3,b_2}}
(
\begin{tikzpicture}[node distance={6mm},thick] 
\node[flag] (v1o1) {$b_1$};
\node[flag][right of=v1o1] (v1o2) {$b_2$};
\node[flag][right of=v1o2] (v1o3) {$b_3$};
\node[vertex] (v1) [below of=v1o2]{$v_1$};
\node[flag] [below of=v1] (v1i1) {$a_1$};
\draw [black] (v1) to (v1o1);
\draw [black] (v1) to (v1o2);
\draw [black] (v1) to (v1o3);
\draw [black] (v1) to (v1i1);
\end{tikzpicture}
\cdot {\binom{321}{id}}
)
\cdot {\binom{21}{id}}
=
\begin{tikzpicture}[node distance={4.7mm},thick] 
\node[flag] (go1) {$b_1$};
\coordinate[right of=go1] (spacego);
\node[flag] (go2) [right of=spacego]{$d_1$};
\coordinate[below of=go1] (v1o1);
\coordinate[below of=go2] (go2l);
\node[vertex] (v1) [below of=v1o1]{$v_1$};
\coordinate[below of=v1] (v1i2);
\coordinate[left of=v1i2] (v1i1);
\coordinate[right of=v1i2] (v1i3);
\coordinate[below of=v1i1] (v1i1l);
\coordinate[below of=v1i2] (v1i2l);
\coordinate[below of=v1i3] (v1i3l);
\coordinate[below of=v1i1l] (v2o1u);
\coordinate[below of=v1i2l] (v2o2u);
\coordinate[below of=v1i3l] (v2o3u);
\coordinate[below of=v2o1u] (v2o1);
\coordinate[below of=v2o2u] (v2o2);
\coordinate[below of=v2o3u] (v2o3);
\node[vertex] (v2) [below of=v2o2]{$v_2$};
\node[flag] (gi1) [below of=v2]{$a_1$};
\coordinate[right of=gi1] (spacegi);
\node[flag] (gi2) [right of=spacegi]{$c_2$};
\draw [black] (go2) to (v1o1);
\draw [black] (go1) to (go2l);
\draw [black] (v1) to (v1o1);
\draw [black] (v1) to (v1i1);
\draw [black] (v1) to (v1i2);
\draw [black] (v1) to (v1i3);
\draw [black] (v1i1) to (v1i1l);
\draw [black] (v1i2) to (v1i3l);
\draw [black] (v1i3) to (v1i2l);
\draw [black] (v1i1l) to (v2o1u);
\draw [black] (v1i2l) to (v2o2u);
\draw [black] (v1i3l) to (gi2);
\draw [black] (go2l) to (v2o3u);
\draw [black] (v2o1u) to (v2o3);
\draw [black] (v2o2u) to (v2o2);
\draw [black] (v2o3u) to (v2o1);
\draw [black] (v2o1) to (v2);
\draw [black] (v2o2) to (v2);
\draw [black] (v2o3) to (v2);
\draw [black] (v2) to (gi1);
\end{tikzpicture} 
\end{align*}
}
We verify that there is a unique composite $\alpha {\eproperadic{\cbprime}{\binom{\sigma}{\tau}}}\beta$ forming this graph $\gamma$.
In particular:
\begin{itemize}
    \item $\alpha$ must be the top vertex and $\beta$ the lower (i.e inputs of $\alpha$ are connected to outputs of $\beta$); 
    \item we identify the subset of flags $\uc'=\{c_1,c_3\}=\{b_3,b_2\} = \ub'$ being connected by the properadic join;
    \item the listing of $\uc' = (c_1,c_3)$ is then determined by the listing of $\alpha$, as $\uc'$ admits the reduced order of $\uc$ (recall each vertex has specific listing data, \cref{Wheeled Graph});
    \item the listing of $\uc'$ determines the listing of $\ub'=(b_3,b_2)$;
    \item this completely determines $\properadic{\cbprime,\binom{\sigma}{\tau}}{} = \properadic{\binom{c_1,c_3}{b_3,b_2},\binom{132}{321}}{}$; and
    \item the final permutation $\binom{\sigma}{\tau}=\binom{21}{id}$ is determined by the relative listing of $\gamma$ to $\alpha \properadic{\binom{c_1,c_3}{b_3,b_2},\binom{132}{321}}{} \beta$.
\end{itemize}

\end{example}

We now present our definition.
\begin{definition}\label{Alternate biased prop}
Let $\mathcal{E}$ be a symmetric monoidal category, a \textbf{alternate prop} $P$ consists of the following.
\begin{enumerate}
    \item A bimodule $P\in \mathcal{E}^{\SC}$.
    \item An extended horizontal composition,
    \begin{align*}
        P\dc \otimes P\ba \xrightarrow{\hcomp,\binom{\sigma}{\tau}} P\binom{\sigma(\ud,\ub)}{(\uc , \ua)\tau}.
    \end{align*}
    \item An alternate properadic composition,
    \begin{align*}
        P\dc \otimes P\ba \xrightarrow{\eproperadic{\binom{\uc'}{\ub'}}{\binom{\sigma}{\tau}}} P\binom{\sigma(\ud,\ub\setminus \ub')}{(\ua,\uc\setminus \uc')\tau},
    \end{align*}
    where $\uc'\subseteq \uc$, $\ub'\subseteq \ub$, $\uc' = \ub'$ and $\uc'$ admits the reduction of the order of $\uc$.
    \item An opposing alternate properadic composition defined by $\eproperadicop{\binom{\uc'}{\ub'}}{\binom{\sigma}{\tau}}(\alpha,\beta):=\eproperadic{\binom{\uc'}{\ub'}}{\binom{\sigma}{\tau}}(\beta,\alpha)$.
    \item Units for properadic and horizontal composition.
\end{enumerate}
In situations where the permutations and indices are either arbitrary or clear, we will use $\hcomp$, $\eproperadicunspecified$ and $\eproperadicopunspecified$ to refer to these compositions. The opposing extended properadic composition $\eproperadicopunspecified$ does not increase the expressive power of this definition of a prop, but rather exists to be paired off with $\eproperadicunspecified$ under a $\Sigma_2$-action in our latter definition of the operad governing props. 
\\\\
We now present the axioms satisfied by these compositions. We will not express the various unity diagrams (as we will quickly move to using a non-unital definition) but the obvious extensions to (Definitions $11.21$ and $11.9$ of \cite{yau2015foundation}) are also present. We have the same biequivariance and associativity axioms of $\hcomp$ as for wheeled props (\cref{Alternate biased wheeled prop}). The other compositions are right, left and switch biequivariant as follows, (the axioms for $\eproperadicopunspecified$ are obvious variants of $\eproperadicunspecified$)
\begin{center}
\begin{tikzcd}
P\dc \otimes P\ba
\arrow[d,"{\eproperadic{\binom{\uc'}{\ub'}}{\binom{\sigma}{\tau}}}"']
\arrow[rd,"{\eproperadic{\binom{\uc'}{\ub'}}{\binom{\sigma'\sigma}{\tau\tau'}}}"]\\
P\binom{\sigma(\ud,\ub)}{(\uc , \ua)\tau}   \arrow[r,"{\binom{\sigma'}{\tau'}}"']& P\binom{\sigma' \sigma(\ud,\ub)}{(\uc , \ua)\tau \tau'} 
\end{tikzcd}
$\quad$
\begin{tikzcd}
P\dc \otimes P\ba
\arrow[d,"{\binom{\sigma_1}{\tau_1}\otimes\binom{\sigma_2}{\tau_2} }"']
\arrow[rd,"{\eproperadic{\binom{\uc''}{\ub''}}{\binom{\sigma'}{\tau'}}}"]\\
P\binom{\sigma_1\ud}{\uc \tau_1} \otimes P\binom{\sigma_2\ub}{\ua \tau_2}  \arrow[r,"{\eproperadic{\binom{\uc'}{\ub'}}{\binom{\sigma}{\tau}}}"']& P\binom{\sigma(\sigma_1 \ud, \sigma_2 \ub \setminus \ub')}{(\ua \tau_2,\uc \tau_1 \setminus \uc')\tau}  
\end{tikzcd}
$\quad$
\begin{tikzcd}
P\dc \otimes P\ba
\arrow[d,"switch"']
\arrow[rd,"{\eproperadic{\binom{\uc'}{\ub'}}{\binom{\sigma}{\tau}}}"]\\
P\ba \otimes P \dc \arrow[r,"{\eproperadicop{\binom{\uc'}{\ub'}}{\binom{\sigma}{\tau}}}"']& P\binom{\sigma(\ud,\ub\setminus \ub')}{(\ua,\uc\setminus \uc')\tau}
\end{tikzcd}
\end{center}
In the second diagram, if $\sigma_{\uc}$ is the unique permutation that gives $\uc' (\tau^{-1}|_{\uc'})$ the same reduced order as $\uc$, then $\uc'' := \uc' (\tau^{-1}|_{\uc'})\sigma_{\uc}$,
which by construction has the same reduced order as $\uc$. Then (in order to connect the exact same wires), we define $\ub'' := \ub' (\tau^{-1}|_{\uc'}) \sigma_{\uc}$.
Notice that in general $\ub''$ will be a dis-contiguous subset of $\ub$.
Finally, in order to match permutations, we define $\sigma',\tau'$ such that $\binom{\sigma(\sigma_1 \ud, \sigma_2 \ub \setminus \ub')}{(\ua \tau_2,\uc \tau_1 \setminus \uc')\tau} = \binom{\sigma'(\ud, \ub \setminus \ub'')}{(\ua ,\uc \setminus \uc'')\tau'}$.
To place the diagrams in context, the left and right compatibility diagram encode similar information to the biequivariance diagram of Definition 11.27 of \cite{yau2015foundation}, and the switch diagram is the definition of $\eproperadicopunspecified$ restated in diagram form. 
\\\\
With these axioms we can prove \cref{simple canonical form of alternate props}, and use it to simplify the remaining axioms. These include the four associativity diagrams of properadic composition, extended in the obvious way (see Definition 11.25 of \cite{yau2015foundation}). We shall refer to these diagrams as the Caterpillar, Lighthouse, Fireworks and Bow diagrams following the conventions of \cite{yau2015foundation} (sometimes abbreviated using the capitalised first letter where obvious). Finally, a new axiom regarding compatibility of $\hcomp$ and $\eproperadicunspecified$ is required, which we refer to as the Third-wheel. 
\\\\
These relations are encoded graphically in \cref{The Relations of the Operad Governing Props}, however the interested reader may also find them explicitly written out in \cref{axioms of the alternate prop}. 
This table uses the same notation as introduced in \cref{The Relations of the Operad Governing Wheeled Props}. 
In other words, each composite of operations is given under the implicit understanding that it is forming the graph on the left. 
As we may use the biequivariance axioms to push all permutations to the top (\cref{simple canonical form of alternate props}), this indicates that the permutations of the operations and the sub-profiles of any properadic joins are uniquely specified by the graph. As such, we may use a suppressed notation where we omit the permutations and the specific indices used by each binary operation.
\end{definition}

As was the case for wheeled props (\cref{simple canonical form of alternate wheeled props}), the equivariance axioms yield the following.
\begin{proposition} \label{simple canonical form of alternate props}
Every valid composite of operations in \cref{Alternate biased prop} admits a simple canonical form, where all non-identity permutations are pushed to the outermost operation.
\end{proposition}
This result not only lets us simplify the other axioms of \cref{Alternate biased prop}, but it will also induce a simple canonical form for every groupoid coloured tree monomial in $\mathbb{P}$ (\cref{biased presentation of the operad governing props}).
We now prove our definition is indeed a prop, by establishing equivalence with a known definition.

\begin{proposition}\label{definitions of props are equivalent}
\cref{Alternate biased prop} and Definition 11.30 of \cite{yau2015foundation} are equivalent definitions of props.
\end{proposition}
\begin{proof}
The equivalence of the underling operations is easy to establish. The vertical composition (which we denote) $\vcomp$, and the horizontal composition $\hcomp$ of Definition 11.30 of \cite{yau2015foundation}, are particular instances of the extended compositions, i.e. $\vcomp: P\binom{\uc}{\ub}\otimes P \binom{\ub}{\ua}\to P \binom{\uc}{\ua}$ is equal to $\eproperadic{\binom{\ub}{\ub}}{\binom{id}{id}}$, and $\hcomp = (\hcomp,\binom{id}{id})$. In the other direction, 
the extended horizontal composition $(\hcomp,\binom{\sigma}{\tau})$ can be obtained by $\hcomp$ followed by an action of the bimodule (\cref{Bimodules}).
The extended properadic composition can be obtained through the following commutative diagram,
\begin{center}
\begin{tikzcd}[column sep = 3cm] \label{eproperadic through vertical comp}
P\dc \otimes P\ba \arrow[r, "\text{pad with units}"]
\arrow[rdddd,dotted, "{\eproperadic{\binom{\uc'}{\ub'}}{\binom{\sigma}{\tau}}}"']
& P\binom{\ud, \ub\setminus \ub'}{\uc,\ub\setminus \ub'} \otimes P\binom{\ub,\uc\setminus \uc'}{\ua,\uc\setminus \uc'} \arrow[d, "{\binom{id}{\sigma'} \otimes \binom{\tau'}{id}}"]\\
&P\binom{\ud, \ub\setminus \ub'}{\uc',\uc\setminus \uc',\ub\setminus \ub'} \otimes P\binom{\ub',\uc\setminus \uc',\ub\setminus \ub'}{\ua,\uc\setminus \uc'}
\arrow[d,"{\circ_V}"]\\
&P\binom{\ud, \ub\setminus \ub'}{\ua,\uc\setminus \uc'} \arrow[d, "\text{remove dangling units}"]\\
&P\binom{\ud, \ub\setminus \ub'}{\ua,\uc\setminus \uc'} \arrow[d, "{\binom{\sigma}{\tau}}"]\\
&P\binom{\sigma(\ud, \ub\setminus \ub')}{(\ua,\uc\setminus \uc')\tau}
\end{tikzcd}
\end{center}
where $\sigma',\tau'$ are the unique permutations (outputting the target profiles) which move the connecting segments $\uc',\ub'$ to the left (recall that $\uc'$ is ordered by $\uc$, and $\ub'$ by $\uc'$), and align the remaining flags with an identity.
\\\\
The axioms of Definition 11.30 of \cite{yau2015foundation} arise from the axioms of the alternate definition as follows
\begin{itemize}
    \item The axioms regarding the extended horizontal composition can be restricted to the standard horizontal composition (associativity, biequivariance and unity).
    \item The axioms regarding the extended properadic composition can be restricted down to the vertical compositions (associativity, biequivariance and unity).
    \item The interchange axiom (11.26) is obtained as follows. As vertical composition is an instance of extended properadic composition, the interchange axiom corresponds to the following graph. 
    
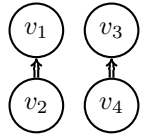
\begin{center}
    \begin{tikzpicture}[node distance={10mm}, thick,
    v/.style={circle, draw}] 
    \node[v] (1) {$v_1$}; 
    \node[v] (2) [below of=1] {$v_2$};
    \node[v] (3) [right of=1] {$v_3$};
    \node[v] (4) [right of=2] {$v_4$};
    \draw [-{Implies},double] (2) to [out=90,in=-90,looseness =3] (1);
    \draw [-{Implies},double] (4) to [out=90,in=-90,looseness =3] (3);
    \end{tikzpicture}   
    \end{center}
    With this diagram as a guide we then apply a chain of axioms of our alternate definition to reproduce the interchange axiom $\hcomp(\vcomp(v_1,v_2),\vcomp(v_3,v_4))=\vcomp(\hcomp(v_1,v_3),\hcomp(v_2,v_4))$.
    To aid readability, an underline is used to group anything that is treated as a 'single variable' by an equation.
    \begin{align*}
        \vcomp(\underline{\hcomp(v_1,v_3)},\hcomp(v_2,v_4))&=\vcomp({\vcomp(\hcomp(v_1,v_3),v_2)},v_4), &(L)\\
        &=\vcomp(\hcomp(\underline{\vcomp(v_1,v_2)},v_3),v_4), &(T)\\
        &=\vcomp(\hcomp(v_3,\underline{\vcomp(v_1,v_2)}),v_4), &(\text{Bi-eq } \hcomp)\\
        &=\hcomp(\underline{\vcomp(v_3,v_4)},\underline{\vcomp(v_1,v_2)}), &(T)\\
        &=\hcomp(\vcomp(v_1,v_2),\vcomp(v_3,v_4)), &(\text{Bi-eq } \hcomp)
    \end{align*}
    We have suppressed the permutations above, but there are non-identity permutations in the second to fourth line that for instance force the output flags of $v_1$ to appear before the output flags of $v_3$ and similar. The needed permutations are reverted to identities by the final biequivariance of $\hcomp$.
\end{itemize}
The axioms of our alternate definition of a prop arise from the axioms of Definition 11.30 \cite{yau2015foundation} as follows.
\begin{itemize}
    \item The axioms which just involve horizontal composition arise as described in wheeled props (\cref{Alternate biased wheeled prop}).
    \item The biequivariance axioms of $\eproperadicunspecified,\eproperadicopunspecified$ arise from the biequivariance diagram of $\circ_V$ (Diagram $11.23$ \cite{yau2015foundation}). This is a straightforward but tedious exercise. Start with one side of the biequivariance diagram, and blow up $\eproperadicunspecified$ using its $\circ_V$ commutative diagram description, then manipulate it into the other side through unit axioms, biequivariance axioms (of $\circ_V$), and actions of the bimodule. 
    
    \item The caterpillar diagram arises from the associativity relation of vertical composition, as the extended properadic composition is defined from vertical composition (with appropriate padding with units).
    \item The third-wheel axiom arises from the interchange axiom and the vertical unit. The corresponding diagram illustrating this result is
    
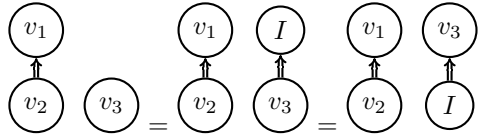
\begin{center}
    \begin{tikzpicture}[node distance={10mm}, thick,
    v/.style={circle, draw}] 
    \node[v] (1) {$v_1$}; 
    \node[v] (2) [below of=1] {$v_2$};
    \node[v] (4) [right of=2] {$v_3$};
    \draw [-{Implies},double] (2) to [out=90,in=-90,looseness =3] (1);
    \end{tikzpicture}   
    =
    \begin{tikzpicture}[node distance={10mm}, thick,
    v/.style={circle, draw}] 
    \node[v] (1) {$v_1$}; 
    \node[v] (2) [below of=1] {$v_2$};
    \node[v] (3) [right of=1] {$I$};
    \node[v] (4) [right of=2] {$v_3$};
    \draw [-{Implies},double] (2) to [out=90,in=-90,looseness =3] (1);
    \draw [-{Implies},double] (4) to [out=90,in=-90,looseness =3] (3);
    \end{tikzpicture}
    =
    \begin{tikzpicture}[node distance={10mm}, thick,
    v/.style={circle, draw}] 
    \node[v] (1) {$v_1$}; 
    \node[v] (2) [below of=1] {$v_2$};
    \node[v] (3) [right of=1] {$v_3$};
    \node[v] (4) [right of=2] {$I$};
    \draw [-{Implies},double] (2) to [out=90,in=-90,looseness =3] (1);
    \draw [-{Implies},double] (4) to [out=90,in=-90,looseness =3] (3);
    \end{tikzpicture}   
    \end{center}
    \item The caterpillar axiom, the interchange axiom and vertical units can collectively be used to construct the lighthouse, fireworks and bow diagrams. Here is an outline of how the lighthouse axiom arises, and the others are similar. In the diagram below, the vertices labelled $I$ denote any appropriate block of units.
    
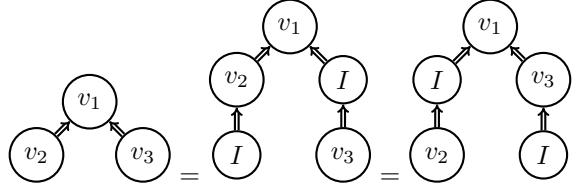
\begin{center}
    \begin{tikzpicture}[thick,
    v/.style={circle, draw}] 
    \node[v] (1) {$v_1$}; 
    \node[v] (2) [below left of=1] {$v_2$};
    \node[v] (3) [below right of=1] {$v_3$};
    \draw [-{Implies},double] (2) to (1);
    \draw [-{Implies},double] (3) to (1);
    \end{tikzpicture}
    =
    \begin{tikzpicture}[thick,
    v/.style={circle, draw}] 
    \node[v] (1) {$v_1$}; 
    \node[v] (2) [below left of=1] {$v_2$};
    \node[v] (3) [below right of=1] {$I$};
    \node[v] (4) [below of=2] {$I$};
    \node[v] (5) [below of=3] {$v_3$};
    \draw [-{Implies},double] (2) to (1);
    \draw [-{Implies},double] (3) to (1);
    \draw [-{Implies},double] (4) to (2);
    \draw [-{Implies},double] (5) to (3);
    \end{tikzpicture}
    =
    \begin{tikzpicture}[thick,
    v/.style={circle, draw}] 
    \node[v] (1) {$v_1$}; 
    \node[v] (2) [below left of=1] {$I$};
    \node[v] (3) [below right of=1] {$v_3$};
    \node[v] (4) [below of=2] {$v_2$};
    \node[v] (5) [below of=3] {$I$};
    \draw [-{Implies},double] (2) to (1);
    \draw [-{Implies},double] (3) to (1);
    \draw [-{Implies},double] (4) to (2);
    \draw [-{Implies},double] (5) to (3);
    \end{tikzpicture}
    \end{center}
    Then the first equality of diagrams corresponds to,
    \begin{align*}
        \circ_v(v_1,\hcomp(v_2,v_3)) &= \circ_v(v_1,\hcomp(\circ_v(v_2,I),\circ_v(I,v_3))) &(\text{Units})\\
        &= \circ_v(v_1, \circ_v(\hcomp(v_2,I),\hcomp(I,v_3))) &(\text{Interchange})\\
        &= \circ_v(\circ_v(v_1,\hcomp(v_2,I)),\hcomp(I,v_3)) &(\text{Caterpillar})\\
        &= \circ_v(\circ_v(v_1,v_2),v_3) &(\text{unit})
    \end{align*}
    and the third diagram can be interpreted similarly. 
\end{itemize}

\end{proof}

\begin{remark}\label{restriction is an alternate definition of a properad}
The restriction of \cref{Alternate biased prop} to just include the properadic composition and the connected axioms, also provides an equivalent definition of a coloured properad (Definition 11.7.2 of \cite{yau2015foundation} is the closet analogue).
This is a straightforward but tedious proof, so we only sketch the idea here.
The connection between the properadic composition of \cref{Alternate biased prop} and Definition 11.7.2 of \cite{yau2015foundation} is given in the preamble of \cref{extended properadic as a composite}.
By the "connected axioms" of \cref{Alternate biased prop}, we mean the the Bow+Caterpillar axioms, and the two relations of the Lighthouse+Fireworks axioms that only use the properadic compositions.
Each of these axioms have their analogue in Definition 11.7.2 of \cite{yau2015foundation}.

\end{remark}

We now define augmented and non-unital versions of this definition, mirroring the corresponding section for wheeled props.

\begin{definition}\label{Trivial Prop}
Let $\mathcal{K}$ be the trivial prop in $Vect_{\mathbb{K}}$ whose only constituents are the vertical and horizontal units.
\end{definition}

\begin{definition}\label{Augmented Prop}
An \textbf{augmentation} of a prop $P$ in $dgVect_{\mathbb{K}}$ is a morphism (of props, see Corollary 11.32 of \cite{yau2015foundation}) $\epsilon:P\to \mathcal{K}$. Props with an augmentation are called \textbf{augmented props}.
\end{definition}

\begin{definition}\label{non-unital alternate prop}
A \textbf{non-unital alternate prop} consists of all the data of an alternate prop, excluding the units and the corresponding axioms.
\end{definition}

\begin{proposition}
Augmented alternate props and non-unital alternate props are isomorphic.
\end{proposition}
\begin{proof}
This is a clear analogue of the classical result for partial operads see Proposition 21 of \cite{markl2008operads}.
\end{proof}
The last definition and isomorphism have only been presented for alternate props, and not Definition 11.30 of \cite{yau2015foundation}. 
The non-unital version of Definition 11.30 of \cite{yau2015foundation}, is not isomorphic to an augmented prop.
This follows from the bow graph counter example given in the preamble \cref{non-unital standard prop counter example}.
More explicitly, we can observe that because the standard vertical composition must connect all inputs of one graph with all outputs of another graph, the corresponding graphs of the non-unital version of Definition 11.30 of \cite{yau2015foundation} will satisfy an additional convexity condition.
This being, if two vertices $u$ and $v$ are connected by an edge then there can by no other vertex $w$ on a path between them.
Thus, for the rest of the paper when we speak of a non-unital prop we mean a non-unital alternate prop.

\subsection{A Groupoid Coloured Quadratic Presentation} \label{biased presentation of the operad governing props}

We now produce a quadratic non-unital groupoid coloured operad 
$\mathbb{P} = F_{\Sigma}^{\SC}(E)/\langle R \rangle$, whose algebras are non-unital props (mirroring \cref{biased presentation of the operad governing wheeled props}). The underlying groupoid of our operad is $\SC:= \ptwoc$. The generators $E$ correspond to the compositions as follows, let $\ba,\dc,\fe \in \ptwoc$ then 
\begin{align*}
    E(\dc,\ba; \fe):=&\{(\hcomp,\binom{\sigma}{\tau})(-,-): \binom{\sigma(\ud,\ub)}{(\uc , \ua)\tau} = \fe\} \sqcup \\
    &\{\eproperadic{\binom{\uc'}{\ub'}}{\binom{\sigma}{\tau}}(-,-): \binom{\sigma(\ud,\ub\setminus \ub')}{(\ua,\uc\setminus \uc')\tau} = \fe\} \sqcup\\
    &\{\eproperadicop{\binom{\ua'}{\ud'}}{\binom{\sigma}{\tau}}(-,-): \binom{\sigma(\ub,\ud\setminus \ud')}{(\uc,\ua\setminus \ua')\tau}= \fe\}
\end{align*}
The necessary left, right and $\Sigma$ actions are then given by the biequivariance axioms of \cref{Alternate biased prop} (note the biequivariance axioms for $\hcomp$ are found in \cref{Alternate biased wheeled prop}). These groupoid actions once again provide a simple canonical form of every groupoid coloured tree monomial (where all permutations are pushed to the top) which enables us to encode the quadratic relations $R$ graphically in \cref{The Relations of the Operad Governing Props}. By construction, it follows that,
\begin{corollary}
Algebras over $\mathbb{P}$ are non-unital props.
\end{corollary}

\begin{table}[ht]
\centering
\begin{tabular}{|c|c|c|}
\hline
Names & Graphs & Relations\\
\hline
Caterpillar 
&
\noindent\parbox[c]{0.75cm}{
\begin{tikzpicture}[thick,
v/.style={circle, draw}] 
\node[v] (1) {$v_1$}; 
\node[v] (2) [below of=1] {$v_2$};
\node[v] (3) [below of=2] {$v_3$};
\draw [-{Implies},double] (2) to (1);
\draw [-{Implies},double] (3) to (2);
\end{tikzpicture}
}
&
$\eproperadicunspecified(\eproperadicunspecified(v_1,v_2),v_3) = \eproperadicunspecified(v_1,\eproperadicunspecified(v_2,v_3))$
\\
\hline
Lighthouse 
&
\noindent\parbox[c]{2.2cm}{
\begin{tikzpicture}[thick,
v/.style={circle, draw}] 
\node[v] (1) {$v_1$}; 
\node[v] (2) [below left of=1] {$v_2$};
\node[v] (3) [below right of=1] {$v_3$};
\draw [-{Implies},double] (2) to (1);
\draw [-{Implies},double] (3) to (1);
\end{tikzpicture}   
}
&
$\eproperadicunspecified(\eproperadicunspecified(v_1,v_2),v_3) = \eproperadicunspecified(\eproperadicunspecified(v_1,v_3),v_2)= \eproperadicunspecified(v_1,\hcomp(v_2,v_3))$
\\
\hline
Fireworks
&
\noindent\parbox[c]{2.2cm}{
\begin{tikzpicture}[thick,
v/.style={circle, draw}] 
\node[v] (3) {$v_3$};
\node[v] (1) [above right of=3] {$v_1$}; 
\node[v] (2) [above left of=3] {$v_2$};
\draw [-{Implies},double] (3) to (1);
\draw [-{Implies},double] (3) to (2);
\end{tikzpicture}    
}
&
$\eproperadicunspecified(\hcomp(v_1,v_2),v_3) = \eproperadicunspecified(v_1,\eproperadicunspecified(v_2,v_3)) =  \eproperadicunspecified(v_2,\eproperadicunspecified(v_1,v_3))$
\\
\hline
Bow 
&
\noindent\parbox[c]{1.5cm}{
\begin{tikzpicture}[thick,
v/.style={circle, draw}] 
\node[v] (1) {$v_1$}; 
\node[v] (2) [below right of=1] {$v_2$};
\node[v] (3) [below left of=2] {$v_3$};
\draw [-{Implies},double] (2) to (1);
\draw [-{Implies},double] (3) to (1);
\draw [-{Implies},double] (3) to (2);
\end{tikzpicture}    
}
&
$\eproperadicunspecified(\eproperadicunspecified(v_1,v_2),v_3) = \eproperadicunspecified(v_1,\eproperadicunspecified(v_2,v_3))$
\\
\hline
Third-wheel 
&
\noindent\parbox[c]{1.5cm}{
\begin{tikzpicture}[thick,
v/.style={circle, draw}] 
\node[v] (1) {$v_1$}; 
\node[v] (2) [below of=1] {$v_2$};
\node[v] (3) [right of=2] {$v_3$};
\draw [-{Implies},double] (2) to (1);
\end{tikzpicture}     
}
&
$\hcomp(\eproperadicunspecified(v_1,v_2),v_3) = \eproperadicunspecified(v_1,\hcomp(v_2,v_3))=\eproperadicunspecified(\hcomp(v_1,v_3),v_2)$
\\
\hline
Associativity $\hcomp$ 
&
\noindent\parbox[c]{2.35cm}{
\begin{tikzpicture}[thick,
v/.style={circle, draw},node distance=0.8cm and 0cm] 
\node[v] (1) {$v_1$}; 
\node[v] (2) [right of=1] {$v_2$};
\node[v] (3) [right of=2] {$v_3$};
\end{tikzpicture}
}
&
$\hcomp(\hcomp(v_1,v_2),v_3)= \hcomp(v_1,\hcomp(v_2,v_3))$
\\
\hline
\end{tabular}
\caption{The relations of the operad governing props, and dually, the non-unital and non-biequivariance relations of an alternate prop.}
\label{The Relations of the Operad Governing Props}
\end{table}

\subsection{The Shuffle Operad}\label{shuffle operad for props}

In this section, we explicitly describe $\mathbb{P}^f=(F_{\Sigma}^{\SC}(E)/R)^f \cong F_{sh}^{\SC}(E^f)/R^f$ (mirroring \cref{shuffle operad for wheeled props}). We first observe how $-^f$ acts on the binary generators and their orbits. The generator $\hcomp$ is handled as in \cref{shuffle operad for wheeled props}, and as the other $\Sigma_2$ actions map generators to generators, i.e. $\eproperadic{\binom{\uc'}{\ub'}}{\binom{\sigma}{\tau}}(\alpha,\beta)\mapsto \eproperadicop{\binom{\uc'}{\ub'}}{\binom{\sigma}{\tau}}(\beta,\alpha)$ and $\eproperadicop{\binom{\uc'}{\ub'}}{\binom{\sigma}{\tau}}(\alpha,\beta)\mapsto \eproperadic{\binom{\uc'}{\ub'}}{\binom{\sigma}{\tau}}(\beta,\alpha)$, we have no need to introduce further generators in our shuffle operad.
\\\\
We now observe how $-^f$ acts on the relations by orienting each element of the orbit, of each relation of $\mathbb{P}$, to use shuffle compositions. As each of our families of relations corresponds to a given graph, we take the orbit of the relations corresponding to each graph.
In the table below, the line $\gamma\cdot p, e_1\leftarrow e_2,e_3$ means: we take the relations corresponding to the graph $\gamma$ and act on them via $p$ (graphically this corresponds to switching the vertex labels of the graph) after we orient these relations as shuffle tree monomials we are left with $e_1=e_2=e_3$, and $e_1$ is the smallest shuffle tree monomial under an order introduced in the next section. 
However, as was the case in \cref{shuffle operad for wheeled props}, we find that many orbit elements admit the same orientation. We will use the notation $\gamma\cdot p_1, \gamma_\cdot p_2, e_1\leftarrow e_2,e_3$ to denote the orbit elements $\gamma\cdot p_1$ and $\gamma\cdot p_2$ being oriented into the same relations.
We now list all $38$ resulting directed relations of the shuffle operad.
\\
{\small
\begin{longtable}[H]{lccc}
Caterpillar:
&
{
\noindent\parbox[c]{1cm}{
\begin{tikzpicture}[thick,
v/.style={circle, draw}] 
\node[v] (1) {$v_1$}; 
\node[v] (2) [below of=1] {$v_2$};
\node[v] (3) [below of=2] {$v_3$};
\draw [-{Implies},double] (2) to (1);
\draw [-{Implies},double] (3) to (2);
\end{tikzpicture}
}
}
&
{
$\begin{aligned}
    C\cdot (123)\\
    C\cdot (132)\\
    C\cdot (213)\\
    C\cdot (231)\\
    C\cdot (312)\\
    C\cdot (321)
\end{aligned}$
}
&
{
$\begin{aligned}
    \eproperadicunspecified(\eproperadicunspecified(v_1,v_2),v_3) &\leftarrow \eproperadicunspecified(v_1,\eproperadicunspecified(v_2,v_3)) \\
    \eproperadicunspecified(\eproperadicunspecified(v_1,v_3),v_2) &\leftarrow \eproperadicunspecified(v_1,\eproperadicopunspecified(v_2,v_3)) \\
    \eproperadicunspecified(\eproperadicopunspecified(v_1,v_2),v_3) &\leftarrow \eproperadicopunspecified(\eproperadicunspecified(v_1,v_3),v_2)\\
    \eproperadicopunspecified(\eproperadicopunspecified(v_1,v_3),v_2) &\leftarrow \eproperadicopunspecified(v_1,\eproperadicunspecified(v_2,v_3)) \\
    \eproperadicopunspecified(\eproperadicunspecified(v_1,v_2),v_3) &\leftarrow \eproperadicunspecified(\eproperadicopunspecified(v_1,v_3),v_2)\\
    \eproperadicopunspecified(\eproperadicopunspecified(v_1,v_2),v_3) &\leftarrow  \eproperadicopunspecified(v_1,\eproperadicopunspecified(v_2,v_3))
\end{aligned}$
}
\\
\\
\hline
\\
Lighthouse:
&
{
\noindent\parbox[c]{2.5cm}{
\begin{tikzpicture}[thick,
v/.style={circle, draw}] 
\node[v] (1) {$v_1$}; 
\node[v] (2) [below left of=1] {$v_2$};
\node[v] (3) [below right of=1] {$v_3$};
\draw [-{Implies},double] (2) to (1);
\draw [-{Implies},double] (3) to (1);
\end{tikzpicture}
}
}
&
{
$\begin{aligned}
    L\cdot (123) , L\cdot (132)\\
    L\cdot (213) , L\cdot (231)\\
    L\cdot (312) , L\cdot (321)
\end{aligned}$
}
&
{
$\begin{aligned}
	\eproperadicunspecified(\eproperadicunspecified(v_1,v_2),v_3) &\leftarrow  \eproperadicunspecified(\eproperadicunspecified(v_1,v_3),v_2), \eproperadicunspecified(v_1,\hcomp(v_2,v_3))\\
    \eproperadicunspecified(\eproperadicopunspecified(v_1,v_2),v_3) &\leftarrow \eproperadicopunspecified(\hcomp(v_1,v_3),v_2), \eproperadicopunspecified(v_1,\eproperadicunspecified(v_2,v_3))\\
    \eproperadicopunspecified(\hcomp(v_1,v_2),v_3) &\leftarrow \eproperadicopunspecified(\eproperadicopunspecified(v_1,v_3),v_2), \eproperadicopunspecified(v_1,\eproperadicopunspecified(v_2,v_3))
\end{aligned}$
}
\\
\\ \hline
\\
Fireworks:
&
{
\noindent\parbox[c]{2.5cm}{
\begin{tikzpicture}[thick,
v/.style={circle, draw}] 
\node[v] (3) {$v_3$};
\node[v] (1) [above right of=3] {$v_1$}; 
\node[v] (2) [above left of=3] {$v_2$};
\draw [-{Implies},double] (3) to (1);
\draw [-{Implies},double] (3) to (2);
\end{tikzpicture}   
}
}
&
{
$\begin{aligned}
    F\cdot (123) , F\cdot (213)\\
    F\cdot (132) , F\cdot (312)\\
    F\cdot (231) , F\cdot (321)
\end{aligned}$
}
&
{
$\begin{aligned}
   \eproperadicunspecified(\hcomp(v_1,v_2),v_3) &\leftarrow \eproperadicopunspecified(\eproperadicunspecified(v_1,v_3),v_2), \eproperadicunspecified(v_1,\eproperadicunspecified(v_2,v_3))\\
   \eproperadicopunspecified(\eproperadicunspecified(v_1,v_2),v_3) &\leftarrow \eproperadicopunspecified(\hcomp(v_1,v_3),v_2), \eproperadicunspecified(v_1,\eproperadicopunspecified(v_2,v_3))\\
   \eproperadicopunspecified(\eproperadicopunspecified(v_1,v_2),v_3) &\leftarrow \eproperadicopunspecified(\eproperadicopunspecified(v_1,v_3),v_2), \eproperadicopunspecified(v_1,\hcomp(v_2,v_3))
\end{aligned}$
}
\\
\\ \hline
\\
Bow:
&
{
\noindent\parbox[c]{1.8cm}{
\begin{tikzpicture}[thick,
v/.style={circle, draw}] 
\node[v] (1) {$v_1$}; 
\node[v] (2) [below right of=1] {$v_2$};
\node[v] (3) [below left of=2] {$v_3$};
\draw [-{Implies},double] (2) to (1);
\draw [-{Implies},double] (3) to (1);
\draw [-{Implies},double] (3) to (2);
\end{tikzpicture}   
}
}
&
{
$\begin{aligned}
    B\cdot(123)\\
    B\cdot(132)\\
    B\cdot(213)\\
    B\cdot(231)\\
    B\cdot(312)\\
    B\cdot(321)
\end{aligned}$
}
&
{
$\begin{aligned}
    \eproperadicunspecified(\eproperadicunspecified(v_1,v_2),v_3) &\leftarrow \eproperadicunspecified(v_1,\eproperadicunspecified(v_2,v_3)) \\
    \eproperadicunspecified(\eproperadicunspecified(v_1,v_3),v_2) &\leftarrow  \eproperadicunspecified(v_1,\eproperadicopunspecified(v_2,v_3))\\
    \eproperadicunspecified(\eproperadicopunspecified(v_1,v_2),v_3) &\leftarrow  \eproperadicopunspecified(\eproperadicunspecified(v_1,v_3),v_2)\\
    \eproperadicopunspecified(\eproperadicopunspecified(v_1,v_3),v_2) &\leftarrow\eproperadicopunspecified(v_1,\eproperadicunspecified(v_2,v_3)) \\
    \eproperadicopunspecified(\eproperadicunspecified(v_1,v_2),v_3) &\leftarrow \eproperadicunspecified(\eproperadicopunspecified(v_1,v_3),v_2)\\
    \eproperadicopunspecified(\eproperadicopunspecified(v_1,v_2),v_3)  &\leftarrow \eproperadicopunspecified(v_1,\eproperadicopunspecified(v_2,v_3))
\end{aligned}$
}
\\
\\ \hline
\\
Third-wheel:
&
{
\noindent\parbox[c]{2cm}{
\begin{tikzpicture}[thick,
v/.style={circle, draw}] 
\node[v] (1) {$v_1$}; 
\node[v] (2) [below of=1] {$v_2$};
\node[v] (3) [right of=2] {$v_3$};
\draw [-{Implies},double] (2) to (1);
\end{tikzpicture}  
}
}
&
{
$\begin{aligned}
    T\cdot(123)\\
    T\cdot(132)\\
    T\cdot(213)\\
    T\cdot(231)\\
    T\cdot(312)\\
    T\cdot(321)
\end{aligned}$
}
&
{
$\begin{aligned}
    \hcomp(\eproperadicunspecified(v_1,v_2),v_3) &\leftarrow \eproperadicunspecified(v_1,\hcomp(v_2,v_3)),\eproperadicunspecified(\hcomp(v_1,v_3),v_2)\\
    \eproperadicunspecified(\hcomp(v_1,v_2),v_3)&\leftarrow \hcomp(\eproperadicunspecified(v_1,v_3),v_2),\eproperadicunspecified(v_1,\hcomp(v_2,v_3)) \\
    \hcomp(\eproperadicopunspecified(v_1,v_2),v_3)&\leftarrow \eproperadicopunspecified(v_1,\hcomp(v_2,v_3)),\eproperadicopunspecified(\hcomp(v_1,v_3),v_2)\\ 
    \eproperadicunspecified(\hcomp(v_1,v_2),v_3) &\leftarrow \hcomp(v_1,\eproperadicunspecified(v_2,v_3)),\eproperadicopunspecified(\hcomp(v_1,v_3),v_2)\\
    \eproperadicopunspecified(\hcomp(v_1,v_2),v_3)&\leftarrow \hcomp(\eproperadicopunspecified(v_1,v_3),v_2),\eproperadicopunspecified(v_1,\hcomp(v_2,v_3)) \\
    \eproperadicopunspecified(\hcomp(v_1,v_2),v_3) &\leftarrow \hcomp(v_1,\eproperadicopunspecified(v_2,v_3)),\eproperadicunspecified(\hcomp(v_1,v_3),v_2)
\end{aligned}$
}
\\
\\ \hline
\\
Ass. of $\hcomp$:
&
{
\noindent\parbox[c]{2.5cm}{
\begin{tikzpicture}[thick,
v/.style={circle, draw},node distance=0.8cm and 0cm] 
\node[v] (1) {$v_1$}; 
\node[v] (2) [right of=1] {$v_2$};
\node[v] (3) [right of=2] {$v_3$};
\end{tikzpicture}  
}
}
&
$A\cdot(123),...,A\cdot(321)$
&
$\hcomp(\hcomp(v_1,v_2),v_3)\leftarrow \hcomp(\hcomp(v_1,v_3),v_2), \hcomp(v_1,\hcomp(v_2,v_3))$
\\
\\ \hline
\caption{The relations of the shuffle operad governing props}
\label{The relations of the shuffle operad governing props}
\end{longtable}
}

This completes the presentation of the relations of the shuffle operad $\mathbb{P}^f$, these relations encode all the equivalent ways to form directed wheel-free graphs with $3$ vertices via shuffle tree monomials.

\subsection{Ordering the Object Coloured Tree Monomials}\label{Ordering the Object Coloured Tree Monomials for Props}

We now define a total admissible order on the underlying $ob(\SC)$-coloured tree monomials of $F_{sh}(E)$. The motivation of this order as in \cref{Ordering the Object Coloured Tree Monomials for Wheeled Props} is to provide a simple unique minimal shuffle tree monomial encoding any wheel-free graph (\cref{UMF Algorithm for Props}). 

\begin{definition} \label{Total Admissible Order for Shuffle Tree Monomials of Props}
Let $\alpha,\beta$ be two tree monomials of $F_{sh}(E)$, we define $\alpha\leq \beta$ if:
\begin{enumerate}
    \item The arity of $\alpha < \beta$
    \item Or, if all the prior are equal, then compare $P^{\alpha}<P^{\beta}$ where
    \begin{itemize}
        \item if $\alpha$ and $\beta$ have arity $n$ then $P^{\alpha} := (P^{\alpha}_1,...,P^{\alpha}_n)$, where $P^{\alpha}_i$ is the word formed out of the generators when stepping from the $i$th input to the root in the tree monomial $\alpha$.
        \item $P^{\alpha}$ and $P^{\beta}$ are compared lexicographically, and two paths are compared only be degree (symbols ignored).
    \end{itemize}
    \item Or, if all the prior are equal, then compare the \textbf{input permutations} (not the permutations of the generators) of the shuffle tree monomials via lexicographic order.
    \item Or, if all the prior are equal, then compare the permutations of the generators with condition $4$ of \cref{Total Admissible Order for Shuffle Tree Monomials of Wheeled Props}.
    \item Or, if all the prior are equal, then we again compare $P^{\alpha}<P^{\beta}$ lexicographically, this time with a total order on generators.
    \begin{itemize}
        \item $\hcomp<\eproperadicunspecified<\eproperadicopunspecified$
        \item if the composition type is the same then the permutations are compared with, $\binom{\sigma}{\tau}\leq \binom{\sigma'}{\tau'}$ if $\sigma < \sigma'$, or $\sigma=\sigma'$ and $\tau<\tau'$
        \item if $s$ and $s'$ are identical generators symbol wise then first compare their input colours left to right, then their output colour. Input and output colours are profiles $\dc,\ba \in \ptwoc$, we say $\dc\leq \ba$ if $\ud< \ub$, or $\ud=\ub$ and $\uc\leq \ua$. We compare two sequences of colours using a degree lexicographic order induced by a total order on $\mathfrak{C}$.
    \end{itemize}
\end{enumerate}
\end{definition}

\begin{lemma}
The order of \cref{Total Admissible Order for Shuffle Tree Monomials of Props} is total and admissible.
\end{lemma}
\begin{proof}
A similar argument to \cref{The order for shuffle tree monomials of wheeled props is admissible}.
\end{proof}

We close this section by noting that we may use this total order on the underlying $ob(\SC)$-coloured tree monomials to order the relations of $\mathbb{P}^f$ by the minimal element in each groupoid coloured tree monomial, and this is precisely the order used in the prior section.

\subsection{The Operad Governing Props is Koszul} \label{proof The Operad Governing Props is Koszul}

This section proves that the groupoid coloured operad governing props $\mathbb{P}$ is Koszul, before outlining how this construction can be modified to yield Koszul operads governing other similar operadic families. We shall use the same proof method as used for $\mathbb{W}$ (in \cref{proof The Operad Governing Wheeled Props is Koszul}) with some minor combinatorial complications thrown in by the nature of props. 
In the prior section, we calculated an explicit presentation of the groupoid coloured shuffle operad $\mathbb{P}^f= F_{sh}^{\SC}(E) / \langle R^f \rangle$.
We now apply \cref{Koszul if corresponding shuffle operad has QGB}, to show that $\mathbb{P}$ is Koszul, if $G$ is a quadratic Groebner basis for the $ob(\SC)$-coloured shuffle operad $F_{sh}^{ob(\SC)}(E^{f_{\SC}}) / \langle G \rangle$, where $G:=E_*^2 \sqcup (R^f)_*$.

\begin{lemma} \label{Groebner Basis for Props}
Under the order of \cref{Total Admissible Order for Shuffle Tree Monomials of Props}, $G$ is a quadratic Groebner basis for $F_{sh}^{ob(\SC)}(E^{f_{\SC}}) / \langle G \rangle$.
\end{lemma}

\begin{proof}
We prove this lemma as follows.
\begin{itemize}
    \item In \cref{UMF Algorithm for Props}, we describe an algorithm that produces a unique minimal shuffle tree monomial encoding every wheel-free graph. We then show this algorithm is well-defined in \cref{UMF alg for props well-defined}.
    \item We then prove every shuffle tree monomial which is not the unique minimal shuffle tree monomial admits a rewrite using $RS(G)$ (see \cref{Rewriting Systems and Groebner Bases}), establishing the confluence of the rewriting system on all shuffle tree monomials. This will be accomplished by \cref{Non Minimal Prop Shuffle Tree Monomials Admit Rewrites}.
    \item Hence, through \cref{RS confluence iff GB}, this proves $G$ is a Groebner basis.
\end{itemize}
\end{proof}
The main complication in this proof compared to wheeled props is the more complicated unique minimal form, which we now detail, before giving an illustrative example (\cref{UMF Example}).

\begin{construction} \label{UMF Algorithm for Props}
The following algorithm produces the unique smallest shuffle tree monomial encoding a directed cycle free graph $\gamma$ with at least one vertex, we denote the outputted tree $UMF_\uparrow(\gamma)$.
\begin{itemize}
    \item Initiate $T=*_v,V = \{v\}$ and $V_R = Vertices(\gamma)\setminus \{v\}$ where $v$ is the smallest vertex in $\gamma$.
    \item While there is a vertex in $V_R$:
    \begin{itemize}
        \item Let $v$ be the smallest vertex in $V_R$ such that there does not exist a directed path (in either direction) in $\gamma$ between $v$ and a vertex in $V$ which passes through another vertex in $V_R$. This can be partitioned into three cases,
        \begin{enumerate}
            \item the vertex $v$ is disconnected from $V$, so update $T = (\hcomp,\binom{id}{id})(T,*_v)$
            \item the input flags of a vertex in $V$ are connected to the output flags of $v$, so $T = (\eproperadicunspecified,\binom{id}{id})(T,*_v)$
            \item the output flags of a vertex in $V$ are connected to the input flags of $v$, so $T = (\eproperadicopunspecified,\binom{id}{id})(T,*_v)$
        \end{enumerate}
        The necessary segments of $\eproperadicunspecified$ and $\eproperadicopunspecified$ may be identified from the graph $\gamma$ and the graph outputted by $T$.
        \item Update $V = V \cup \{v\}$ and $V_R = V_R\setminus \{v\}$.
    \end{itemize}
    \item Update the permutations of the root-most operation of $T$ (if $T$ is not a vertex) so that the corresponding graph outputted by $T$ has the same ordering on its open flags as $\gamma$.
    \item Return $T$.
\end{itemize}
\end{construction}

\begin{example} \label{UMF Example}
The algorithm applied to the following graph $\gamma$ yields,
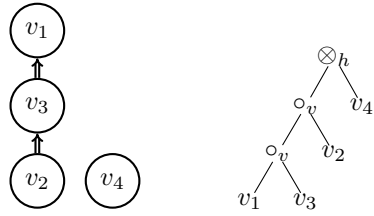
\begin{center}
\begin{tikzpicture}[thick,
v/.style={circle, draw}] 
\node[v] (1) {$v_1$}; 
\node[v] (2) [below of=1] {$v_3$};
\node[v] (3) [below of=2] {$v_2$};
\node[v] (4) [right of=3] {$v_4$};
\draw [-{Implies},double] (2) to (1);
\draw [-{Implies},double] (3) to (2);
\end{tikzpicture}
$\quad\quad\quad$
\begin{forest}
fairly nice empty nodes,
for tree={inner sep=0, l=0}
[{$\hcomp$}
    [{$\eproperadicunspecified$}
        [{$\eproperadicunspecified$}
            [{$v_1$}]
            [{$v_3$}] 
        ]
        [{$v_2$}]
    ]
    [{$v_4$}]
]
\end{forest}
\end{center}
We note that there does not exist a shuffle tree monomial forming $\gamma$ which contains $\hcomp(v_1,v_2)$ as a subtree (as we would need access to a composition $\circ'(\hcomp(v_1,v_2),v_3)$ which connects both inputs and outputs of both of its arguments). This illustrates why the restriction of the main loop of the algorithm is needed.
\end{example}

\begin{lemma} \label{UMF alg for props well-defined}
The unique minimal form algorithm of \cref{UMF Algorithm for Props} is well-defined.
\end{lemma}
\begin{proof}
It is straightforward to verify that this algorithm produces a shuffle tree monomial that produces the graph $\gamma$ (note that $T_*$ is given the smallest vertex of the graph at initialisation, so all compositions used in the construction of $T_*$ are valid shuffle operadic compositions). It remains to verify that the algorithm produces the minimal shuffle tree encoding $\gamma$.
\\\\
Firstly, we observe that the algorithm constructs a maximal length path from the minimal vertex to the root, i.e. our shuffle tree is 'left aligned'. So given our shuffle tree monomial has this shape, we observe that maximising the length of the path from the next smallest vertex (and so on) is equivalent to minimising the shuffle tree monomial permutation (ordering of the vertices of the graph). Our algorithm does this by greedily picking the smallest 'valid' vertex. If the algorithm at any stage picked a smaller vertex by ignoring the restriction, then it would be impossible to form the graph $\gamma$ from this subgraph $\gamma_*$, as illustrated in \cref{UMF Example}.
\\\\
Finally, given our algorithm produces a left aligned shuffle tree monomial with a minimal shuffle tree permutation, this information uniquely determines the contents of each of the generators. Indeed, every generator of the shuffle tree monomial must perform a specific horizontal composition or properadic join over particular segments, and any permutations can be passed up and down the tree via the action of the groupoid.
\end{proof}

\begin{lemma} \label{Non Minimal Prop Shuffle Tree Monomials Admit Rewrites}
Let $T$ be a shuffle tree monomial for a direct cycle free graph $\gamma$ such that $T\neq UMF_\uparrow(\gamma)$, then $T$ is rewritable by $G$.
\end{lemma}
\begin{proof}
Suppose that $T$ is larger than $UMF_{\uparrow}(\gamma)$ as a result of not being normalised with respect to the action of the groupoid. Then there exists an internal edge of $T$ which sits above a non-identity permutation. This internal edge defines a corresponding action of the groupoid which we can translate into an element of $E^2_*$. This defines a corresponding rewrite (which will push the permutation up the tree). So given we have access to this rewrite for the remainder of this proof, we suppose that $T$ is normalised with respect to the action of the groupoid. 
\\\\
Suppose that $T$ is larger than $UMF_{\uparrow}(\gamma)$ as a result of not having all generators on the path from the minimal vertex to the root (not being left aligned). Then $T$ must have the following form
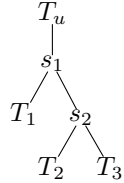
\begin{center}
\begin{forest}
fairly nice empty nodes,
for tree={inner sep=0, l=0}
[{$T_u$}
    [{$s_1$}
        [{$T_1$}]
        [{$s_2$}
            [{$T_2$}]
            [{$T_3$}] 
        ]
    ]
]
\end{forest}
\end{center}
where $T_1,T_2,T_3$ are subtrees of $T$, $T_u$ is the remainder of the tree $T$ (close to the root), $s_1,s_2$ are arbitrary generators, and due to $T$ being a shuffle tree monomial, the minimal vertices of $T_1<T_2<T_3$. Let $T' = s_1(T_1,s_2(T_2,T_3))$ then the corresponding shuffle tree monomial of $T'$ is $s_1(v_1,s_2(v_2,v_3))$. This shuffle tree monomial describes how to form a graph with $3$ vertices, and as $T$ is normalised $s_1(v_1,s_2(v_2,v_3))$ must be normalised. Observe that in the relations of $\mathbb{P}^f$ whenever we have a shuffle tree monomial of the form $s_1(v_1,s_2(v_2,v_3))$ forming a particular graph with $3$ vertices, that it appears as the leading term of a directed equation. So let $g$ be the corresponding directed equation with $lt(g)=s_1(v_1,s_2(v_2,v_3))$, then $T$ will admit a rewrite through the subtree $T'$ to $r_g(T)$.
\\\\
Suppose that $T$ has all generators on the path from the minimal vertex to the root, but has a different shuffle permutation to $UMF_\uparrow(\gamma)$. It must be the case that $T$ admits a decomposition of the form
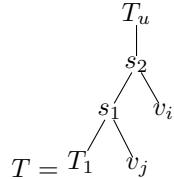
\begin{center}
$T=$
\begin{forest}
fairly nice empty nodes,
for tree={inner sep=0, l=0}
[{$T_u$}
    [{$s_2$}
        [{$s_1$}
            [{$T_1$}]
            [{$v_j$}] 
        ]
        [{$v_i$}]
    ]
]
\end{forest}
\end{center}
where $i<j$ and there does not exist a directed path from a vertex of $T_1$ to $v_i$ that passes through $v_j$, or a directed path from $v_i$ to a vertex of $T_1$ that passes through $v_j$ (if such a decomposition didn't exist then $T=UMF_\uparrow(\gamma)$). Here $s_1,s_2$ are arbitrary generators and $T_1$ and $T_u$ are the rest of the tree $T$. Let $T'= s_2(s_1(T_1,v_j),v_i)$ then the corresponding shuffle tree monomial of $T'$ is $s_2(s_1(v_1,v_3),v_2)$. This (normalised) shuffle tree monomial describes how to form a graph with three vertices subject to the condition above. The only graphs with 3 vertices that do not meet the required condition are
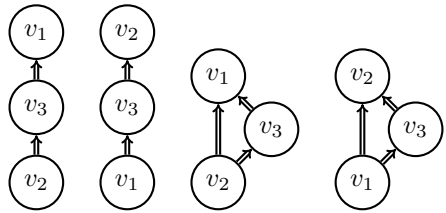
\begin{center}
\begin{tikzpicture}[thick,
v/.style={circle, draw}] 
\node[v] (1) {$v_1$}; 
\node[v] (2) [below of=1] {$v_3$};
\node[v] (3) [below of=2] {$v_2$};
\draw [-{Implies},double] (2) to (1);
\draw [-{Implies},double] (3) to (2);
\end{tikzpicture}
\quad
\begin{tikzpicture}[thick,
v/.style={circle, draw}] 
\node[v] (1) {$v_2$}; 
\node[v] (2) [below of=1] {$v_3$};
\node[v] (3) [below of=2] {$v_1$};
\draw [-{Implies},double] (2) to (1);
\draw [-{Implies},double] (3) to (2);
\end{tikzpicture}
\quad 
\begin{tikzpicture}[thick,
v/.style={circle, draw}] 
\node[v] (1) {$v_1$}; 
\node[v] (2) [below right of=1] {$v_3$};
\node[v] (3) [below left of=2] {$v_2$};
\draw [-{Implies},double] (2) to (1);
\draw [-{Implies},double] (3) to (1);
\draw [-{Implies},double] (3) to (2);
\end{tikzpicture}   
\quad 
\begin{tikzpicture}[thick,
v/.style={circle, draw}] 
\node[v] (1) {$v_2$}; 
\node[v] (2) [below right of=1] {$v_3$};
\node[v] (3) [below left of=2] {$v_1$};
\draw [-{Implies},double] (2) to (1);
\draw [-{Implies},double] (3) to (1);
\draw [-{Implies},double] (3) to (2);
\end{tikzpicture}   
\end{center}
which correspond to $C\cdot(132),C\cdot(231),B\cdot(132)$ and $B\cdot(231)$ in \cref{The relations of the shuffle operad governing props}. 
Manually inspecting the other relations of $\mathbb{P}^f$, we observe that for every other graph with $3$ vertices, that $s_2(s_1(v_1,v_3),v_2)$ only appears as the leading term of directed equations. So for the corresponding $g$ with leading term $s_2(s_1(v_1,v_3),v_2)$, the tree $T$ will admit a rewrite through the subtree $T'$ to $r_g(T)$.

\end{proof}

\subsection{Applying These Techniques to Other Operadic Structures}
\label{Koszul operads for other operadic}

The construction of this section can be easily specialised to provide Koszul operads governing di-ops (dioperadic composition + horizontal composition) or ops (operadic composition + horizontal composition). This follows from observing that dioperadic composition is an obvious restriction of properadic composition, and hence we can restrict the alternate biased definition of a prop \cref{Alternate biased prop} into alternate biased definitions of these structures (with the axioms restricted in the obvious ways). These new biased definitions then induce quadratic presentations of their operads. The restriction of the total order of \cref{Total Admissible Order for Shuffle Tree Monomials of Props} then provides the exact same unique minimal shuffle tree monomials governing these (restricted) directed cycle free graphs as that provided by \cref{UMF Algorithm for Props}. 
Then as non-minimal trees admit rewrites (\cref{Non Minimal Prop Shuffle Tree Monomials Admit Rewrites}) this proves the restricted operads are Koszul.
\\\\
A less straightforward specialisation is the construction of a Koszul operad governing properads, here is a brief outline. The alternate biased definition of a prop with its horizontal composition removed provides an alternate definition of a properad and hence a quadratic presentation of the operad governing properads (\cref{restriction is an alternate definition of a properad}). The corresponding shuffle operad is an obvious restriction of the shuffle operad for props, with no third wheel, no associativity of $\hcomp$ and no ways to form fireworks/lighthouse with $\hcomp$. If we apply the restriction of the total order of \cref{Total Admissible Order for Shuffle Tree Monomials of Props}, then we obtain a new unique minimal (properadic) shuffle tree monomial forming any connected directed cycle free graph. Then, proving that any non-minimal shuffle tree monomial admits a rewrite reproves (see \cref{introdution}) that the operad governing properads is Koszul. Dioperads and operads can then be obtained as a further specialisation of this construction.
\\\\
We note that there is no easy way to specialise the construction of a Koszul operad governing wheeled props into Koszul operads governing connected wheeled operadic structures, e.g. wheeled properads. This is because wheeled props form dioperadic joins by horizontal composition then contractions. The overall methodology of this paper however can still be applied starting from an alternate biased definition of a wheeled properad that uses extended dioperadic compositions and extended contractions.
We don't repeat this construction, as it is a known result
(see \cref{introdution}). 
Other potential targets for the methods of this paper include disconnected/Schwarz modular operads \cite{kaufmann2024schwarz}, and multi-oriented props \cite{merkulov2020multi}.
We note that \cite{kaufmann2024schwarz} subsequently developed distributive rewriting techniques for groupoid coloured operads to prove the operad governing disconnected modular operads is Koszul.
This distributive specialisation can also be used to study the operad governing wheeled props, but not the operad governing props, as $\eproperadicunspecified$ and $\hcomp$ are non-distributive (see for instance the lighthouse relation of \cref{The Relations of the Operad Governing Props}).

\section{Homotopy (Wheeled) Props}
\label{homotopy (wheeled) props}

Given the operads $\mathbb{W}$ and $\mathbb{P}$ are Koszul (\cref{def:koszul}), we immediately obtain the following corollary.
\begin{corollary}
\quad
\begin{enumerate}
    \item $\mathbb{W}_\infty:= \Omega(\mathbb{W}^{\antishriek})$ is a (quadratic) minimal model for $\mathbb{W}$.
    \item $\mathbb{P}_\infty:=\Omega(\mathbb{P}^{\antishriek})$ is a (quadratic) minimal model for $\mathbb{P}$.
\end{enumerate}
\end{corollary}
As such, we define a \textbf{homotopy (wheeled) prop} to be a algebra over $\mathbb{P}_\infty$ (respectively, $\mathbb{W}_\infty$).
This includes non-trivial examples such as,
\begin{itemize}
    \item The homology of (wheeled) props, as shown in \cref{homotopy transfer theory}.
    \item The weakly vertically associative (wheeled) props of \cite{wahl2016hochschild}.
    \item The completions of virtual and welded tangles are examples of homotopy wheeled props (\cite{dancso2021circuit}, \cite{dancso2023topological}). These structures permit a topological characterisation of the Kashiwara-Vergne groups.
\end{itemize}
In this section, we unpack what it means to be a homotopy (wheeled) prop.
We also show in \cref{Nesting Models and Polytope Based Techniuqes}, that we cannot use polytopes to describe these homotopy (wheeled) props, in a clear departure from the theory of connected operadic structures.
Finally, we close this paper by exploring consequences of homotopy transfer theory on these structures (\cref{homotopy transfer theory}).
\\\\
We first seek to understand the minimal models. The cobar construction (\cref{Cobar Construction}) yields,
\begin{align*}
    \Omega(\mathbb{W}^{\antishriek}) = (F(s^{-1}\mathbb{W}^{\antishriek}), d),\quad \Omega(\mathbb{P}^{\antishriek}) = (F(s^{-1}\mathbb{P}^{\antishriek}), d) 
\end{align*}
where the differentials are given by the cooperadic expansion of the Koszul dual cooperads $\mathbb{W}^{\antishriek}$ and $\mathbb{P}^{\antishriek}$.
We now unpack what this entails for $\mathbb{W}_\infty$, and similar reasoning works for $\mathbb{P}_\infty$.
The differential $d$ of $\Omega(\mathbb{W}^{\antishriek})$ will be induced by understanding how the cooperadic expansion into two pieces $\triangle_{(1)}:\mathbb{W}^{\antishriek} \to \mathbb{W}^{\antishriek}\otimes \mathbb{W}^{\antishriek}$, also known as the \textbf{infinitesimal decomposition}, acts on the basis elements of $\mathbb{W}^{\antishriek}$.
If $\mu \in \mathbb{W}^{\antishriek}$, then we will denote the image of $\mu$ under the infinitesimal decomposition map by $\triangle_{(1)}(\mu) = \sum (\mu_{(1)} \circ_i \mu_{(2)})\sigma $.
Recall from \cref{Koszul Dual Cooperad}, that the Koszul dual cooperad $\mathbb{W}^{\antishriek}$ is universal amongst the sub-cooperads $C$ of $F^c(sE)$ such that following composite is $0$.
\begin{center}
\begin{tikzcd}
C \arrow[r,hookrightarrow] & F^c(sE) = F(sE) \arrow[r,twoheadrightarrow] & F(sE)^{(2)}/s^2R
\end{tikzcd}
\end{center} 
That is to say, ${\mathbb{W}^{\antishriek}}^{(1)} = E$, ${\mathbb{W}^{\antishriek}}^{(2)} = R$ and higher degree elements of $\mathbb{W}^{\antishriek}$ are all in the ideal of $R$. 
A algebra over $\mathbb{W}_\infty$ is a morphism of groupoid coloured operads $\varphi:W_\infty \to End_A$. 
As such, a homotopy wheeled prop will be a family of graded vector spaces $A\dc_{\dc \in \ptwoc}$ together with operations,
\begin{align*}
    m_{\gamma}: A\binom{\ud_1}{\uc_1}\otimes...\otimes A\binom{\ud_k}{\uc_k} \to A\dc
\end{align*}
where $m_\gamma = \varphi(\mu_\gamma)$, and $\mu_\gamma$ is a basis element of  $\mathbb{W}^{\antishriek}$.
This tells us, each operation will be compatible with the groupoid $\ptwoc$, i.e. they will admit left, right and $\Sigma_{k}$ actions by translating across the actions of $\mathbb{W}$.
The decomposition $\triangle_{(1)}(\mu_\gamma) = \sum (\mu_{\gamma_1}\circ_i \mu_{\gamma_2})\sigma $, will induce relations of the form
\begin{align*}
    \partial_A(m_\gamma)=\sum \pm (m_{\gamma_1} \circ_{i} m_{\gamma_2}) \sigma
\end{align*}
In particular, we know that the following operations and relations will be present,
\begin{itemize}
    \item There will be differentials, i.e. degree $-1$ operations $d_1:A\dc \to A\dc$ such that $d_1^2 = 0$.
    \item There will be degree $0$ operations for each generator of the operad $\mathbb{W}$ that satisfy $\partial_A(m_\gamma)=0$.
    \item There will be operations of degree $1$, witnessing that each relation holds up to homotopy. Here is an example that illustrates a potential subtlety.
    For the operad $\mathbb{P}$, this particular lighthouse graph witnesses the following relations.
    \begin{center}
    \noindent\parbox[c]{2.2cm}{
    \begin{tikzpicture}[thick,
    v/.style={circle, draw}] 
    \node[v] (1) {$v_1$}; 
    \node[v] (2) [below left of=1] {$v_2$};
    \node[v] (3) [below right of=1] {$v_3$};
    \draw [thick] (2) to (1);
    \draw [thick] (3) to (1);
    \end{tikzpicture}   
    }    $\eproperadicunspecified(\eproperadicunspecified(v_1,v_2),v_3) = \eproperadicunspecified(\eproperadicunspecified(v_1,v_3),v_2)= \eproperadicunspecified(v_1,\hcomp(v_2,v_3))$
    \end{center}
    A basis for these relations are any two distinct pairs, for instance
    \begin{align}\label{eq:Lighthouse basis}
        \eproperadicunspecified(\eproperadicunspecified(v_1,v_2),v_3) \stackrel{r_1}{=} \eproperadicunspecified(\eproperadicunspecified(v_1,v_3),v_2), \quad \quad \eproperadicunspecified(\eproperadicunspecified(v_1,v_2),v_3) \stackrel{r_2}{=}\eproperadicunspecified(v_1,\hcomp(v_2,v_3))
    \end{align}
    Under this choice of basis, the relations yield operators $m_{r_1}, m_{r_2}:A\binom{\ud_1}{\uc_1}\otimes A\binom{\ud_2}{\uc_2}\otimes A\binom{\ud_3}{\uc_3} \to A\dc$. In this example, 
    \begin{align*}
        \partial_A(m_{r_2}) = m_{e_1}\circ_1 m_{e_2} - m_{e_3} \circ_2 m_{e_4}
    \end{align*}
    where $e_2$ is the generator that performs $\circ_v(v_1,v_2)$ in the context of $\gamma$, and $e_1$ is the generator that takes the output of $e_2$ together with $v_3$ and forms $\gamma$. 
    Note that individually each $m_{r_i}$ witness a homotopy between two terms, but together they witness homotopies between three terms.
    
\end{itemize}

We also stress that in these definitions of homotopy (wheeled) props, the action of the symmetric group remains strict. 
The effect this has on the definition is best understood through a simpler suboperad.
Let $\mathfrak{1}$ denote the discrete category with a single object.
Let $Com$ be the commutative operad (see for instance Section 13.1 of \cite{loday2012algebraic}), which can either be thought of as an uncoloured operad, or a $\mathfrak{1}$-coloured operad.
In addition, $Com$ is known to admit the quadratic presentation,
\begin{center}
$Com = F($ 
{\tiny
\begin{forest}
fairly nice empty nodes,
for tree={inner sep=0, l=0}
[
    [{$1$},]
    [{$2$},] 
]
\end{forest}
}
$=$
{\tiny
\begin{forest}
fairly nice empty nodes,
for tree={inner sep=0, l=0}
[
    [{$2$},]
    [{$1$},] 
]
\end{forest}
}
$) / \langle$
{\tiny
\begin{forest}
fairly nice empty nodes,
for tree={inner sep=0, l=0}
[,
    [
        [{$1$},]
        [{$2$},] 
    ]
    [{$3$},] 
]
\end{forest}
}
$-$
{\tiny
\begin{forest}
fairly nice empty nodes,
for tree={inner sep=0, l=0}
[,
    [{$1$},] 
    [
        [{$2$},]
        [{$3$},] 
    ]
]
\end{forest}
}
$\rangle$.
\end{center}
\begin{proposition}\label{inclusion of com}
The commutative operad is isomorphic to suboperads of both $\mathbb{W}$ and $\mathbb{P}$. 
\end{proposition}
\begin{proof}
Observe that $\mathbb{W}(\binom{\emptyset}{\emptyset},\binom{\emptyset}{\emptyset};\binom{\emptyset}{\emptyset})$ and $\mathbb{P}(\binom{\emptyset}{\emptyset},\binom{\emptyset}{\emptyset};\binom{\emptyset}{\emptyset})$ only contain a single basis element $(\hcomp, \binom{id_\emptyset}{id_{\emptyset}})$.
We can graphically interpret this basis element as the graph with two vertices, neither of which has any inputs or outputs, and consequently the graph has no inputs or outputs.
Furthermore, we can observe either directly from the biequivariance axiom of horizontal composition (\cref{Alternate biased wheeled prop}), or from the following graphical argument, that the $\Sigma_2$ action on $(\hcomp, \binom{id_\emptyset}{id_{\emptyset}})$ is trivial.
\begin{center}
$(\hcomp, \binom{id_\emptyset}{id_{\emptyset}}) \cdot (21)=$
\begin{tikzpicture}[thick,
v/.style={circle, draw}] 
\node[v] (1) {$v_1$}; 
\node[v] (2) [right of=1] {$v_2$};
\end{tikzpicture}
$ \cdot (21)=$
\begin{tikzpicture}[thick,
v/.style={circle, draw}] 
\node[v] (1) {$v_2$}; 
\node[v] (2) [right of=1] {$v_1$};
\end{tikzpicture}
$=$
\begin{tikzpicture}[thick,
v/.style={circle, draw}] 
\node[v] (1) {$v_1$}; 
\node[v] (2) [right of=1] {$v_2$};
\end{tikzpicture}
$=(\hcomp, \binom{id_\emptyset}{id_{\emptyset}}) $.
\end{center}
Equality two follows as the action of the symmetric group on the graph switches the label of the two vertices, and equality three follows from the definition of a graph (\cref{def:graph}), see also \cref{Example of Action of the Groupoid}. 
\\\\
Let $\mathbb{H}$ be the suboperad of both $\mathbb{W}$ and $\mathbb{P}$ generated by $(\hcomp, \binom{id_\emptyset}{id_{\emptyset}})$.
We observe that $\mathbb{H}$ will satisfy the horizontal associativity relation (Row 1 of \cref{The Relations of the Operad Governing Wheeled Props}, or equivalently the horizontal associativity axiom of \cref{Alternate biased wheeled prop}), and will satisfy no further relations, as all other relations of $\mathbb{W}$ and $\mathbb{P}$ involve either contractions or properadic compositions.
In addition, we observe that we no longer need to use the full groupoid $\SC:=\ptwoc$ to describe $\mathbb{H}$ but rather only $\emptyset^{op}\times \emptyset$ which is a single object discrete category, and hence isomorphic to $\mathfrak{1}$.
Thus, the generating groupoid coloured modules of $Com$ and $\mathbb{H}$ are isomorphic (\cref{def:morphism of groupoid coloured bimodules}), through the isomorphism of groupoids $f_0:\mathfrak{1}\to \emptyset^{op}\times \emptyset$, and the compatible bijective map of generators  $f_1({\tiny
\begin{forest}
fairly nice empty nodes,
for tree={inner sep=0, l=0}
[
    [{},]
    [{},] 
]
\end{forest}
}):=(\hcomp, \binom{id_\emptyset}{id_{\emptyset}})$.
Hence, $Com$ and $\mathbb{H}$ are isomorphic as groupoid coloured operads.
\\\\
We note that the graphical interpretation of the sole basis element in each $\mathbb{H}(\binom{\emptyset}{\emptyset}^n;\binom{\emptyset}{\emptyset})$ is the graph with $n$ vertices, none of which have any inputs or outputs.
This basis element has a trivial $\Sigma_n$ action, as permuting the labels of these vertices will not produce a new graph (\cref{def:graph}).

\end{proof}

It is a classical result that the commutative operad is Koszul, and that its Koszul dual is Lie, i.e. $Com^\antishriek$ is the co-Lie cooperad.
When we construct the minimal model of $Com$ using the (uncoloured) Koszul machine $C_\infty:=\Omega(Com^\antishriek)$, the associativity relation is relaxed up to homotopy, and commutativity continues to hold strictly.
As such, a $C_\infty$ algebra can be regarded as an $A_\infty$ algebra with the additional property that the operations are trivial on skew-symmetric shuffle products (see Proposition 13.1.6 of \cite{loday2012algebraic}, and Example 3.134 of \cite{markl2002operads}). 
Explicitly, they satisfy relations such as:
\begin{align*}
    m_2(\alpha,\beta) - (-1)^{|\alpha||\beta|}m_2(\beta,\alpha) &= 0\\
    m_3(\alpha,\beta,\gamma) - (-1)^{|\alpha||\beta|}m_3(\beta,\alpha,\gamma)
    +(-1)^{|\alpha|(|\beta|+|\gamma|)}m_3(\beta,\gamma,\alpha)&=0\\
    m_3(\alpha,\beta,\gamma) - (-1)^{|\beta||\gamma|}m_3(\alpha,\gamma,\beta)
    +(-1)^{(|\alpha|+|\beta|)|\gamma|}m_3(\gamma,\alpha,\beta)&=0.
\end{align*}
Analogous relations hold in higher degree.
The inclusion of $Com$ into both $\mathbb{W}$ and $\mathbb{P}$ tells us these relations will hold (up to inclusion) in algebras over $\mathbb{W}_\infty$ and $\mathbb{P}_\infty$.
Consequently, these structure will satisfy appropriately generalised versions of the prior relations, once one accounts for the presence of the contraction operators for $\mathbb{W}$, and the properadic compositions for $\mathbb{P}$.
\\\\
A 'simple combinatorial model' for homotopy (wheeled) props is beyond the scope of this paper.
However, it is straightforward to perform low dimensional computations, as the following examples show.
These calculations will also act as counter examples in the next section.

\begin{example}\label{wheeled prop polytope counter example}
Recall from \cref{The Relations of the Operad Governing Wheeled Props}, that in the operad $\mathbb{W}$, the followings graphs encode relations in $R = {\mathbb{W}^{\antishriek}}^{(2)}$, and thus the corresponding infinitesimal decompositions.
\begin{table}[!htbp]
\centering
\begin{tabular}{ |c|c|c| } 
\hline
Graph & Encodes Relation $r$ & $\triangle_{(1)}(r)$\\
\hline
\noindent\parbox[c]{2cm}{
\vspace{-0.25cm}
\begin{tikzpicture}[node distance={10mm}, thick,
v/.style={circle, draw}] 
\node[v] (1) {$v_1$}; 
\node[v] (2) [right of=1] {$v_2$};
\draw [thick] (1) to [out=120,in=240,looseness =3] (1);
\end{tikzpicture}
\vspace{-0.75cm}
}
& 
\noindent\parbox[c]{1.2cm}{
\begin{forest}
fairly nice empty nodes,
for tree={inner sep=0, l=0}
[
[[{$v_1$}]]
[{$v_2$}] 
]
\end{forest}
}
$-$
\noindent\parbox[c]{1.2cm}{
\begin{forest}
fairly nice empty nodes,
for tree={inner sep=0, l=0}
[
[
    [{$v_1$}]
    [{$v_2$}] 
]
]
\end{forest}
}
$=0$
&
{\tiny
\begin{forest}
fairly nice empty nodes,
for tree={inner sep=0, l=0}
[
    [, ]
    [, ] 
]
\end{forest}
}
$\circ_1$
{\tiny
\begin{forest}
fairly nice empty nodes,
for tree={inner sep=0, l=0}
[[,]]
\end{forest}
}
$-$
{\tiny
\begin{forest}
fairly nice empty nodes,
for tree={inner sep=0, l=0}
[[,]]
\end{forest}
}
$\circ_1$
{\tiny
\begin{forest}
fairly nice empty nodes,
for tree={inner sep=0, l=0}
[
    [,]
    [,] 
]
\end{forest}
}
\\
\hline
\noindent\parbox[c]{1cm}{
\vspace{-0.3cm}
\begin{tikzpicture}[node distance={10mm}, thick,
v/.style={circle, draw}] 
\node[v] (1) {$v_1$};
\draw [thick,red] (1) to [out=120,in=240,looseness =3] (1);
\draw [thick,blue] (1) to [out=60,in=-60,looseness =3] (1);
\end{tikzpicture}
\vspace{-0.5cm}
}
& 
\noindent\parbox[c]{0.6cm}{
\begin{forest}
fairly nice empty nodes,
for tree={inner sep=0, l=0}
[ 
[ , edge={thick,blue}
[{$v_1$}, edge={thick,red}
]
]
]
\end{forest}
}
$-\ \ $
\noindent\parbox[c]{0.6cm}{
\begin{forest}
fairly nice empty nodes,
for tree={inner sep=0, l=0}
[ 
[ , edge={thick,red}
[{$v_1$}, edge={thick,blue}
]
]
]
\end{forest}
}
$=0$
&
{\tiny
\begin{forest}
fairly nice empty nodes,
for tree={inner sep=0, l=0}
[[,edge={thick,blue}]]
\end{forest}
}
$\circ_1$
{\tiny
\begin{forest}
fairly nice empty nodes,
for tree={inner sep=0, l=0}
[[,edge={thick,red}]]
\end{forest}
}
-
{\tiny
\begin{forest}
fairly nice empty nodes,
for tree={inner sep=0, l=0}
[[,edge={thick,red}]]
\end{forest}
}
$\circ_1$
{\tiny
\begin{forest}
fairly nice empty nodes,
for tree={inner sep=0, l=0}
[[,edge={thick,blue}]]
\end{forest}
}
\\
\hline
\end{tabular}    
\end{table}

\raisebox{1.7em}{We use the graph
$\gamma = $}
{\tiny
\begin{tikzpicture}[thick,
v/.style={circle, draw}] 
\node[v] (1) {$v_1$}; 
\node[v] (2) [right of=1] {$v_2$};
\draw (1) [red, thick, in =120, out=-120,looseness=3] to (1);
\draw (1) [blue, thick,in =-120, out=60] to  (2);
\end{tikzpicture}
}
\raisebox{1.7em}{,
to encode $\alpha-\beta \in \mathbb{W}^\antishriek$, where $\alpha,\beta$ are any two elements of
}
\begin{center}
{\tiny
\begin{forest}
fairly nice empty nodes,
for tree={inner sep=0, l=0}
[
[, edge={thick,blue}
    [[{$v_1$}, edge={thick,red}]]
    [{$v_2$}] 
]
]
\end{forest}
}
$=$
{\tiny
\begin{forest}
fairly nice empty nodes,
for tree={inner sep=0, l=0}
[
[, edge={thick,blue}
[, edge={thick,red}
    [{$v_1$}]
    [{$v_2$}] 
]
]
]
\end{forest}
}
$=$
{\tiny
\begin{forest}
fairly nice empty nodes,
for tree={inner sep=0, l=0}
[
[, edge={thick,red}
[, edge={thick,blue}
    [{$v_1$}]
    [{$v_2$}] 
]
]
]
\end{forest}
}
\end{center}

We calculate,
\begin{center}
\raisebox{1.7em}{$
\triangle_{(1)}(
$}
{\tiny
\begin{tikzpicture}[thick,
v/.style={circle, draw}] 
\node[v] (1) {$v_1$}; 
\node[v] (2) [right of=1] {$v_2$};
\draw (1) [red, thick, in =120, out=-120,looseness=3] to (1);
\draw (1) [blue, thick,in =-120, out=60] to  (2);
\end{tikzpicture}
}
\raisebox{1.7em}{$
)
= \pm
$}
{\tiny
\begin{tikzpicture}[thick,
v/.style={circle, draw}] 
\node[v] (1) {$v_1$}; 
\draw (1) [blue, thick, in =60, out=-60,looseness=3] to (1);
\end{tikzpicture}
}
\raisebox{1.7em}{$
\circ_1
$}
{\tiny
\begin{tikzpicture}[thick,
v/.style={circle, draw}] 
\node[v] (1) {$v_1$}; 
\node[v] (2) [right of=1] {$v_2$};
\coordinate[above right =0.3cm of 1] (3); 
\coordinate[below left =0.3cm of 2] (4); 
\draw (1) [red, thick, in =120, out=-120,looseness=3] to (1);
\draw (1) to (3);
\draw (2) to (4);
\end{tikzpicture}
}
\raisebox{1.7em}{$
\pm
$}
{\tiny
\begin{tikzpicture}[thick,
v/.style={circle, draw}] 
\node[v] (1) {$v_1$}; 
\draw (1) [red, thick, in =120, out=-120,looseness=3] to (1);
\draw (1) [blue, thick,in =-60, out=60,looseness=3] to  (1);
\end{tikzpicture}
}
\raisebox{1.7em}{$
\circ_1
$}
\raisebox{0.8em}{
{\tiny
\begin{tikzpicture}[thick,
v/.style={circle, draw}] 
\node[v] (1) {$v_1$}; 
\node[v] (2) [right of=1] {$v_2$};
\coordinate[above right =0.3cm of 1] (3); 
\coordinate[above left =0.3cm of 1] (4); 
\coordinate[below left =0.3cm of 1] (5);
\coordinate[below left =0.3cm of 2] (6);
\draw (1) to (3);
\draw (1) to (4);
\draw (1) to (5);
\draw (2) to (6);
\end{tikzpicture}
}
}
\end{center}
The first term arises from cutting the first two tree monomials under the blue unary contraction, the second term arises from cutting the last two terms under the contractions/(above the binary fork).
Encoding successive calculations via diagram, we compute the subcomplex generated by $\alpha-\beta$,
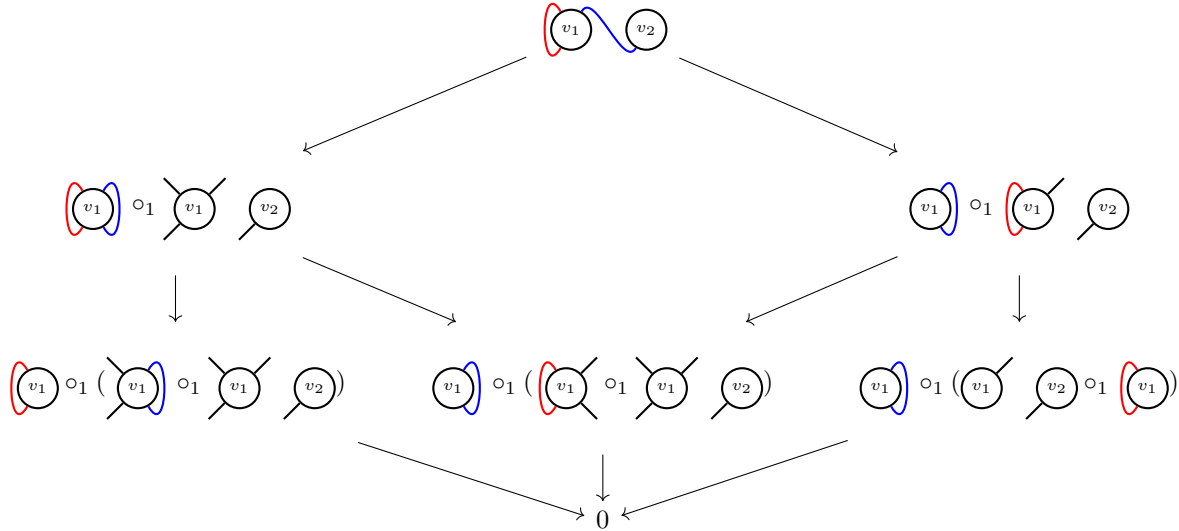
\begin{center}
\begin{tikzcd}
&
{\tiny
\begin{tikzpicture}[thick,
v/.style={circle, draw}] 
\node[v] (1) {$v_1$}; 
\node[v] (2) [right of=1] {$v_2$};
\draw (1) [red, thick, in =120, out=-120,looseness=3] to (1);
\draw (1) [blue, thick,in =-120, out=60] to  (2);
\end{tikzpicture}
}
\arrow[dl]
\arrow[dr]
&\\
{\tiny
\begin{tikzpicture}[thick,
v/.style={circle, draw}] 
\node[v] (1) {$v_1$}; 
\draw (1) [red, thick, in =120, out=-120,looseness=3] to (1);
\draw (1) [blue, thick,in =-60, out=60,looseness=3] to  (1);
\end{tikzpicture}
}
\circ_1
{\tiny
\begin{tikzpicture}[thick,
v/.style={circle, draw}] 
\node[v] (1) {$v_1$}; 
\node[v] (2) [right of=1] {$v_2$};
\coordinate[above right =0.3cm of 1] (3); 
\coordinate[above left =0.3cm of 1] (4); 
\coordinate[below left =0.3cm of 1] (5);
\coordinate[below left =0.3cm of 2] (6);
\draw (1) to (3);
\draw (1) to (4);
\draw (1) to (5);
\draw (2) to (6);
\end{tikzpicture}
}
\arrow[d]
\arrow[dr]
&
&
{\tiny
\begin{tikzpicture}[thick,
v/.style={circle, draw}] 
\node[v] (1) {$v_1$}; 
\draw (1) [blue, thick, in =60, out=-60,looseness=3] to (1);
\end{tikzpicture}
}
\circ_1
{\tiny
\begin{tikzpicture}[thick,
v/.style={circle, draw}] 
\node[v] (1) {$v_1$}; 
\node[v] (2) [right of=1] {$v_2$};
\coordinate[above right =0.3cm of 1] (3); 
\coordinate[below left =0.3cm of 2] (4); 
\draw (1) [red, thick, in =120, out=-120,looseness=3] to (1);
\draw (1) to (3);
\draw (2) to (4);
\end{tikzpicture}
}
\arrow[d]
\arrow[dl]
\\
{\tiny
\begin{tikzpicture}[thick,
v/.style={circle, draw}] 
\node[v] (1) {$v_1$}; 
\draw (1) [red, thick, in =120, out=-120,looseness=3] to (1);
\end{tikzpicture}
}
\circ_1
(
{\tiny
\begin{tikzpicture}[thick,
v/.style={circle, draw}] 
\node[v] (1) {$v_1$};
\coordinate[above left =0.3cm of 1] (3); 
\coordinate[below left =0.3cm of 1] (4);
\draw (1) to (3);
\draw (1) to (4);
\draw (1) [blue, thick,in =-60, out=60,looseness=3] to  (1);
\end{tikzpicture}
}
\circ_1
{\tiny
\begin{tikzpicture}[thick,
v/.style={circle, draw}] 
\node[v] (1) {$v_1$}; 
\node[v] (2) [right of=1] {$v_2$};
\coordinate[above right =0.3cm of 1] (3); 
\coordinate[above left =0.3cm of 1] (4); 
\coordinate[below left =0.3cm of 1] (5);
\coordinate[below left =0.3cm of 2] (6);
\draw (1) to (3);
\draw (1) to (4);
\draw (1) to (5);
\draw (2) to (6);
\end{tikzpicture}
}
)
\arrow[dr]
&

{\tiny
\begin{tikzpicture}[thick,
v/.style={circle, draw}] 
\node[v] (1) {$v_1$}; 
\draw (1) [blue, thick,in =-60, out=60,looseness=3] to  (1);
\end{tikzpicture}
}
\circ_1
(
{\tiny
\begin{tikzpicture}[thick,
v/.style={circle, draw}] 
\node[v] (1) {$v_1$};
\coordinate[above right =0.3cm of 1] (5);
\coordinate[below right =0.3cm of 1] (6);
\draw (1) to (5);
\draw (1) to (6);
\draw (1) [red, thick, in =120, out=-120,looseness=3] to (1);
\end{tikzpicture}
}
\circ_1
{\tiny
\begin{tikzpicture}[thick,
v/.style={circle, draw}] 
\node[v] (1) {$v_1$}; 
\node[v] (2) [right of=1] {$v_2$};
\coordinate[above right =0.3cm of 1] (3); 
\coordinate[above left =0.3cm of 1] (4); 
\coordinate[below left =0.3cm of 1] (5);
\coordinate[below left =0.3cm of 2] (6);
\draw (1) to (3);
\draw (1) to (4);
\draw (1) to (5);
\draw (2) to (6);
\end{tikzpicture}
}
)
\arrow[d]
&
{\tiny
\begin{tikzpicture}[thick,
v/.style={circle, draw}] 
\node[v] (1) {$v_1$}; 
\draw (1) [blue, thick,in =-60, out=60,looseness=3] to  (1);
\end{tikzpicture}
}
\circ_1
(
{\tiny
\begin{tikzpicture}[thick,
v/.style={circle, draw}] 
\node[v] (1) {$v_1$}; 
\node[v] (2) [right of=1] {$v_2$};
\coordinate[above right =0.3cm of 1] (3); 
\coordinate[below left =0.3cm of 2] (6);
\draw (1) to (3);
\draw (2) to (6);
\end{tikzpicture}
}
\circ_1
{\tiny
\begin{tikzpicture}[thick,
v/.style={circle, draw}] 
\node[v] (1) {$v_1$};
\draw (1) [red, thick, in =120, out=-120,looseness=3] to (1);
\end{tikzpicture}
}
)
\arrow[dl]
\\
&
0
&
\end{tikzcd}    
\end{center}
Where the image of the infinitesimal decomposition of a particular term (up to signs), is given by the target of all arrows whose source is that term.
Note the bracketing on the third line can be rearranged using the associativity of operadic composition.
We emphasise:
\begin{center}
\raisebox{1em}{The bargraph }
{\tiny
\begin{tikzpicture}[node distance={10mm}, thick,
v/.style={circle, draw}] 
\node[v] (1) {$v_1$}; 
\node[v] (2) [right of=1] {$v_2$};
\draw [thick,blue] (1) to [out=60,in=-120] (2);
\end{tikzpicture}
}
\raisebox{1em}{
never appears in the diagram, as it encodes no relation in $\mathbb{W}^\antishriek$.
}
\end{center}

\end{example}

\begin{example}\label{prop polytope counter example}
Let red and blue lighthouse graphs, provide a graphical encoding of our choice of basis for the lighthouse relations,

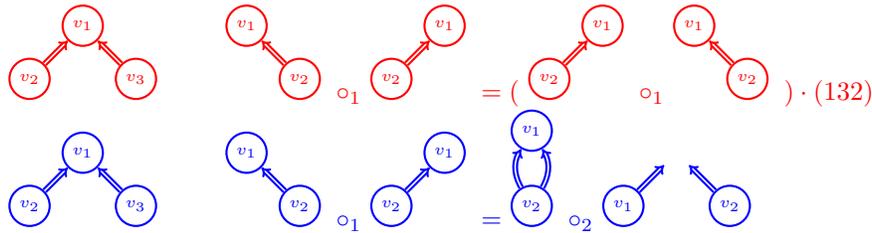
\begin{center}
\begin{tabular}{cl}
{\tiny
\begin{tikzpicture}[thick, red,
v/.style={circle, draw}] 
\node[v] (1) {$v_1$}; 
\node[v] (2) [below left of=1] {$v_2$};
\node[v] (3) [below right of=1] {$v_3$};
\draw [-{Implies},double,red] (2) to (1);
\draw [-{Implies},double, red] (3) to (1);
\end{tikzpicture}
}
\quad\quad
&
{\tiny
\begin{tikzpicture}[thick, red,
v/.style={circle, draw}] 
\node[v] (1) {$v_1$}; 
\node[v] (3) [below right of=1] {$v_2$};
\draw [-{Implies},double,red] (3) to (1);
\end{tikzpicture}
}
$\color{red}{\circ_1}$
{\tiny
\begin{tikzpicture}[thick, red,
v/.style={circle, draw}] 
\node[v] (1) {$v_1$}; 
\node[v] (2) [below left of=1] {$v_2$};
\draw [-{Implies},double,red] (2) to (1);
\end{tikzpicture}
}
$\color{red}{=(}$
{\tiny
\begin{tikzpicture}[thick, red,
v/.style={circle, draw}] 
\node[v] (1) {$v_1$}; 
\node[v] (2) [below left of=1] {$v_2$};
\draw [-{Implies},double,red] (2) to (1);
\end{tikzpicture}
}
$\color{red}{\circ_1}$
{\tiny
\begin{tikzpicture}[thick, red,
v/.style={circle, draw}] 
\node[v] (1) {$v_1$}; 
\node[v] (3) [below right of=1] {$v_2$};
\draw [-{Implies},double,red] (3) to (1);
\end{tikzpicture}
}
$\color{red}{)\cdot(132)}$\\
{\tiny
\begin{tikzpicture}[thick, blue,
v/.style={circle, draw}] 
\node[v] (1) {$v_1$}; 
\node[v] (2) [below left of=1] {$v_2$};
\node[v] (3) [below right of=1] {$v_3$};
\draw [-{Implies},double,blue] (2) to (1);
\draw [-{Implies},double, blue] (3) to (1);
\end{tikzpicture}
}
\quad\quad
&
{\tiny
\begin{tikzpicture}[thick, blue,
v/.style={circle, draw}] 
\node[v] (1) {$v_1$}; 
\node[v] (3) [below right of=1] {$v_2$};
\draw [-{Implies},double,blue] (3) to (1);
\end{tikzpicture}
}
$\color{blue}{\circ_1}$
{\tiny
\begin{tikzpicture}[thick, blue,
v/.style={circle, draw}] 
\node[v] (1) {$v_1$}; 
\node[v] (2) [below left of=1] {$v_2$};
\draw [-{Implies},double,blue] (2) to (1);
\end{tikzpicture}
}
$\color{blue}{=}$
{\tiny
\begin{tikzpicture}[thick, blue,
v/.style={circle, draw}] 
\node[v] (1) {$v_1$}; 
\node[v] (2) [below of=1] {$v_2$};
\draw [-{Implies},double,blue, bend left] (2) to (1);
\draw [-{Implies},double,blue, bend right] (2) to (1);
\end{tikzpicture}
}
$\color{blue}{\circ_2}$
{\tiny
\begin{tikzpicture}[thick, blue,
v/.style={circle, draw}] 
\node[thick, white,
v/.style={circle, draw}] (1) {$\phantom{v_1}$};
\node[v] (2) [below left of=1] {$v_1$};
\node[v] (3) [below right of=1] {$v_2$};
\draw [-{Implies},double,blue] (2) to (1);
\draw [-{Implies},double,blue] (3) to (1);
\end{tikzpicture}
}
\end{tabular}
\end{center}

In terms of \cref{eq:Lighthouse basis}, the red relation is $r_1$, and the blue relation is $r_2$.
Then, let $\alpha-\beta \in {\mathbb{P}^\antishriek}^{(3)}$, be any two terms of $\mathbb{P}$ which generate the following graph.
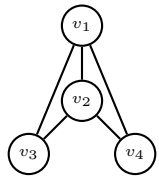
\begin{center}
{\tiny
\begin{tikzpicture}[thick, 
v/.style={circle, draw}] 
\node[v] (1) {$v_1$}; 
\node[v] (2) [below of=1]{$v_2$}; 
\node[v] (3) [below left of=2] {$v_3$};
\node[v] (4) [below right of=2] {$v_4$};
\draw (2) to (1);
\draw (3) to (1);
\draw (4) to (1);
\draw (3) to (2);
\draw (4) to (2);
\end{tikzpicture}
}
\end{center}

We partially compute a portion of the subcomplex generated by $\alpha-\beta$, using the notation of the prior example.
We have omitted many elements of the subcomplex for brevity.
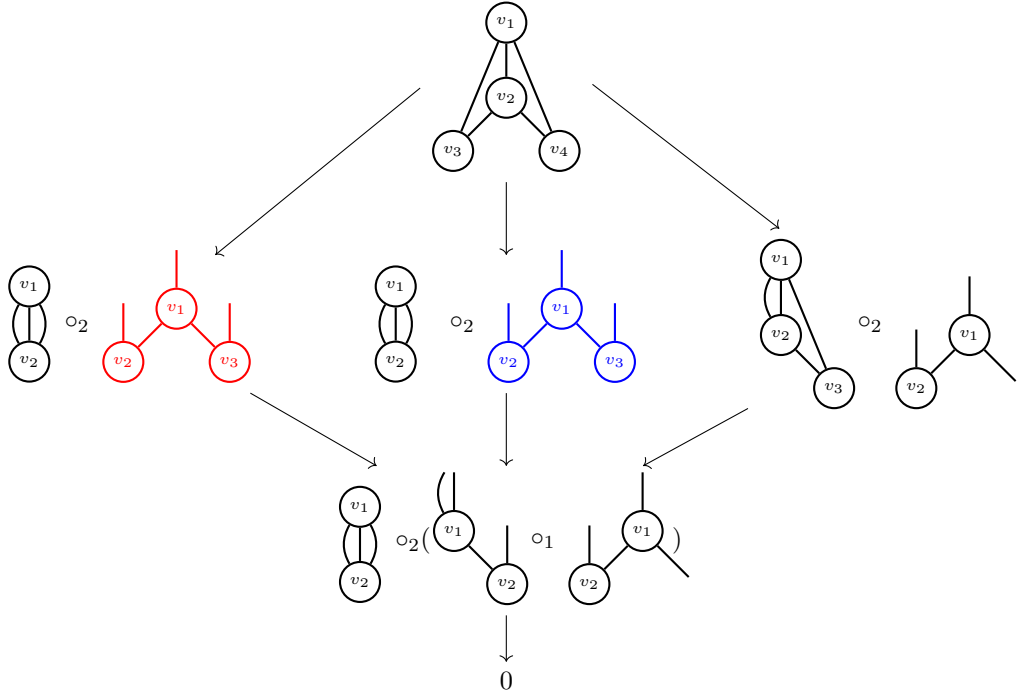
\begin{center}
\begin{tikzcd}
&
{\tiny
\begin{tikzpicture}[thick, 
v/.style={circle, draw}] 
\node[v] (1) {$v_1$}; 
\node[v] (2) [below of=1]{$v_2$}; 
\node[v] (3) [below left of=2] {$v_3$};
\node[v] (4) [below right of=2] {$v_4$};
\draw (2) to (1);
\draw (3) to (1);
\draw (4) to (1);
\draw (3) to (2);
\draw (4) to (2);
\end{tikzpicture}
}
\arrow[ld]
\arrow[d]
\arrow[rd]
\\
{\tiny
\noindent\parbox[c]{0.75cm}{
\begin{tikzpicture}[thick, v/.style={circle, draw}] 
\node[v] (1) {$v_1$}; 
\node[v] (2) [below of=1] {$v_2$};
\draw (2) to (1);
\draw[bend left] (2) to (1);
\draw[bend right]  (2) to (1);
\end{tikzpicture}
}
}
\noindent\parbox[c]{0.5cm}{
$\circ_2$
}
{\tiny
\noindent\parbox[c]{2cm}{
\vspace{-0.5cm}
\begin{tikzpicture}[thick, red, v/.style={circle, draw}] 
\node[v] (1) {$v_1$}; 
\node[v] (2) [below left of=1] {$v_2$};
\node[v] (3) [below right of=1] {$v_3$};
\node[thick, white, v/.style={circle, draw}] (4) [above of=2] {$\phantom{v_2}$};
\node[thick, white, v/.style={circle, draw}] (5) [above of=1]{$\phantom{v_2}$};
\node[thick, white, v/.style={circle, draw}] (6) [above of=3]{$\phantom{v_2}$};
\draw [red] (2) to (1);
\draw [red] (3) to (1);
\draw [red] (2) to (4);
\draw [red] (1) to (5);
\draw [red] (3) to (6);
\end{tikzpicture}
}
}
\arrow[rd]
&
{\tiny
\noindent\parbox[c]{1cm}{
\begin{tikzpicture}[thick, v/.style={circle, draw}] 
\node[v] (1) {$v_1$}; 
\node[v] (2) [below of=1] {$v_2$};
\draw (2) to (1);
\draw[bend left] (2) to (1);
\draw[bend right]  (2) to (1);
\end{tikzpicture}
}
}
\noindent\parbox[c]{0.5cm}{
$\circ_2$
}
{\tiny
\noindent\parbox[c]{2cm}{
\vspace{-0.5cm}
\begin{tikzpicture}[thick, blue, v/.style={circle, draw}] 
\node[v] (1) {$v_1$}; 
\node[v] (2) [below left of=1] {$v_2$};
\node[v] (3) [below right of=1] {$v_3$};
\node[thick, white, v/.style={circle, draw}] (4) [above of=2] {$\phantom{v_2}$};
\node[thick, white, v/.style={circle, draw}] (5) [above of=1]{$\phantom{v_2}$};
\node[thick, white, v/.style={circle, draw}] (6) [above of=3]{$\phantom{v_2}$};
\draw [blue] (2) to (1);
\draw [blue] (3) to (1);
\draw [blue] (2) to (4);
\draw [blue] (1) to (5);
\draw [blue] (3) to (6);
\end{tikzpicture}
}
}
\arrow[d]
&
{\tiny
\noindent\parbox[c]{1.3cm}{
\vspace{0.0cm}
\begin{tikzpicture}[thick, 
v/.style={circle, draw}] 
\node[v] (1) {$v_1$}; 
\node[v] (2) [below of=1]{$v_2$};
\node[v] (4) [below right of=2] {$v_3$};
\draw (2) to (1);
\draw [bend left] (2) to (1);
\draw (4) to (1);
\draw (4) to (2);
\end{tikzpicture}
}
}
\noindent\parbox[c]{0.5cm}{
$\circ_2$
}
{\tiny
\noindent\parbox[c]{2cm}{
\vspace{0.2cm}
\begin{tikzpicture}[thick, v/.style={circle, draw}] 
\node[v] (1) {$v_1$}; 
\node[v] (2) [below left of=1] {$v_2$};
\node[thick, white, v/.style={circle, draw}] (3) [below right of=1] {$\phantom{v_3}$};
\node[thick, white, v/.style={circle, draw}] (4) [above of=2] {$\phantom{v_2}$};
\node[thick, white, v/.style={circle, draw}] (5) [above of=1]{$\phantom{v_2}$};
\draw (2) to (1);
\draw (3) to (1);
\draw (2) to (4);
\draw (1) to (5);
\end{tikzpicture}
}
}
\arrow[ld]
\\
&
{\tiny
\noindent\parbox[c]{0.75cm}{
\vspace{0.1cm}
\begin{tikzpicture}[thick, v/.style={circle, draw}] 
\node[v] (1) {$v_1$}; 
\node[v] (2) [below of=1] {$v_2$};
\draw (2) to (1);
\draw[bend left] (2) to (1);
\draw[bend right]  (2) to (1);
\end{tikzpicture}
}
}
\noindent\parbox[c]{0.5cm}{
$\circ_2 (\ $
}
{\tiny
\noindent\parbox[c]{1.3cm}{
\vspace{-0.35cm}
\begin{tikzpicture}[thick, v/.style={circle, draw}] 
\node[v] (1) {$v_1$}; 
\node[v] (2) [below right of=1] {$v_2$};
\node[thick, white, v/.style={circle, draw}] (3) [above of=1] {$\phantom{v_2}$};
\node[thick, white, v/.style={circle, draw}] (4) [above of=2]{$\phantom{v_2}$};
\draw (2) to (1);
\draw [bend left] (1) to (3);
\draw (1) to (3);
\draw (2) to (4);
\end{tikzpicture}
}
}
\noindent\parbox[c]{0.5cm}{
$\circ_1$
}
{\tiny
\noindent\parbox[c]{1.3cm}{
\vspace{-0.35cm}
\begin{tikzpicture}[thick, v/.style={circle, draw}] 
\node[v] (1) {$v_1$}; 
\node[v] (2) [below left of=1] {$v_2$};
\node[thick, white, v/.style={circle, draw}] (3) [below right of=1] {$\phantom{v_3}$};
\node[thick, white, v/.style={circle, draw}] (4) [above of=2] {$\phantom{v_2}$};
\node[thick, white, v/.style={circle, draw}] (5) [above of=1]{$\phantom{v_2}$};
\draw (2) to (1);
\draw (3) to (1);
\draw (2) to (4);
\draw (1) to (5);
\end{tikzpicture}
}
}
\noindent\parbox[c]{0.1cm}{
$\ )$
}
\arrow[d]
\\
&
0
\end{tikzcd}
\end{center}

\end{example}

\subsection{Nesting Models and Polytope Based Techniques}\label{Nesting Models and Polytope Based Techniuqes}

The minimal models of the Koszul operads governing connected operadic structures all have polytope based interpretations (\cite{batanin2021koszul}, \cite{kaufmann2021koszul}, \cite{laplante2022diagonal}, \cite{ward2022massey}), however this no longer holds true for (wheeled) props.

\begin{theorem} \label{nesting complexes of props not polytopes}
There exist subcomplexes of $\mathbb{W}_\infty$ and $\mathbb{P}_\infty$ which are \textbf{not} isomorphic as lattices, to the face poset of convex polytopes.
\end{theorem}
\begin{proof}
A necessary condition for a lattice to be the face poset of a convex polytope is that it meets the diamond condition (Theorem 2.7 \cite{ziegler2012lectures}).
That is to say, every interval of length two has precisely four elements, and thus looks like a diamond.
Consequently, the diagrams of Examples \ref{wheeled prop polytope counter example} and \ref{prop polytope counter example}, cannot be isomorphic as lattices, to the face posets of convex polytopes.
This follows for $\mathbb{W}_\infty$, as
\cref{wheeled prop polytope counter example} has an interval of length two which contain only three elements.
For $\mathbb{P}_\infty$, we observe that \cref{prop polytope counter example} contains an interval of length $2$ which has at least five elements.
If another of the $\binom{3}{2,1}$ bases for the lighthouse relation is chosen (instead of \cref{eq:Lighthouse basis}), then a similar counter example will be obtained.
This is because every possible basis pair contains an overlapping term.
\end{proof}

As $\mathbb{W}_\infty$ and $\mathbb{P}_\infty$ are the minimal models for $\mathbb{W}$ and $\mathbb{P}$, these obstructions in homology will be present in all models for these operads.
Consequently, there does not exist a polytope based nesting model for the operads governing (wheeled) props (see Section 4 of \cite{laplante2022diagonal}, for what this means for the operad governing operads).

\begin{remark} \label{Nesting complex not convex obstructions}
The fact that these nesting complexes are not isomorphic to the face poset of convex polytopes can be interpreted as obstructions to using the techniques of \cite{batanin2021koszul} and \cite{kaufmann2021koszul} to prove that the operads governing (wheeled) props are Koszul.
These papers use polytopes to prove these results, explicitly though the hypergraph polytopes of \cite{batanin2023minimal} in the first case, and implicitly through cubical Feynman categories in the second case.
\\\\
Another way to see how this obstruction manifests, specifically for cubical Feynman categories, is as follows.
By Definition 7.4 of \cite{kaufmann2017feynman}, given the definition of $\mathbb{W}$, the bar graph should have degree $2$, and as such, should admit a free $\Sigma_2$ action on all the ways to form the graph.
However, there is only one way to form the bar graph, a horizontal composition followed by a contraction, so no such action can exist.
Similarly, given the definition of $\mathbb{P}$, the lighthouse graph should have degree $2$ for this operad, and have exactly two ways to form it, when we know there are three.
The lighthouse graph was previously identified as an impediment to the existence of a cubical Feynman category for props in Figure 10. of \cite{kaufmann2017feynman}.
\end{remark}

\begin{remark}
\label{acyclic tubings provide minimal model of operad governing properads}
Properads may be seen as props restricted to connected graphs.
As outlined in \cref{restriction is an alternate definition of a properad} and \cref{Koszul operads for other operadic}, the Koszul groupoid coloured quadratic operad governing properads $\mathbb{P}^c$ ($c$ for connected), consists of just the properadic generators of $\mathbb{P}$ and its associated connected axioms.
Given the work of \cite{batanin2021koszul},  \cite{batanin2023minimal} and \cite{kaufmann2021koszul}, it is known that $\mathbb{P}^c$ has a polytope based nesting model, as a result of more general theory.
However, it is possible to show that the specific polytope based nesting model for properads corresponds to the poset associahedra of Galashin \cite{galashin2021poset}.
This is a useful observation, as given the recent realisations of poset associahedra as convex polytopes, performed in \cite{chiara2023cyclonestohedra} and \cite{sack2023realization}, it is now possible to carry out the program of \cite{laplante2022diagonal}, at the level of properads.
That is to say, one can define a minimal model for the operad governing properads whose groupoid coloured module consists of connected wheel-free graphs with an acyclic tubing, and whose operadic composition is given by substitution of nested/tubed graphs, e.g.
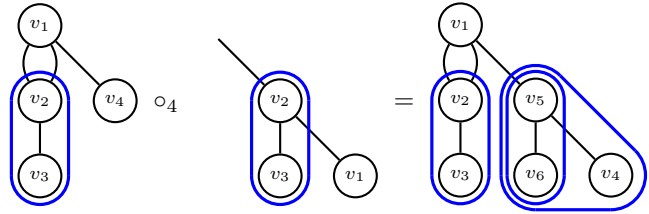
\begin{center}
{\scriptsize
\begin{tikzpicture}[thick,
v/.style={circle, draw}] 
\node[v] (1) {$v_1$}; 
\node[v] (2) [below of=1] {$v_2$};
\node[v] (3) [below of=2] {$v_3$};
\node[v] (5) [right of=2] {$v_4$};
\draw [bend left] (2) to (1);
\draw [bend right] (2) to (1);
\draw (3) to (2);
\draw (5) to (1);
\myroundpoly[blue, very thick]{2,3}{0.38cm}
\end{tikzpicture}
}
\raisebox{3.7em}{$\circ_4 $}
{\scriptsize
\begin{tikzpicture}[thick, v/.style={circle, draw}]
\node[thick, white, v/.style={circle, draw}] (1) {$v_1$}; 
\node[thick, white, v/.style={circle, draw}] (2) [below of=1] {$v_2$};
\node[thick, white, v/.style={circle, draw}] (3) [below of=2] {$v_3$};
\node[v] (5) [right of=2] {$v_2$}; 
\node[v] (6) [right of=3] {$v_3$};
\node[v] (4) [right of=6] {$v_1$};
\draw (4) to (5);
\draw (5) to (1);
\draw (5) to (6);
\myroundpoly[blue, very thick]{5,6}{0.38cm}
\end{tikzpicture}
}
\raisebox{3.7em}{$= $}
{\scriptsize
\raisebox{-0.25em}
{
\begin{tikzpicture}[thick,
v/.style={circle, draw}] 
\node[v] (1) {$v_1$}; 
\node[v] (2) [below of=1] {$v_2$};
\node[v] (3) [below of=2] {$v_3$};
\node[v] (5) [right of=2] {$v_5$}; 
\node[v] (6) [right of=3] {$v_6$};
\node[v] (4) [right of=6] {$v_4$};
\draw [bend left] (2) to (1);
\draw [bend right] (2) to (1);
\draw (3) to (2);
\draw (4) to (5);
\draw (5) to (1);
\draw (5) to (6);
\myroundpoly[blue, very thick]{2,3}{0.38cm}
\myroundpoly[blue, very thick]{5,6}{0.38cm}
\myroundpoly[blue, very thick]{5,4,6}{0.46cm}
\end{tikzpicture}
}
}
\end{center}
Then, by mirroring Section 4 of \cite{laplante2022diagonal}, it is possible to use this minimal model, and the realisation of the polytopes, to provide an explicit functorial tensor product of homotopy properads.
This is a subject for future research.
\end{remark}

\subsection{Homotopy Transfer Theory}\label{homotopy transfer theory}

One of the main motivations for constructing Koszul operads governing (wheeled) props is to extend homotopy transfer theory (HTT) to these structures.
We recall the necessary theory developed by Ward \cite{ward2022massey} in the groupoid coloured case, before deriving consequences in formality theory, and recovering a theorem of Mac Lane \cite{maclane1965categorical} as an instance of HTT.

\begin{definition}[Definition 2.56 \cite{ward2022massey}]
Let $A,B$ be $\mathbb{V}$-modules (in $dgVect$). We say $B$ is a \textbf{homotopy (deformation) retract} of $A$, if there are a family of homotopy (deformation) retracts indexed by $ob(\mathbb{V})$ such that $h_v,i_v,p_v$ are $Aut(v)$-equivariant.
\begin{center}
\begin{tikzcd}
A(v) 
\arrow[r,shift left=1ex,"{p_v}"]
\arrow[loop left,looseness=2,"{h_v}"]&
B(v)
\arrow[l,shift left=1ex,"{i_v}"] 
\end{tikzcd}
\end{center}
See for instance \cite{loday2012algebraic} Section 1.5.5, for the definition of a homotopy/deformation retract.
\end{definition}
As we are operating over a field of characteristic $0$, a fundamental example is

\begin{proposition}
[Lemma 2.57 \cite{ward2022massey}] \label{homology as deformation retract}
For a given $\mathbb{V}$-module $A$, the homology $H(A)$ is a deformation retract of $A$.
\end{proposition}

From here Ward specialises HTT to this particular deformation retract, but as he points out in his proof of Theorem 2.58, the only new requirement of his extension over the uncoloured case of \cite{loday2012algebraic} is the equivariance of the retract.

\begin{proposition} [Generalising Theorem 2.58 \cite{ward2022massey}, and Theorem 10.3.1 of \cite{loday2012algebraic}] Let $P$ be a Koszul groupoid coloured operad and $A,B$ be $\mathbb{V}$-modules such that $B$ is a homotopy retract of $A$. Then any $P_\infty$-algebra structure on $A$ can be transferred into a $P_\infty$-algebra structure on $B$ such that $i$ extends into a $\infty$-quasi-isomorphism.
\end{proposition}
\begin{proof}
Ward's specialised proof works with these slightly more general assumptions. The structure maps are then obvious alterations of those in the uncoloured case see Section 10.3 of \cite{loday2012algebraic}, i.e. the formulae work for groupoid coloured trees given the equivariance of the retract.
\end{proof}
Ward's specialisation to homology is then summarised in the following corollary.
\begin{corollary}[Theorem 2.58 \cite{ward2022massey}]
Let $P$ be a Koszul groupoid coloured operad, and $A$ a dg algebra over $P$. The homology $H(A)$ admits the structure of a $P_\infty$ algebra, and the inclusion $i:A\to H(A)$ extends to a $\infty$-quasi-isomorphism. 
\end{corollary}
The higher $P_\infty$ operations on the homology are known as \textbf{Massey products}. The existence of Massey products has applications in formality.
\begin{definition} [Section 11.46  \cite{loday2012algebraic}]
A dg-algebra $A$ over a Koszul groupoid coloured operad $P$ is said to be formal if there exists a $\infty$-quasi-isomorphism of dg-$P$-algebras between it and its homology $H(A)$.
\end{definition}

From this definition, it follows that a necessary condition for a dg-algebra to be formal, is for its Massey products to vanish uniformly (\cite{deligne1975real}). (If the Massey products do not vanish uniformly, then the best we could hope for would be a $\infty$-quasi-isomorphism of dg-$P_\infty$-algebras.) 
In addition, given our assumptions, we also have the following result.

\begin{proposition} [Proposition 11.4.10 \cite{loday2012algebraic}]
\label{Massey products characterise formality}
Let $P$ be a Koszul groupoid coloured operad and $A$ a dg-$P$-algebra. If the Massey products on $H(A)$ vanish, then $A$ is formal. 

\end{proposition}
\begin{proof}
Using the extensions of Ward, the proof of \cite{loday2012algebraic} applies verbatim to the groupoid coloured case.
\end{proof}

So, as a consequence of these results and the constructions of this paper, we have the following two results for dg (wheeled) props over a field of characteristic $0$.
\begin{corollary}
The homology of a (wheeled) prop admits Massey products.
\end{corollary}

\begin{corollary}
Thus, if $P$ is a $dg$ (wheeled) prop and $A$ a $dg$-$P$-algebra, then
\begin{itemize}
    \item if the Massey products on $H(A)$ vanish then $A$ is formal, and
    \item if $A$ is formal, then the Massey products on $H(A)$ vanish uniformly.
\end{itemize}

\end{corollary}

Thus we have a new tool in the study of the formality of (wheeled) props.
One wheeled prop, which is likely formal, is the category of arrow diagrams (\cite{dancso2023topological}).
However, in this case, it might be easier to demonstrate formality by constructing an explicit $GT$ action (see \cite{petersen2014minimal} and \cite{andersson2022deformation}).
Alternatively, it would be interesting if the non-formality of a particular (wheeled) prop could be demonstrated using these Massey products, perhaps using a similar approach to \cite{livernet2015non}.
\\\\
It also turns out that one of the first theorems regarding (homotopy) props can be seen as a consequence of HTT.
In \cite{maclane1965categorical}, Mac Lane defines both PROPs and PACTs (an early form of a homotopy prop).
He then expounds the existence of higher homotopies for the following particular prop.
Let $K(H_{fc})$ be the prop governing dg-Hopf algebras with commutative products.
Let $K(H_{fcc})$ be the prop governing dg-Hopf algebras with commutative products and co-commutative co-products.
\begin{theorem*}[25.1 of \cite{maclane1965categorical}]\label{Mac Lane theorem}
If $U$ is a $K(H_{fcc})$ algebra, then there is a PACT $P\supset K(H_{fc})$, which acts on the bar construction $B^\cdot(U)$ and on the reduced bar construction $\overline{B^\cdot}(U)$, and a map $\theta:P\to K(H_{fcc})$ of PACTs such that the induced homology map 
\begin{align*}
    \theta_*:H(P) \cong H(K(H_{fcc}))
\end{align*}
is an isomorphism. 
\end{theorem*}
He described this result as "a covert statement on the existence of higher homotopies", then used an action of $K(H_{fcc})$ on $U$ to induce an action on $\overline{B^{\cdot}}(U)$, which then induced an explicit example of a higher homotopy.
However, by the constructions of this section we know that the homology of $K(H_{fcc})$ is a homotopy prop, with higher homotopies corresponding to the Massey products. Furthermore, the higher homotopies identified by Mac Lane must at least be $\infty$-isomorphic to the Massey products (see Theorem 10.3.10 of \cite{loday2012algebraic}).

\begin{remark} \label{deformation theory link}
Having access to homotopy transfer theory for (wheeled) props provides an alternate pathway to the deformation theory of these structures (for instance, for props see \cite{vallette2009deformationI}, \cite{vallette2009deformationII}, and for wheeled props see \cite{andersson2022deformation}, \cite{markl2009wheeled}, \cite{merkulov2010wheeled}).
In particular, given that $\mathbb{W}$ and $\mathbb{P}$ are now known to be Koszul, the HTT techniques of Section 12.2 of \cite{loday2012algebraic} can also be applied. 
We note that this potential approach to defining homotopy props, and studying their deformations, was already suggested in a remark in Section 4.1 of \cite{vallette2009deformationI}.
We note that some care should be taken when comparing the results of this paper to theirs, as their graphical definition of a homotopy prop, introduced in Section 4.3, strictifies to a different notion of a prop than the one employed in this paper.
This can be seen as their notion of an admissible subgraph of a connected graph provides no means of forming a connected graph via disconnected subgraphs (i.e. through use of horizontal composition). 
\end{remark}

\section{Appendix: Axioms of an Alternate Prop}\label{axioms of the alternate prop}

In presenting \cref{Alternate biased prop}, we chose to encode the non-unital and non-biequivariance axioms of the definition graphically in \cref{The Relations of the Operad Governing Props}. We now formally write out these axioms for the interested reader. We will first introduce some simplifying terminology.
\\\\
Let $\uc',\uc''$ be two disjoint subsequences of $\uc$, which both admit the reduction of the order of $\uc$. Then there exists a unique way to merge them into a subsequence of $\uc$ (which also admits the reduction of the order of $\uc$), which we denote $m(\uc',\uc'')$. In other words, there exists a unique permutation $\sigma_m$ such that $m(\uc',\uc'') = (\uc',\uc'')\sigma_m$. 
Then, if in addition we have disjoint subsequences of $\ub$ say $\ub',\ub''$ such that $(\uc',\uc'') = (\ub',\ub'')$ then we also reorder these sequences by the \textbf{same} permutation $\sigma_m$, and in an abuse of notation, denote this reordering
\begin{align*}
    \binom{m(\uc',\uc'')}{m(\ub',\ub'')}:= \binom{(\uc',\uc'')\sigma_m}{(\ub',\ub'')\sigma_m}
\end{align*}
To ease readability, we make use of the following conventions. For each commutative diagram, we shall include the corresponding graph behind the axiom. These graphs each have $3$ vertices $v_1,v_2,v_3$ with respective profiles $\fe,\dc$ and $\ba$. We shall use primes to indicate dis-contiguous sub-profiles, if a profile $\ue$ has multiple sub-profiles then we will use $\ue'$ to refer to the sub-profile associated to the join to the smallest neighbouring vertex, and $\ue''$ for the other sub-profile.
Any equalities satisfied by these sub-profiles are informed by the diagrams. We will also suppose that all input sub-profiles have the same respective order as their enveloping profile. For instance, the caterpillar diagram has sub-profiles $\ub',\uc',\ud',\ue'$, they satisfy $\uc'=\ub'$, and $\ue'=\ud'$, and the orders on $\uc',\ue'$ are the restriction of the orders on $\uc,\ue$ respectively. Any permutations present on the commutative diagram are assumed to be such that the concrete assignments of the final profile $\binom{\underline{g}}{\underline{h}}$ (as defined by the arrows) are all equal. 
\begin{center}
\begin{longtable}{ccc}
Caterpillar:
&
\noindent\parbox[c]{1cm}{
\begin{tikzpicture}[thick,
v/.style={circle, draw}] 
\node[v] (1) {$v_1$}; 
\node[v] (2) [below of=1] {$v_2$};
\node[v] (3) [below of=2] {$v_3$};
\draw [-{Implies},double] (2) to (1);
\draw [-{Implies},double] (3) to (2);
\end{tikzpicture}
}
&
\begin{tikzcd} [column sep=0.75in]
P\fe \otimes P\dc \otimes P\ba \arrow[d,"{\eproperadic{\binom{\ue'}{\ud'}}{\binom{id}{id}}\otimes id}"]
\arrow[r,"{\eproperadic{\binom{\uc'}{\ub'}}{\binom{id}{id}}\otimes id}"]
&P \fe \otimes P\binom{\ud,\ub\setminus \ub'}{\ua,\uc\setminus \uc'}
\arrow[d,"{\eproperadic{\binom{\ue'}{\ud'}}{\binom{\sigma_1}{\tau_1}}\otimes id}"]
\\
P\binom{\uf, \ud\setminus \ud'}{\uc,\ue\setminus \ue'} \otimes P\ba 
\arrow[r,"{\eproperadic{\binom{\uc'}{\ub'}}{\binom{\sigma_2}{\tau_2}}}"]
&
P\binom{\underline{g}}{\underline{h}}
\end{tikzcd}
\\
Lighthouse:
&
\noindent\parbox[c]{2.5cm}{
\begin{tikzpicture}[thick,
v/.style={circle, draw}] 
\node[v] (1) {$v_1$};
\node[v] (2) [below left of=1] {$v_2$};
\node[v] (3) [below right of=1] {$v_3$};
\draw [-{Implies},double] (2) to (1);
\draw [-{Implies},double] (3) to (1);
\end{tikzpicture}   
}
&
\begin{tikzcd} [column sep=0.75in]
P\fe \otimes P\dc \otimes P\ba \arrow[d,"{\eproperadic{\binom{\ue'}{\ud'}}{\binom{id}{id}}\otimes id}"]
\arrow[r,"{id\otimes (\hcomp , \binom{id}{id})}"]
\arrow[dd, bend right=70,"switch"']
&P \fe \otimes P\binom{\ud,\ub}{\uc,\ua}
\arrow[d,"{\eproperadic{\binom{m(\ue',\ue'')}{m(\ud',\ub')}}{\binom{\sigma_1}{\tau_1}}}"]
\\
P\binom{\uf, \ud\setminus \ud'}{\uc,\ue\setminus \ue'} \otimes P\ba 
\arrow[r,"{\eproperadic{\binom{\ue''}{\ub'}}{\binom{\sigma_2}{\tau_2}}}"]
&
P\binom{\underline{g}}{\underline{h}}\\
P\fe \otimes P\ba \otimes P\dc
\arrow[r,"{\eproperadic{\binom{\ue''}{\ub'}}{\binom{id}{id}}}\otimes id"]
&
P\binom{\uf,\ub\setminus \ub'}{\ua,\ue\setminus \ue''} \otimes P\dc
\arrow[u,"{\eproperadic{\binom{\ue'}{\ud'}}{\binom{\sigma_3}{\tau_3}}}"']
\end{tikzcd}\\
Fireworks
&
\noindent\parbox[c]{2.5cm}{
\begin{tikzpicture}[thick,
v/.style={circle, draw}] 
\node[v] (3) {$v_3$};
\node[v] (1) [above right of=3] {$v_1$}; 
\node[v] (2) [above left of=3] {$v_2$};
\draw [-{Implies},double] (3) to (1);
\draw [-{Implies},double] (3) to (2);
\end{tikzpicture}    
}
&
\begin{tikzcd} [column sep=0.75in]
P\fe \otimes P\dc \otimes P\ba \arrow[d,"{id\otimes\eproperadic{\binom{\uc'}{\ub''}}{\binom{id}{id}}}"]
\arrow[r,"{(\hcomp , \binom{id}{id})\otimes id}"]
\arrow[dd, bend right=70,"switch"']
&P \binom{\uf,\ud}{\ue,\uc} \otimes P\binom{\ub}{\ua}
\arrow[d,"{\eproperadic{\binom{m(\ue',\uc')}{m(\ub',\ub'')}}{\binom{\sigma_1}{\tau_1}}}"]
\\
P\fe \otimes P\binom{\ud,\ub\setminus \ub''}{\ua,\uc\setminus \uc'}
\arrow[r,"{\eproperadic{\binom{\ue'}{\ub'}}{\binom{\sigma_2}{\tau_2}}}"]
&
P\binom{\underline{g}}{\underline{h}}\\
P\dc \otimes P\fe \otimes P\ba
\arrow[r,"{id\otimes\eproperadic{\binom{\ue'}{\ub'}}{\binom{id}{id}}}"]
&
P\dc \otimes P\binom{\uf,\ub\setminus \ub'}{\ua,\ue\setminus \ue'}
\arrow[u,"{\eproperadic{\binom{\uc'}{\ub''}}{\binom{\sigma_3}{\tau_3}}}"']
\end{tikzcd}\\
Bow 
&
\noindent\parbox[c]{2.5cm}{
\begin{tikzpicture}[thick,
v/.style={circle, draw}] 
\node[v] (1) {$v_1$}; 
\node[v] (2) [below right of=1] {$v_2$};
\node[v] (3) [below left of=2] {$v_3$};
\draw [-{Implies},double] (2) to (1);
\draw [-{Implies},double] (3) to (1);
\draw [-{Implies},double] (3) to (2);
\end{tikzpicture}    
}
&
\begin{tikzcd} [column sep=0.75in]
P\fe \otimes P\dc \otimes P\ba \arrow[d,"{\eproperadic{\binom{\ue'}{\ud'}}{\binom{id}{id}}\otimes id}"]
\arrow[r,"{\eproperadic{\binom{\uc'}{\ub''}}{\binom{id}{id}}\otimes id}"]
&P \fe \otimes P\binom{\ud,\ub\setminus \ub''}{\ua,\uc\setminus \uc'}
\arrow[d,"{\eproperadic{\binom{m(\ue',\ue'')}{m(\ud',\ub')}}{\binom{\sigma_1}{\tau_1}}}"]
\\
P\binom{\uf,\ud\setminus \ud'}{\uc,\ue\setminus \ue'} \otimes P\ba 
\arrow[r,"{\eproperadic{\binom{m(\ue'',\uc')}{m(\ub',\ub)}}{\binom{\sigma_2}{\tau_2}}}"]
&
P\binom{\underline{g}}{\underline{h}}
\end{tikzcd}\\
Third-wheel
&
\noindent\parbox[c]{2.5cm}{
\begin{tikzpicture}[thick,
v/.style={circle, draw}] 
\node[v] (1) {$v_1$}; 
\node[v] (2) [below of=1] {$v_2$};
\node[v] (3) [right of=2] {$v_3$};
\draw [-{Implies},double] (2) to (1);
\end{tikzpicture}     
}
&
\begin{tikzcd} [column sep=0.75in]
P\fe \otimes P\dc \otimes P\ba \arrow[d,"{\eproperadic{\binom{\ue'}{\ud'}}{\binom{id}{id}}\otimes id}"]
\arrow[r,"{id\otimes (\hcomp , \binom{id}{id})}"]
\arrow[dd, bend right=70,"switch"']
&P \fe \otimes P\binom{\ud,\ub}{\uc,\ua}
\arrow[d,"{\eproperadic{\binom{\ue'}{\ud'}}{\binom{\sigma_1}{\tau_1}}}"]
\\
P\binom{\uf,\ud\setminus \ud'}{\uc,\ue\setminus \ue'} \otimes P\ba 
\arrow[r,"{(\hcomp,{\binom{\sigma_2}{\tau_2}})}"]
&
P\binom{\underline{g}}{\underline{h}}\\
P\fe \otimes P\ba \otimes P\dc
\arrow[r,"{(\hcomp,{\binom{id}{id}})\otimes id}"]
&
P\binom{\uf,\ub}{\ue,\ua} \otimes P\dc
\arrow[u,"{\eproperadic{\binom{\ue'}{\ud'}}{\binom{\sigma_3}{\tau_3}}}"']
\end{tikzcd}
\end{longtable}
\end{center}

\bibliographystyle{alpha}
\addtocontents{toc}{\SkipTocEntry}
\bibliography{sources}

\end{document}